\documentclass[a4, 11pt.]{amsart}
\usepackage{amssymb,amscd, hyperref, bm, color }
\usepackage{amsmath}
\usepackage{enumerate}
\usepackage{stmaryrd}  
\usepackage[mathcal]{eucal}
\usepackage{array,float}
\usepackage{longtable}
\usepackage{multirow}
\usepackage{cite}
\usepackage[left=3cm, right=3cm]{geometry}

\usepackage{xy}
\input xy
\xyoption{all}
\usepackage{pdflscape} 
\usepackage{hyperref}
\hypersetup{
	colorlinks=blue,
	linkcolor=blue,
	filecolor=magenta,      
	urlcolor=cyan,
}

\usepackage[draft]{graphicx}

\numberwithin{equation}{section}
\newtheorem{thm}{Theorem}[section]
\newtheorem{proposition}[thm]{Proposition}

\newtheorem{corollary}[thm]{Corollary}
\newtheorem{lemma}[thm]{Lemma}

\newtheorem*{claim1}{Claim~1}
\newtheorem*{claim2}{Claim~2}
\newtheorem*{claim3}{Claim~3}

\theoremstyle{definition}
\newtheorem{definition}[thm]{Definition}
\newtheorem*{remark*}{Remark}
\newtheorem{remark}[thm]{Remark}

\newtheorem{example}[thm]{Example}

\newcommand{\RNum}[1]{\uppercase\expandafter{\romannumeral #1\relax}}

\newcommand{\SL}{\mathrm{SL}}
\newcommand{\GL}{\mathrm{GL}}

\newcommand{\cO}{\mathfrak{o}}
\newcommand{\mat}[4]{\left[ \begin{matrix}
#1 & #2 \\  #3 & #4 \end{matrix}\right] }

\newcommand{\val}{\mathrm{val}}

\renewcommand{\wp}{\mathfrak{p}}
\renewcommand{\det}{{\mathrm{det}}}

\newcommand{\Char}{\mathrm{Char}}
\newcommand{\til}[1]{\tilde{#1}}
\newcommand{\bol}[1]{\bold{#1}}
\newcommand{\G}{\mathrm{G}}

\renewcommand{\SS}{\mathrm{SS}}
\newcommand{\IR}{\mathrm{IR}}
\newcommand{\NS}{\mathrm{SNS}}
\newcommand{\pr}{\mathrm{pr}}
\newcommand{\calp}{\mathcal{P}}
\newcommand{\dvr}{\mathrm{DVR}}
\newcommand{\dvrtwo}{\dvr_2}
\newcommand{\dvrp}{\dvr_p}
\newcommand{\dvrtwoplus}{{\dvrtwo^+}}
\newcommand{\dvrtwozero}{{\dvrtwo^\circ}}
\newcommand{\ee}{{\mathrm{e}}}

\title[Representation Growth of $\SL_2(\cO)$]{Representation Growth of Compact Special \\ Linear Groups of degree two}

\author{M Hassain}
\address{ MH: Department of Mathematics,
  Indian Institute of Science,
   Bangalore 560012, India. }

\email{hassainm@iisc.ac.in}

\author{Pooja Singla}
\address{ PS:  Department of Mathematics and Statistics,
  Indian Institute of Technology Kanpur,
  Kanpur 208016, India. }
\email{psingla@iitk.ac.in}

\keywords{ Compact Special linear groups of degree two, Representation growth, group algebras, odd level representations of $\SL_2(\cO)$ }

\subjclass[2010]{20C15, 20G25, 20F69, 20H05, 20G05}

\begin{document}
\begin{abstract}

	We study the finite-dimensional continuous complex representations of $\SL_2$ over the ring of integers of non-Archimedean local fields of even residual characteristic. We prove that for characteristic two, the abscissa of convergence of the representation zeta function is 1, resolving the last remaining open case of this problem. We additionally prove that, contrary to the expectation, the group
	algebras of $\mathbb C[\SL_2(\mathbb Z/(2^{2 r} \mathbb Z))]$ and  $\mathbb C[\SL_2(\mathbb F_2[t]/(t^{2r}))]$ are not
	isomorphic for any $r > 1$. 
	This is the first known class of
	reductive groups over finite rings wherein the representation
	theory in the equal and mixed characteristic settings is genuinely different. From our methods, we explicitly obtain the primitive representation zeta polynomials of $\SL_2\left (\mathbb F_2[t]/(t^{2r}) \right) $ and $\SL_2\left (\mathbb Z/2^{2r} \mathbb Z \right) $ for $1 \leq r \leq 3$.

\end{abstract}
\maketitle

\tableofcontents

\section{Introduction} 
\label{sec:introduction} 
Let $F$ be a non-Archimedean local field with ring of integers $\cO,$ maximal ideal $\wp,$ and finite residue field $\mathbb F_q$ of characteristic $p.$ The finite quotient $\cO/\wp^\ell$ is denoted by $\cO_\ell.$  Let $\GL_n(\cO) $ be the group of $n \times n$ invertible matrices with entries from $\cO.$ Let $\SL_n(\cO)$ be the subgroup of $\GL_n(\cO)$ consisting of all determinant one matrices. The group $\GL_n(\cO)$ is a maximal compact subgroup of $\GL_n(F).$ It is well known that the continuous irreducible representations of $\GL_n(\cO)$ play an important role in explicitly constructing 
the supercuspidal representations of $\GL_n(F)$ (see Howe~\cite{MR0492087}, Kutzko~\cite{MR0507253}, and Paskunas~\cite{MR2180458}). In the recent past, there has been considerable work towards understanding the continuous irreducible representations of $\GL_n(\cO)$ and $\SL_n(\cO),$ see
 Chen-Stasinski~\cite{MR3703469}, Hill~\cite{MR2739795}, Krakovski-Onn-S.~\cite{MR1311772}, Stasinski~\cite{MR3737836}, Stasinski-Stevens~\cite{MR3743488}.

In this article, our focus is on the construction of the finite-dimensional continuous complex irreducible representations of the rigid groups $\SL_2(\cO)$ and to describe their representation growth. Recall that a group $G$ is called {\it rigid} if for every $n \in \mathbb N,$ the group $G$ has a finite number, say $r_n(G),$ of inequivalent irreducible representations of dimension $n.$  
For a rigid group $G$ and $s \in \mathbb C,$ the Dirichlet's series 
\[
\zeta_G(s) = \sum_{\rho \in \mathrm{Irr}(G)} \frac{1}{\dim(\rho)^s}
\]
is called the {\it representation zeta function of $G.$} The {\it abscissa of convergence} $\alpha(G)$ of $\zeta_G(s)$ is the infimum of all $\alpha \in \mathbb R$ such that $\zeta_G(s)$ converges
in the complex half-plane $\{s \mid \mathrm{Re}(s) > \alpha(G)    \}.$ For a rigid group $G$ and $n \in \mathbb N$ the function $n \mapsto r_n(G)$ is called the { \it representation growth function} of $G.$ The group $G$ is said to have {\it polynomial representation growth} (PRG) if the sequence $R_N(G) = \sum_{n =1}^N r_n(G)$ is bounded by a polynomial in $N.$ The
representation growth of a group with PRG can be studied using its representation zeta function. 
The abscissa of convergence $\alpha(G)$ of the series $\zeta_G(s)$ gives the precise degree of polynomial
growth in the sense that $\alpha(G)$ is the smallest value such that $R_N(G) = O(1 + N^{\alpha(G)+\epsilon})$ for every
$\epsilon \in \mathbb R^{>0}$ (see Voll~\cite[Section~3.2]{voll2009newcomers}).

The representation growth of arithmetic
groups and the associated abscissa of convergence was first studied by Larsen and Lubotzky in \cite{MR2390327}. The question of understanding representation growth of $p$-adic analytic groups occurs naturally in their work. A compact
$p$-adic analytic group $G$ is rigid if and only if it is {\it FAb}, that is, if every open
subgroup has finite abelianization (see Bass-Lubotzky-Magid-Mozes~\cite[Proposition 2]{MR1950883}). Making use of model theory,
Jaikin-Zapirain  \cite{MR2169043} proved that for $p > 2,$ the representation zeta function of a FAb
compact $p$-adic analytic group is a rational function in $p^{-s}.$  This means
that the coefficients of the Dirichlet generating function satisfy a linear recurrence relation. He also proved that  $\alpha(\SL_2(\cO)) = 1$ for $p \neq 2.$  It is also known that $\alpha(\SL_2(\cO)) = 1$ for $\Char(\cO) = 0$ and $p=2$, see Avni-Klopsch-Onn-Voll~\cite{MR2858928}. Therefore the remaining cases in this direction are of $\Char(\cO) =2$. It is already known that 
$\alpha(\SL_2(\cO)) \in [1, 22]$ for all rings $\cO,$ where the lower bound is by Larsen and Lubotzky~\cite[Proposition~6.6]{MR2390327} and the upper bound is by Aizenbud-Avni~\cite{MR3480557}. Recently, H\"{a}s\"{a}-Stasinski in \cite{2017arXiv171009112H}, while studying the twist representation zeta functions of $\GL_n(\cO)$ and relating these  with the representation zeta function of $\SL_n(\cO),$ improved the bounds for $\alpha(\SL_2(\cO))$ and proved that $\alpha(\SL_2(\cO)) \in [1, 5/2]$ for all $\cO.$ They also gave a new proof of the fact that $ \alpha(\SL_2(\cO)) = 1$ for $\Char(\cO ) = 0$. It is clear from their work that the case of $\SL_2(\cO)$ for $\Char(\cO)=2$ is more complicated than that of $\Char(\cO)=0$ or $\Char(\cO)=p\neq 2$. We complete the above results by determining $\alpha(\SL_2(\cO))$ for all $\cO$ such that $\Char(\cO)=2.$ More specifically, we prove the following result regarding $\alpha(\SL_2(\cO)).$  
\begin{thm}
	\label{thm:abs-of-convergence}
	Let $\cO$ be a compact discrete valuation ring with finite residue field such that $2 \mid |\cO/\wp|$. Then the abscissa of convergence of the representation zeta function of $\SL_2(\cO)$ is $1$ for all $\cO.$   
\end{thm}

For the proof, we first study the orbits and stabilizers of certain $\SL_2(\cO_r)$ action  in Sections~\ref{sec:Stabilizer-of-Sl2-action} and \ref{sec:orbits-stabilizers}. Then in
 Section~\ref{sec:proof-of-main-thms}, we study the numbers and dimensions of irreducible representations of $\SL_2(\cO_r)$. We use this to obtain a proof of the above result in Section~\ref{sec:abs-of-conv}.  We refer the reader to Avni~\cite{MR2831112}, Avni-Klopsch-Onn-Voll~\cite{MR3011874} and H\"{a}s\"{a}-Stasinski~\cite{2017arXiv171009112H} for more results on abscissa of convergence of compact linear groups. 

For a prime $p$, let $\dvrp$ denote the set of all compact discrete valuation rings with finite residue field  of characteristic $p.$ Let  $\dvrp^{\circ} = \{ \cO \in \dvrp \mid \Char(\cO) = 0   \}$ and  $\dvrp^{+} = \{ \cO \in \dvrp \mid \Char(\cO) = p   \}$. For $\cO \in \dvrp^{\circ}$, we say $\cO$ has ramification index $\ee$ if the ramification index of $F$ over $\mathbb Q_2$ is $\ee$, that is $p \cO = \pi^\ee \cO$. 
A construction of all irreducible representations of $\SL_2(\cO)$ for $\cO \in \dvr_p$ with $p \neq 2$ is already known, see Jaikin-Zapirain~\cite{MR2169043}. An important ingredient in this construction is the fact that the trace is a non-degenerate bilinear form on the set  $\mathrm{sl}_2(\mathbb F_q) = \{   X \in M_2(\mathbb F_q) \mid \mathrm{trace}(X) = 0 \}$ for odd $q$. But this does not hold for $p=2$ and therefore $p = 2$ requires different methods to deal with. A construction of irreducible representations of $\SL_2(\mathbb Z_2)$ has appeared in a series of articles by Nobs and Nobs-Wolfart~\cite{MR0444787,  MR0429742,  MR0387432, MR0444788}. The methods adopted in these articles were dependent on the structure of $\mathbb Z_2$  and therefore  can not be easily extended to other rings such as $\mathbb F_2 \llbracket t \rrbracket .$

We give a  
construction of ``primitive'' irreducible representations of  
$\SL_2( \cO/(\wp^{2 r}))$ for all  $\cO \in \dvrtwoplus$ with $r \geq 1$, and for $\cO \in \dvrtwozero$ with ramification index $\ee$ and $r \geq 2\ee$.  (see Theorem~\ref{main_theorem-2}).   This construction was also partially discussed in H\"{a}s\"{a}-Stasinski~\cite{2017arXiv171009112H}. The specific benefit we harness from our approach is to obtain information regarding the dimensions of irreducible representations. 
This helps us to prove the following result. 
\begin{thm}
	\label{thm:group-algebras}
	 Let $\cO \in \dvrtwozero$ with ramification index $\ee$, $\cO' \in \dvrtwoplus$ such that $\cO/\wp \cong \cO'/\wp',$ Then the group algebras $\mathbb C[\SL_2(\cO_{2r})]$ and  $\mathbb C[\SL_2( \cO'_{2r})]$  are not isomorphic for any $r > \ee.$ 
\end{thm}
 In particular: 
\begin{corollary}
The group algebras  $\mathbb C[\SL_2(\mathbb Z/(2^{2 r} \mathbb Z))]$ and  $\mathbb C[\SL_2(\mathbb F_2[t]/(t^{2r}))]$ are not isomorphic for any $r > 1.$  
\end{corollary}
See Theorem~\ref{thm:zeta poly not eaqal- gp algebra} for another formulation of Theorem~\ref{thm:group-algebras} and Section~\ref{sec:group-algebras} for its proof.  In Section~\ref{sec:Examples}, we prove that $\mathbb C[\SL_2(\mathbb Z/(2^{2}))] \cong \mathbb C[\SL_2(\mathbb F_2[t]/(t^{2}))]$ and therefore the bound $r> \ee$  is sharp for $\cO = \mathbb Z_2$ and $\cO' = \mathbb F_2 \llbracket t \rrbracket .$  These results are  surprising in view of the following well known isomorphisms of group algebras. 
\begin{enumerate}
\item  $\mathbb C[\GL_2(\mathbb Z/(p^{r} \mathbb Z))] \cong \mathbb C[\GL_2(\mathbb F_p[t]/(t^{r}))]$ for all prime $p$ and for all $r \geq 1.$
\item $\mathbb C[\SL_2(\mathbb Z/(p^{r} \mathbb Z))] \cong \mathbb C[\SL_2(\mathbb F_p[t]/(t^{r}))]$ for all prime $p$ such that $2 \nmid p$ and for all $r \geq 1.$ 
\end{enumerate}
 See Onn~\cite{MR2456275} and Stasinski-Stevens~\cite{MR3743488} for independent proofs of (1) and see Jaikin-Zapirain~\cite{MR2169043} and Krakovski-Onn-S.~\cite{MR3737836} for proof of (2). The isomorphisms of (1) and (2) hold more generally for compact rings $\cO$ and $\cO'$ such that their residue fields are isomorphic. It is interesting to recall the following related conjecture of Onn~\cite[Conjecture~1.2]{MR2456275}): 
 
 \medskip 
 {\it The group algebras $\mathbb C[\GL_n(\cO_r)]$ and $\mathbb C[\GL_n(\cO'_r)]$ are isomorphic
 for compact valuation rings $\cO$ and $\cO'$ such that their residue fields are isomorphic. }
 \medskip 
 
 This conjecture is proved to be true for $\GL_2(\cO_r)$ in Onn~\cite{MR2456275} for $r \geq 1$ and for $\GL_n(\cO_2)$ in ~\cite{MR2684153}. Theorem~\ref{thm:group-algebras} indicates that this conjecture when extended to $\SL_n$ instead of $\GL_n$ will require some modification with respect to the characteristic of the rings $\cO$ and $\cO'.$ 
We here briefly outline the idea of the proof of  Theorem~\ref{thm:group-algebras}. For a finite group $G,$  let $n^{max}(G)$ be the highest possible dimension of an irreducible representation of $G$ and $\#n^{max}(G)$ be the number of inequivalent irreducible representations of $G$ of dimension $n^{max}(G).$ We prove that for $\cO$ and $\cO'$ as given in Theorem~\ref{thm:group-algebras},  $n^{max}(\SL_2(\cO_{2r})) = n^{max}(\SL_2(\cO'_{2r}))$ and $\#n^{max}(\SL_2(\cO_{2r})) \neq \# n^{max}(\SL_2(\cO'_{2r})) .$

At last, in Section~\ref{sec:Examples}, we consider the examples of $\SL_2(\cO_i)$ for $\cO =  \mathbb Z_2$ or $\mathbb F_2 \llbracket t \rrbracket $ and $i = 2, 4$ and $6$ in detail.
 We describe the dimensions of all primitive, irreducible representations of these groups and obtain their primitive representation zeta polynomials, see Section~\ref{subsec:primitive-zeta-polynomials}. In particular, we obtain the following information about $n^{\mathrm{max}}$ and $\#n^{\mathrm{max}}$ in these cases. 
    \[
    	\begin{array}{|c|c|c|}
    	\hline
    		G & n^{max} (G) & \# {n^{\max}} (G)\\
    		\hline 
    		\SL_2(\cO_2) &  3 & \begin{cases}
    		  4&  \mathrm{for} \,\, \cO = \mathbb Z_2 \\
    		   4 & \mathrm{for} \,\, \cO = F_2 \llbracket t \rrbracket
    		\end{cases}     \\ 
    		\hline 
    		\SL_2(\cO_4) & 24  & \begin{cases}
    	2	& \mathrm{for} \,\, \cO = \mathbb Z_2 \\
    	1	&\mathrm{for} \,\, \cO = F_2 \llbracket t \rrbracket
    		\end{cases}     \\ 
    		\hline 
    		\SL_2(\cO_6) &  96 & \begin{cases}
    	8	& \mathrm{for} \,\, \cO = \mathbb Z_2 \\
    	6 	& \mathrm{for} \,\, \cO = F_2 \llbracket t \rrbracket 
    		\end{cases}     \\ 
    		\hline 
    	\end{array}     
    \]

    From above, it follows that  the group algebras $\mathbb{C}[\SL_2(\mathbb Z/ 2^{2r}\mathbb Z)]$  and $\mathbb C[\SL_2(\mathbb F_2[t]/(t^{2r}))]$ are not isomorphic for $1 < r \leq 3,$ see also Section~\ref{subsec:primitive-zeta-polynomials}. In view of Theorem~\ref{thm:group-algebras}, Theorem~\ref{thm:abs-of-convergence} indicates that although the structure of group algebras of the finite quotients $\SL_2(\cO_r)$ of $\SL_2(\cO)$ depend on the characteristic of the ring $\cO$ for $p = 2,$ the abscissa of convergence of $\SL_2(\cO)$ is uniform for all $\cO$ and $p.$ In the next section, we describe the basic framework including notations used in this article and also summarize the main ideas towards the proof of Theorem~\ref{thm:abs-of-convergence}.

\subsection*{Acknowledgement:} This work is supported in part by the UGC CAS-II grant. The authors would like to thank Alexander Stasinski and J. $\mathrm H\ddot{\mathrm a}\mathrm s\ddot{\mathrm a}$ immensely  for sharing their preprint of ~\cite{2017arXiv171009112H} and providing many helpful comments on this article. 
They thank Uri Onn and Dipendra Prasad for the encouragement and helpful comments. The authors heartily thank the anonymous referee for constructive feedback on this article, which greatly improved this article's presentation. The second named author acknowledges the SERB MATRICS grant's support (MTR/2018/000094). 

\section{Basic framework and outline of ideas} 
\label{sec:basic-framework} Recall that $F$ is a non-Archimedean local field with ring of integers $\cO,$ maximal ideal $\wp$ and finite residue field $\mathbb F_q$ of characteristic $p.$ The finite quotients $\cO/\wp^k$ are denoted by $\cO_k.$ For $\G = \GL_n$ or $\SL_n,$ 
it is well known that for every finite dimensional complex continuous irreducible representation $\rho$ of $\G(\cO),$ there exists a smallest natural number $r = r(\rho)$ such that $\rho$ factors through the principal congruence quotient $\G(\cO_{r+1}).$ In this case, we say $\rho$ is a representation of level $r.$ We will focus on constructing the level $r$ irreducible representations of groups $\SL_2(\cO)$ for every $r \geq 1.$ 

 For  $i \in \mathbb N,$  let  $M_n(\cO_{i})$  be the set of 
$n \times n$ matrices with entries from $\cO_i.$
 For any $i \leq r,$ there exists a natural surjective ring homomorphism $\cO_{r} \rightarrow \cO_i.$ By applying entry wise, this gives ring homomorphism  $f_{r,i}: M_n(\cO_{r})  \rightarrow M_n(\cO_i).$ The restriction of $f_{r,i}$ to $\GL_n(\cO_r),$ denoted by $\rho_{r,i},$ defines a surjective group homomorphism from $\GL_n(\cO_r)$ onto $\GL_n(\cO_i)$ and $\rho'_{r,i}: \SL_n(\cO_{r})  \rightarrow \SL_n(\cO_i)$ is the corresponding homomorphism obtained by restricting $f_{r,i}$ to $\SL_n(\cO_{r}) .$ For any $A \in M_n(\cO_{r}),$ the image $f_{r,1}(A)$ is denoted by $\bar{A}.$ Let $M^i =  \mathrm{ker} (\rho_{r,i})$ and $K^i = \mathrm{ker}(\rho'_{r,i}).$ Then it is clear that $K^i = M^i \cap \SL_n(\cO_r).$ The groups $M^i$ and $K^i$ are called the $i^{th}$ congruence subgroups of $\GL_n(\cO_r)$ and $\SL_n(\cO_r)$ respectively.
For $i \geq r/2,$ the group  $M^i $  is isomorphic to the abelian additive group  $M_n(\cO_{r-i})$ and  the subgroup $K^i$ of $M^i$  is isomorphic to the subgroup of $M_n(\cO_{r-i})$ consisting of all matrices with trace zero. 

An irreducible representation $\rho$ of $\G(\cO_r)$ is called a {\it twist} of another representation $\rho'$  if there exists a one-dimensional representation $\chi$ of $\G(\cO_r)$ such that $\rho' \cong \rho\otimes \chi.$ We say an irreducible representation $\rho$ of $\G(\cO_r)$ is {\it primitive} if neither $\rho$ nor any of its twists 
are trivial when restricted to $(r-1)^{th}$ congruence subgroup and $\rho$  is called {\it imprimitive} if it is not primitive. The set of imprimitive representations of $\G(\cO_r)$ can be constructed from the representations of  $\G(\cO_{r-1}).$ 
So to understand irreducible representations of $\G(\cO),$ it is enough to understand the primitive irreducible representations of each level. We use Clifford Theory (see Isaacs~\cite[Theorem~6.11]{MR2270898}) as our main tool for the construction of the primitive irreducible representations of $\SL_2(\cO_r)$.

 The group of one-dimensional representations of an abelian group $H$ is denoted by $\widehat{H}.$ 
Let $\pi$ be a fixed uniformizer of the ring $\cO.$ We fix an additive  one-dimensional representation $\psi: \cO_r  \rightarrow \mathbb C^\times$ such that $\psi(\pi^{r-1}) \neq 1.$ 
For $A \in M_n(\cO_i),$ let $\tilde{A} \in M_n(\cO_r)$ be an arbitrary lift of $A$ satisfying $f_{r,i}(\tilde{A}) = A.$ For any $i \leq r/2,$ define $\psi_A: I + \pi^{r-i} M_n(\cO_r) \rightarrow \mathbb C^\times$ by $\psi_A(I + \pi^{r-i} B) = \psi(\pi^{r-i}\mathrm{trace}(\tilde{A}B))$ for all $I + \pi^{r-i} B \in M^{r-i}.$  Then $\psi_A$ is a well defined one-dimensional representation of $M^{r-i}=I + \pi^{r-i} M_n(\cO_r).$
For $i \leq r/2,$ we obtain  
\begin{equation}
\label{eq:duality}
M_n(\cO_{i}) \cong \widehat {M^{r-i}}\,;\,\,\,\, A \mapsto \psi_A.
\end{equation}
We remark that for $\Char(\cO)=2,$ we shall  work with more specific $\psi,$ see section~\ref{sec:characterization-of-Ea-results} for its description.

For $r \geq 2,$ fix $\ell = \lceil r/2 \rceil$ and $\ell' = \lfloor r/2 \rfloor.$ Then (\ref{eq:duality}) gives an isomorphism of $M^\ell$ and $\widehat{M^\ell} \cong M_n(\cO_{\ell'}).$ The abelian normal subgroup $M^\ell$ of $\GL_n(\cO_r)$ and $\widehat{M^\ell}$ are known to play a very important role in explicitly constructing the irreducible representations of $\GL_n(\cO_r)$ via Clifford theory. For $\SL_n(\cO_r),$ the corresponding role is played by groups $K^\ell$ and  $\widehat{K^\ell}.$ We first describe $\widehat{K^\ell}$ below. 

Define an equivalence relation on $M_n(\cO_{\ell'})$ by $A \sim B$ if and only if $A = B +xI$ for some $x \in \cO_{\ell'}.$ The set of equivalence classes is denoted by $M_n(\cO_{\ell'})/\sim$ and the equivalence class of $A \in M_n(\cO_{\ell'})$ is denoted by $[A].$ Then  $\psi_A|_{K^\ell} = \psi_B|_{K^\ell}$ if and only if $A \sim B.$ Therefore $[A] \mapsto \psi_A|_{K^\ell}$ gives a bijection between $M_2(\cO_{\ell'})/\sim$ and $\widehat{K^\ell}.$ In view of this, we denote $\psi_A|_{K^\ell}$ by $\psi_{[A]}.$

Recall that a matrix $A \in M_n(\cO_r)$ is called {\it cyclic} if there exists a vector $v \in \cO_r^{\oplus n}$  such that $\{v, Av, A^{2}v, \ldots, A^{n-1} v\}$ generate $\cO_r^{\oplus n}$ as a free $\cO_r$-module. 
As an application of Nakayama's  lemma we see that a matrix $A \in M_n(\cO_\ell)$ is cyclic if and only if $\bar{A}$ is. Motivated by this, we say one-dimensional representations $\psi_A$ and $\psi_{[A]}$ of groups $M^\ell$ and $K^\ell$ respectively are cyclic if and only if $A$ is cyclic. Note that any $2\times 2$ matrix $A$ 
 is either cyclic or its projection $\bar{A}$ is scalar.
 We observe that an irreducible representation $\rho$ of $\GL_2(\cO_r)$ is primitive if and only if $\langle \rho|_{M^\ell}, \psi_{A} \rangle \neq 0 $ implies $A$ is a cyclic matrix. A similar criterion also holds for irreducible representations of  $\SL_2(\cO_r)$ as well.

  For a finite  group $G,$ let $\mathrm{Irr}(G)$ be the set of  inequivalent irreducible representations of $G.$ For $H\leq G$ and $\phi \in \mathrm{Irr}(H),$  let $\mathrm{Irr}(G \mid \phi)=\{\rho \in \mathrm{Irr}(G) \mid  \langle \rho|_{H}, \phi \rangle \neq 0  \}.$
For $H \unlhd G$ and $\phi \in \mathrm{Irr}(H),$  let $C_G(\phi) = \{ g \in G \mid \phi^g \cong \phi \}$ be the inertia group (also called stabilizer) of $\phi.$ For a group  $G$ and $g \in G,$ let $C_G(g) = \{h \in G \mid gh = hg   \}$ denote the centralizer of $g$ in $G.$ 
Recall that all representations that we deal with are always finite dimensional.

By Clifford theory,  the following bijections (via induction) hold for every $\psi_A \in \widehat{M^\ell}$ and $\psi_{[A]} \in \widehat{K^\ell}.$ 
\begin{equation}
\label{eq:gl-clifford}
\mathrm{Irr}( C_{\GL_n(\cO_{r})}(\psi_{A}) \mid \psi_A)  \leftrightarrow \mathrm{Irr}( \GL_n(\cO_r) \mid \psi_A),
\end{equation}
\begin{equation}
\label{eq:sl-clifford}
 \mathrm{Irr}(C_{\SL_n(\cO_{r})}(\psi_{[A]}) \mid \psi_{[A]}) \leftrightarrow \mathrm{Irr}(\SL_n(\cO_{r}) \mid \psi_{[A]}).  
\end{equation} 

Now onwards, we will concentrate on understanding the sets on left side of the above bijections. 
These are well understood
 for $\GL_2(\cO_{r})$
 for general $p$ and $\SL_2(\cO_r)$ for $p \neq 2$. However, we do not have a complete understanding of the left side sets of (\ref{eq:sl-clifford}) for $\SL_2(\cO_r)$ with  $p=2.$ 
As a first step towards understanding $\SL_2(\cO_r)$ representations, we consider the $\SL_2(\cO_r)$ action on $\widehat{K^\ell}$. Let $\Sigma^{\SS}$( $\Sigma^{\IR}$, $\Sigma^{\NS}$) be $\{ \psi_{[A]} \} \subseteq \widehat{K^\ell}$ such that   $\bar{A}$ is split semisimple (irreducible, Split non semisimple). The group $\SL_2(\cO_{r})$ acts on $\Sigma^{\SS},$ $\Sigma^{\IR}$, and $\Sigma^{\NS}$ via $\psi^g_{[A]} = \psi_{[A^g]}$. The orbits of this action are denoted by  $\Sigma^{\SS}_\cO$,  $\Sigma^{\IR}_\cO$, $\Sigma^{\NS}_{\cO}$ respectively. 
We study the stabilizer and orbits of this $\SL_2(\cO_r)$ action in Sections~\ref{sec:Stabilizer-of-Sl2-action} and \ref{sec:orbits-stabilizers}.  In particular, we prove the following results regarding the number of orbits. These results play an important role in proving Theorems~\ref{thm:abs-of-convergence} and \ref{thm:group-algebras}. We write $ A(q) \asymp B(q) $ to indicate that $A(q) = O(B(q))$ and $B(q) = O(A(q))$; that is, they have the same order of
magnitude.  Recall that $\dvrp$, for a prime $p$, denotes the set of all compact discrete valuation rings with finite residue field of characteristic $p.$ Now onwards,  we will also assume that $\cO/\wp \cong \mathbb F_q$ for $\cO \in \dvr_2$.

\begin{thm}\label{thm:number-of-orbits-split semisimple} 

Let $r \geq 2.$ 
\begin{enumerate}
		\item For $\cO \in  \dvrtwoplus,$  $|\Sigma^{\SS}_\cO| = |\Sigma^{\IR}_\cO| =  (q-1)(q^{\ell'-1}) .$ 
		\item For $\cO \in  \dvrtwozero,$  $|\Sigma^{\SS}_\cO| = |\Sigma^{\IR}_\cO| =  \begin{cases}  \frac{(q-1)(q^{\ell'-1})}{2}  & \mathrm{if}\,\,  \ell' > \ee, \\ (q-1)(q^{\ell'-1})& \mathrm{if}\,\,  \ell' \leq \ee. \end{cases}$ 
		\item For $\cO \in  \dvrtwozero,$ the set $\Sigma_\cO^\NS$ satisfies $|\Sigma_\cO^\NS| \asymp q^{\frac{r}{2}}.$ 
	\end{enumerate} 
\end{thm}
The results regarding  $\Sigma_\cO^\NS$ for $\cO \in  \dvrtwoplus$ were obtained by  H\"{a}s\"{a}-Stasinski~\cite{2017arXiv171009112H}, see Proposition~\ref{prop:SNS-orbits-Char-2}. We also refer the reader to Appendix~\ref{appendix} for a few comments regarding our notations vs those of H\"{a}s\"{a}-Stasinski~\cite{2017arXiv171009112H}.
Our next lemma gives the information regarding stabilizers of $\SL_2(\cO_r)$ action on $\widehat{K^\ell}$. It illustrates an important difference between $p =2$ and $p \neq 2$ cases. Fix a lift $\tilde{A} =  \mat 0 {\tilde{a}^{-1} \tilde{\alpha}}{\tilde a}{\tilde{\beta}}  \in M_2(\cO_{r})$ of cyclic $A = \mat 0 {a^{-1}\alpha } a\beta \in M_2(\cO_{\ell'})$. For $i \in \{ \ell, \ell' \}$, define 
\[ \mathrm{h}_{\tilde{A}}^{i} = \{ x \in \cO_{r } \mid 2x = 0  \,\, \mathrm{mod}\, (\pi^{i} ), \, \,  x(x+\tilde \beta) = 0 \,\, \mathrm{mod}\, (\pi^{i} ) \}.
\]
Let  $\mathrm{H}_{\tilde{A}}^{i} = \{ e_x = \mat1{\tilde{a}^{-1}x}01 \mid x \in h_{\tilde{A}}^i \}  $. Then $\mathrm{H}_{\tilde{A}}^{i} $ is an abelian group for $i \in \{ \ell, \ell' \}$. We remark that for the definition of $h_{\tilde{A}}^i;$ $i=\ell' , \ell$ by a lift of $A=\mat 0 {a^{-1}\alpha } a\beta$ we always consider a matrix of the form $\mat 0 {\tilde{a}^{-1} \tilde{\alpha}}{\tilde a}{\tilde{\beta}}$. It suffices for our purposes to consider lifts of this kind, see Remark~\ref{remark:lift-dependence}. 
See Section~\ref{sec:Stabilizer-of-Sl2-action}
for a few results regarding  $h_{\tilde{A}}^{\ell'}$ and $h_{\tilde{A}}^{\ell}$. 

\begin{lemma}
	\label{lem:stabilizer-form}
	Let $\cO \in \dvrp$ and  $A = \mat 0 {a^{-1}\alpha } a\beta \in M_2(\cO_{\ell'})$ be cyclic. Then  
	\begin{enumerate}	
		\item  $ C_{\GL_2(\cO_{r})} (\psi_A)  = C_{\GL_2(\cO_{r})} (\tilde{A}) M^{\ell'}.$ 
		\item For $p \neq 2,$  $ C_{\SL_2(\cO_{r})} (\psi_{[A]})  = C_{\SL_2(\cO_{r})}(\psi_A).$ 
		\item For $p=2,$ $C_{\SL_2(\cO_{r})} (\psi_{[A]})  = C_{\SL_2(\cO_{r})} (\psi_{A}) \mathrm{H}_{\tilde{A}}^{\ell'}.$ 	
	\end{enumerate}	
\end{lemma}

See Section~\ref{sec:Stabilizer-of-Sl2-action} for its proof. 
The group $C_{\GL_2(\cO_{r})} (\tilde{A})$ is abelian for cyclic $A$. Hence the one-dimensional representations $\psi_A$ and $\psi_{[A]}$ extend to the groups $C_{\GL_2(\cO_{r})} (\tilde{A}) M^{\ell}$ and $C_{\GL_2(\cO_{r})} (\tilde{A}) M^{\ell} \cap \SL_2(\cO_r)$ respectively (see Lemmas~\ref{lem:centralizer-form} and \ref{lem:centralizer-sl}). Now onwards, we use the following notations for simplification:

\begin{itemize} 
	\item $C_G^\ell(\tilde{A}) :=  C_{\GL_2(\cO_{r})} (\tilde{A}) M^{\ell}$ and $C_G^{\ell'}(\tilde{A}) := C_{\GL_2(\cO_{r})} (\tilde{A}) M^{\ell'}.$
	\item  $C_S^\ell(\tilde A):= C_G^\ell(\tilde{A}) \cap \SL_2(\cO_{r})$ and $C_S^{\ell'}(\tilde A) := C_G^{\ell'}(\tilde{A}) \cap \SL_2(\cO_{r}).$
\end{itemize}

Let $\mathcal{E}_{\tilde{A}}$ be the set of all extensions of $\psi_{[A]}$ to  $C_S^\ell(\tilde{A}).$ For any $\chi \in \mathcal{E}_{\tilde{A}},$ let
$
H_{\tilde{A}}(\chi) = \{ g \in  C_S^\ell(\tilde{A}) \mathrm{H}_{\tilde{A}}^{\ell} \mid \chi^g = \chi  \}. 
$
The group $\mathrm{H}_{\tilde{A}}^\ell$
is abelian, therefore each $\chi \in \mathcal{E}_{\tilde{A}}$ can be extended to a character of $H_{\tilde{A}}(\chi).$
Consider the set
\[
\mathbb E_{\tilde{A}} = \{e_\lambda  = \mat 1{\tilde{a}^{-1} \lambda} 01 \in \mathrm{H}_{\tilde{A}}^\ell \mid \psi_{[A]}\,\, \mathrm{extends}\,\, \mathrm{to}\,\, C_S^{\ell}(\tilde{A}) \langle e_\lambda\rangle\},
\]
where $\langle e_\lambda\rangle$ denotes the group generated by $ e_\lambda.$ 
It is clear that $H_{\tilde{A}}(\chi) \cap  \mathrm{H}_{\tilde{A}}^\ell  \subseteq \mathbb E_{\tilde{A}}.$ 
Consider the hierarchy of groups given in Figure~\ref{hierarchy-of-groups}, where $A \rightarrow B$ means $A \subseteq B.$

\begin{tiny}
\begin{figure}[h]
\[
\xymatrix{
		& & \SL_2(\cO_r) \\ 
		& & C_{\SL_2(\cO_{r})}(\psi_{[A]})  \ar[u]  \\
		C_S^\ell(\tilde{A}) \mathrm{H}_{\tilde{A}}^{\ell} \ar[urr]  &  &   \\
			C_S^\ell(\tilde{A})\mathbb E_{\tilde{A}} \ar[u]  &  &   \\
		H_{\tilde{A}}(\chi)\ar[u]& & C_{\SL_2(\cO_{r})}(\psi_A) \ar[uuu]\\
		& C_S^\ell(\tilde{A})  \ar[ul] \ar[ur] & \\
		& K^\ell \ar[u] & \\
	} 
	\]
	\caption{}\label{hierarchy-of-groups}
\end{figure}
\end{tiny}  

The set $\mathbb E_{\tilde{A}}$ plays a very important role throughout this article. Our first goal is to  characterize the elements of $\mathbb E_{\tilde{A}}.$ 
We remark that this characterization for  zero characteristic $\cO$ case is easier than the positive characteristic case. We include the statements of all results regarding the characterization of $\mathbb E_{\tilde{A}}$ in Section~\ref{sec:characterization-of-Ea-results}. Proofs of the main results of this section are complicated and tedious, so we postpone those to Sections~\ref{sec:proof of thm:condition for extension},  \ref{sec:proof of thm:simplification of extension conditions} and  \ref{sec:proof of thm:quotient_abelian}.
 We use the characterization of $\mathbb E_{\tilde{A}}$ to prove asymptotic results regarding the numbers and dimensions of irreducible representations of $\SL_2(\cO_{r}).$
 We require the following notation to state our results explicitly.
\[
\mathbf R_\cO = \begin{cases}  4 \ee&  \mathrm{if}\,\, \cO \in \dvrtwozero \mathrm{\,\,and\,\, ramification \,\, index \,\,of \,\, }\cO = \ee,   \\   2 &  \mathrm{if}\,\,  \cO \in \dvrtwoplus.  \end{cases}
\]

 \begin{thm}
 	\label{thm: Na-and-dimension-asymptomatics}
 	Let $\cO \in \dvrtwo$ and $r \geq \mathbf R_\cO.$ Let  $A \in M_2(\cO_{\ell'})$ be cyclic and $n_A =  | \mathrm{Irr}(\SL_2(\cO_r) \mid \psi_{[A]}) |$. 
 	\begin{enumerate} 
 		\item $ n_A \asymp \frac{	|C_{\SL_2(\cO_{\ell'})}(A)| \times q^{\ell}}
 		{|\mathrm{h}_{\tilde{A}}^{\ell'}|}.$
 		\item 	Every $\rho \in  \mathrm{Irr}(\SL_2(\cO_r) \mid \psi_{[A]}) $ satisfies $$\dim(\rho) \asymp  \frac{|\SL_2(\cO_r)|}{ 	|C_{\SL_2(\cO_{\ell'})}(A)| \times q^{3\ell}}.$$
 	\end{enumerate} 
 \end{thm} 
 
 A proof of this result is included in Section~\ref{sec:proof-of-main-thms}.
 This result and a few results of Sections~\ref{sec:Stabilizer-of-Sl2-action} and \ref{sec:orbits-stabilizers} play major role in proving Theorem~\ref{thm:abs-of-convergence} (see Section~\ref{sec:abs-of-conv}). 
 We also use the characterization results of $\mathbb E_{\tilde{A}}$  to give a construction of all primitive irreducible representations of $\SL_2(\cO_{2 \ell})$. 

\begin{thm}
	\label{main_theorem-2}
	Let $\cO \in \dvrtwo$  and $2\ell \geq \mathbf R_\cO.$ The following holds for every cyclic $A \in M_2(\cO_{\ell})$ and the group $\mathbb M_A = C_{\SL_2(\cO_{2\ell})}(\psi_{A}) \mathbb E_{\tilde{A}}$. 
	\begin{enumerate}
		
		\item The one-dimensional representation $\psi_{[A]}$ of $K^\ell$  extends to the group $\mathbb {M}_A.$ 
		\item The group $\mathbb M_A/K^\ell$ is either abelian or is a semidirect product of an abelian group with a group of order two. Hence every representation $\rho \in  \mathrm{Irr}(\mathbb M_A\mid  \psi_{[A]}) $ is of dimension either  one or two. 
		\item By Clifford theory, the set of equivalence classes of  $\mathrm{Irr}(\mathbb {M}_A\mid  \psi_{[A]})$ under the conjugation action of $C_{\SL_2(\cO_{2\ell})}(\psi_{[A]})$ is in bijective correspondence with $\mathrm{Irr}( \SL_2(\cO_{2\ell})\mid  \psi_{[A]})$ given by $\rho \mapsto \mathrm{Ind}_{\mathbb {M}_A}^{\SL_2(\cO_{2\ell})}(\rho)$.

	\end{enumerate}
\end{thm} 
See Section~\ref{sec:construction of odd level rep} for its proof. Some qualitative corollaries of Theorem~\ref{main_theorem-2} are discussed at the start of Section~\ref{sec:Examples}. Let $\calp_G(X)$ be the representation zeta polynomial of the finite group $G$ (see Definition~\ref{def:rep zeta poly}).
We use Theorem~\ref{main_theorem-2} to prove the following result in Section~\ref{sec:group-algebras}.
\begin{thm}\label{thm:zeta poly not eaqal- gp algebra} Let $\cO \in \dvrtwozero$ with ramification index $\ee $ and $\cO' \in \dvrtwoplus$ such that $\cO/\wp \cong \cO'/\wp'.$  Then  
	\[
	\calp_{\SL_2(\cO_{2 \ell})}(X)\neq \calp_{\SL_2(\cO'_{2 \ell}) }(X)
	\]
	for any $\ell > \ee.$
\end{thm}

Theorem~\ref{thm:group-algebras} directly follows from this result. For reader's convenience, below we describe the interdependencies of various sections, where $\xymatrix{Section~A \ar@{:>}[r] &Section~B}$ denotes that the results of Section $A$ are used in Section~B. 
\[
\xymatrix{
& & Section~\ref{sec:introduction} \ar@{:>}[d] & & \\
& & Section~\ref{sec:basic-framework}\ar@{:>}[d] \ar@{:>}[rrd] \ar@{:>}[lld] & & \\
Section~\ref{sec:Stabilizer-of-Sl2-action} \ar@{:>}[rrd] & & Section~\ref{sec:orbits-stabilizers} \ar@{:>}[d]  & & Section~\ref{sec:characterization-of-Ea-results} \ar@{:>}[lld] \ar@{:>}[d] \\ 
&  & Sections~\ref{sec:rep-growth},\ref{sec:odd-level-reps},\ref{sec:Examples} & & Sections~\ref{sec:proof of thm:condition for extension}, \ref{sec:proof of thm:simplification of extension conditions}, 
\ref{sec:proof of thm:quotient_abelian}
}
\]

{\bf Now onwards, except  for Lemma~\ref{lem:centralizer-form} and the proof of  Lemma~\ref{lem:stabilizer-form}, we always assume that $\cO \in \dvrtwo$ and  $\cO/\wp \cong \mathbb F_q$. For $\cO \in \dvrtwozero$, we also fix $\ee$ to be the ramification index of $\cO$, that is, $2 \cO = \pi^\ee$. } Note that the ramification index $\ee$ depends on $\cO$. However, we deal with only one $\cO \in \dvrtwozero$ at a given time, so there will not be any confusion regarding this notation. 

 \section{Stabilizers of $\widehat{K^\ell}$ under $\SL_2(\cO_{ r})$ action}
\label{sec:Stabilizer-of-Sl2-action}

This section contains results regarding stabilizers of $\SL_2(\cO_r)$ action on $\widehat{K^\ell}$. The following lemma is well known, see Hill~\cite{MR1334228} for a proof. 
\begin{lemma}
	\label{lem:centralizer-form} 
	Let $\cO \in \dvrp$  for a prime $p$. Let $m, r \in \mathbb N$ such that $r \geq m$. The following are true for any cyclic matrix $A \in M_2(\cO_m)$. 
	\begin{enumerate}
		\item Any lift $\tilde{A} \in M_2(\cO_r)$ of $A$ is cyclic. 
		\item The centralizer of $A$ in $\GL_2(\cO_m),$ denoted $C_{\GL_2(\cO_m)}(A) ,$ consists of invertible matrices of the form $x I + yA$ for $x, y \in \cO_m.$
		\item $\rho_{r, m}(C_{\GL_2(\cO_r)}(\tilde{A})) =  C_{\GL_2(\cO_m)}(A)$ for every lift $\tilde{A}$ of $A$. 
		\item For every cyclic $A \in M_2(\cO_{\ell'}),$ the one-dimensional representation $\psi_A \in \widehat{M^\ell}$ extends to $C_G^\ell(\tilde{A})=C_{\GL_2(\cO_{r})} (\tilde{A}) M^{\ell}$ for any lift $\tilde{A} \in M_2(\cO_r)$ of $A.$ 
	\end{enumerate}
\end{lemma} 

This lemma immediately implies the next result regarding $C_S^\ell(\tilde{A})$. 
\begin{lemma}
	\label{lem:centralizer-sl}
		Let $\cO \in \dvrtwo$. The following are true for any cyclic matrix $A \in M_2(\cO_{\ell'})$. 
	\begin{enumerate}
		\item The one-dimensional representation $\psi_{[A]}$ extends to $C_S^\ell(\tilde{A}) = (C_{\GL_2(\cO_{r})} (\tilde{A}) M^{\ell}) \cap \SL_2(\cO_r)$ for any lift $\tilde{A} \in M_2(\cO_r)$ of $A.$ 
		\item The quotient group $C_S^\ell(\tilde{A}) /K^\ell$ is abelian. 
		
	\end{enumerate}
\end{lemma} 
\begin{proof} Here (1) follows from Lemma~\ref{lem:centralizer-form}(3) and the fact that $\psi_{[A]}$ is a restriction of $\psi_A.$ For (2), we claim that $[C_{\GL_2(\cO_{r})} (\tilde{A}) M^{\ell}, C_{\GL_2(\cO_{r})} (\tilde{A}) M^{\ell}] \subseteq K^\ell.$ This 0follows because the quotient $\frac{C_{\GL_2(\cO_{r})} (\tilde{A}) M^{\ell}}{M^\ell}$ is abelian by Lemma~\ref{lem:centralizer-form}(2) and every element of the commutator $[C_{\GL_2(\cO_{r})} (\tilde{A}) M^{\ell}, C_{\GL_2(\cO_{r})} (\tilde{A}) M^{\ell}]$ has determinant one.  	
\end{proof}

The above result is used very frequently in this article. Sometimes we use it even without explicitly mentioning it. When $\bar{A}$ is cyclic, the following lemma is easy and well known in the literature.
\begin{lemma}
	\label{gl-centralizer-cardinality}
	Let $A \in M_2(\cO_m)$ such that $\bar{A}$ is cyclic. Then  
	\[
	|C_{\GL_2(\cO_{m})}(A)| = \begin{cases} (q^2-1)q^{2 m -2} &
	\,\, \mathrm{for} \,\, \bar{A} \,\, \mathrm{irreducible}, \\ (q-1)^2 q^{2m -2}&
	\,\, \mathrm{for} \,\, \bar{A} \,\, \mathrm{split\,\, semisimple}, \\ (q^2-q)q^{2 m -2} & \,\, \mathrm{for} \,\, 
	\bar{A} \,\, \mathrm{ split, non-semisimple} . \end{cases}
	\]
\end{lemma}
\begin{lemma} 
	\label{lem:image-det-map} For 
	$A = \mat 0 {a^{-1}\alpha} a \beta \in M_2(\cO_m)$ such that $\beta$ is invertible, the group $C_{\GL_2(\cO_{m})}(A)$ maps onto $\cO_{m}^\times$ under the determinant map. 
	
\end{lemma} 
\begin{proof} 
	The image of determinant map from $C_{\GL_2(\cO_{m})}(A)$  consists of $\{ x^2 + \beta xy -\alpha y^2 \mid x,y \in \cO_{m}\} \cap \cO_{m}^\times.$ To prove the result, it is enough to prove $ \cO_{m}^\times \subseteq \{ x^2 +\beta xy - \alpha y^2 \mid x,y \in \cO_{m} \} $. 
	Let  $w \in \cO_{m}^\times.$  The map $v\mapsto v^2$ is a bijection on $\mathbb F_q^\times$. Therefore there exists $w' \in \cO_{m}^\times$ such that $w(w')^{-2} = 1 + z$ for some $z \in \pi \cO_{m}.$  
	For $z \in \pi \cO_{m},$  we note that $y=0$ is a solution   for $\beta y - \alpha y^2 = z$  in the residue field. 
	Therefore by Hensel's lemma, $\beta y - \alpha y^2 = z$  has a solution in  $\cO_{ m}.$ 
	Let $y_z \in  \cO_{m} $ such that $\beta y_z -\alpha  y_z^2 = z.$ Then 
	$$ w = (w')^2 (1+z) = (w')^2 +\beta  w' (w'y_z) -\alpha (w' y_z) ^2 \in \{ x^2 +\beta xy - \alpha y^2 \mid x,y \in \cO_{m} \}. $$ This completes the proof. 
\end{proof} 

\begin{lemma}
	\label{lem:SNS-char-o-proof}
	For $\cO \in \dvrtwozero$ and $j\geq 1,$ $|\{x^2 \mid x\in \cO_j^\times\}|\asymp |\cO_j|=q^j.$
\end{lemma} 
\begin{proof} For the proof, we show that $1+\pi^{2\ee+1} \cO_j \subseteq \{x^2 \mid x\in \cO_j^\times\}.$ Since $2 \in \pi^{\ee} \cO_j^\times,$ there exists $w\in \cO_j^\times$ such that $2=\pi^{\ee} w.$ By Hensel's lemma, for every $x \in \cO_j$ there exists $y\in \cO_j$ such that $y+w^{-1}y^2=\pi^e w^{-1}x.$ The result  now follows from 
$$1+\pi^{2\ee+1} x=1+2\pi^\ee (y+w^{-1}y^2)=1+2\pi^\ee y +(\pi^\ee y)^2=(1+\pi^\ee y)^2$$ and  $|1+\pi^{2\ee+1} \cO_j| \asymp q^j$ for all $j \geq 1$. 
\end{proof} 

Let $\cO \in \dvrtwoplus$ and $A = \mat 0 {a^{-1} \alpha} a \beta \in M_2(\cO_{\ell'}).$
We define $\bol s$ and $\bol k$ of $A$ as follows (see also Section~\ref{sec:characterization-of-Ea-results}).
\begin{eqnarray}
\bol  k & = &  \val(\beta). \nonumber \\
\bol s & = & \begin{cases}
2 \lfloor \bol k/2 \rfloor + 1 &  \mathrm{if}\,\,   \alpha = v^2 \,\,\mathrm{mod}\,(\pi^ \bol k ),  \\ m&  \mathrm{if}\,\, \alpha = v_1^2 + \pi^m v_2^2 \,\,\mathrm{mod}\,(\pi^ \bol k ) \,\,\mathrm{for}\,\,\mathrm{odd}\,\, m < \bol k\,\,\mathrm{and}\,\, v_2 \in \cO_{\ell'}^\times. 
\end{cases}\nonumber
\end{eqnarray}
We note that the parameters $\bol k$ and $\bol s$  depend only on the trace and determinant of $A.$ So these parameters are conjugation invariant.

\begin{thm}\label{thm:C-S-L asymp to..}
	Let $j \geq 1$ and $A = \mat{0 }{a^{-1}\alpha}{a}{\beta} \in M_2(\cO_{j})$ be cyclic. Then 
	\begin{enumerate} 
		\item For $A \in \mathrm{Type}(\SS)$, 
		$|C_{\SL_2(\cO_j)}(A)|= (q-1)q^{j-1}$. 
		\item  For $A \in \mathrm{Type}(\IR)$, 
		$|C_{\SL_2(\cO_j)}(A)|= (q+1)q^{j-1}$		
		\item	For $A \in \mathrm{Type}(\NS),$  
		\begin{equation}
		|C_{\SL_2(\cO_j)}(A)|\asymp \begin{cases}q^{j}&  \mathrm{for}\,\, \cO \in \dvrtwozero ,\\
		q^{j +\lfloor \frac{\bol s}{2}  \rfloor} & \mathrm{for}\,\, \cO \in \dvrtwoplus.
		\end{cases}
		\end{equation}
	\end{enumerate} 
\end{thm}
\begin{proof}
	For $\mathrm{Type}(\SS)$ and $\mathrm{Type}(\IR),$ the result follows from Lemma~\ref{gl-centralizer-cardinality} and Lemma~\ref{lem:image-det-map}.
	For $\mathrm{Type}(\NS)$ with $\cO \in \dvrtwozero,$  Lemmas~\ref{gl-centralizer-cardinality} and \ref{lem:SNS-char-o-proof} together with the fact that  $\{x^2 \mid x\in \cO_{\ell'}^\times \}\subseteq \det(C_{\GL_2(\cO_{\ell'})}(A)$ will give the result.
	The results for $\mathrm{Type}(\NS)$ with $\cO\in \dvrtwoplus$ follows by \cite[Lemma 4.21]{2017arXiv171009112H}(see also \ref{appendix}).
\end{proof}

We proceed now to prove Lemma~\ref{lem:stabilizer-form}. 

\begin{proof}[{\bf Proof of  Lemma~\ref{lem:stabilizer-form}} ]  By the definition of     $C_{\GL_2(\cO_{r})} (\psi_A) $ and Lemma~\ref{lem:centralizer-form}(3), we obtain $C_{\GL_2(\cO_{r})} (\psi_A)  = C_{\GL_2(\cO_{r})} (\tilde{A}) M^{\ell'}.$ For $p\neq 2$ we note that $A$ and $A+xI$ are not conjugate for any $x \neq 0.$ Therefore $C_{\SL_2(\cO_{r})} (\psi_{[A]})  = C_{\SL_2(\cO_{r})} (\psi_A) .$ Now onwards, we assume $p=2.$ First of all note that for $\lambda \in \mathrm{h}_{\tilde{A}}^{\ell'},$
	$$\mat{1}{\tilde{a}^{-1} \lambda}{0}{1} \tilde{A}\mat{1}{\tilde{a}^{-1} \lambda}{0}{1}^{-1} =\tilde{A}+\lambda \mathrm{I}\,\, \mathrm{mod}\,(\pi^{\ell'}).$$
	Hence 
	$   C_{\SL_2(\cO_{r})} (\psi_{A}) \mathrm{H}_{\tilde{A}}^{\ell'} \subseteq C_{\SL_2(\cO_{r})} (\psi_{[A]}).$ Conversely, for $g \in C_{\SL_2(\cO_{r})} (\psi_{[A]}),$ there exists a $\lambda \in \cO_{ r}$ such that $g\tilde{A}g^{-1}=\tilde{A}+\lambda \mathrm{I}\,\, \mathrm{mod}\,(\pi^{\ell'}).$
	By equating trace and determinant, we obtain $\lambda \in \mathrm{h}_{\tilde{A}}^{\ell'}.$ Therefore
	$$g\mat{1}{\tilde{a}^{-1} \lambda}{0}{1}^{-1} \tilde{A} \mat{1}{\tilde{a}^{-1} \lambda}{0}{1}g^{-1} =g(\tilde{A} -\lambda \mathrm{I})g^{-1}=\tilde{A} \,\, \mathrm{mod}\,(\pi^{\ell'}).$$
	Thus $g\mat{1}{\tilde{a}^{-1} \lambda}{0}{1}^{-1}\in    C_{\SL_2(\cO_{r})} (\psi_{A}).$ Hence $C_{\SL_2(\cO_{r})} (\psi_{[A]}) \subseteq   C_{\SL_2(\cO_{r})} (\psi_{A}) \mathrm{H}_{\tilde{A}}^{\ell'} .$
\end{proof}

The following remark relating the lift $\tilde{A}$ of any cyclic $A \in M_2(\cO_{\ell'})$ and the group $C_{\SL_2(\cO_{r})} (\psi_{[A]}) $ is important. 
\begin{remark}
	\label{remark:lift-dependence}  The group $H_{\tilde{A}}^{\ell'}$ depends only  on $A$ and not on $\tilde{A}$ for any cyclic $A=\mat{0 }{a^{-1}\alpha}{a}{\beta} \in M_2(\cO_{\ell'})$  and its lift $\tilde{A} \in M_2(\cO_r)$.  This, along with Lemmas~\ref{lem:stabilizer-form} and ~\ref{lem:centralizer-form}, implies that the group $C_{\SL_2(\cO_{r})} (\psi_{[A]})$  does not depend on the choice of lift $\tilde{A}$ of $A.$ 
\end{remark} 
In Sections~\ref{sec:characterization-of-Ea-results} and \ref{sec:proof-of-main-thms}, we will make a choice of $\tilde{A}$ according to our convenience.
By above remark this does not effect the final construction of irreducible representations of $\SL_2(\cO_r).$ 
We need the following result about $\mathrm{h}_{\tilde{A}}^{\ell}$ and $\mathrm{h}_{\tilde{A}}^{\ell'}.$ A proof of this result is available in 
\cite[Lemma~5.5]{2017arXiv171009112H}. We include it here for reader's convenience.

\begin{proposition}
	\label{prop:charactrization of h_Atilda^ell and h_Atilda^ell'}
	Let $j=\ell,\ell'.$ Recall that  $\mathrm{h}_{\tilde{A}}^{j} = \{ x \in \cO_{r } \mid 2x = 0  \,\, \mathrm{mod}\, (\pi^{j} ), \, \,  x(x+\tilde \beta) = 0 \,\, \mathrm{mod}\, (\pi^{j} ) \}.$  
	\begin{enumerate}
		\item For $\cO \in \dvrtwoplus $ and  $ j\geq 2,$ 
		$$\mathrm{h}_{\tilde{A}}^{j}=\begin{cases}
			\{ 0, \tilde \beta\}+	\pi^{j-\val(\beta)} \cO_r &   \mathrm{if}\,\,  \val(\tilde\beta)<\lceil \frac{j}{2} \rceil ,\\
			\pi^{\lceil\frac{j}{2}\rceil} \cO_r &  \mathrm{if}\,\,\bol \val(\tilde\beta)\geq\lceil \frac{j}{2} \rceil.
		\end{cases}$$
		\item For $\cO \in \dvrtwozero$ and  $ j\geq 2\ee,$ 
		$	\mathrm{h}_{\tilde{A}}^{j}=\pi^{j-\min \{\ee, \val(\tilde\beta) \}}\cO_r.$
		This result also holds for $j> \ee $ and $\beta \in \cO_{\ell'}^\times.$
	\end{enumerate}
	
\end{proposition}
\begin{proof}
	Let $\cO \in \dvrtwoplus.$ 
	In this case  $2=0$ implies
	$$\mathrm{h}_{\tilde{A}}^{j}=\{x \in \cO_r \mid x(x+\tilde \beta)\in \pi^j \cO_r \}=\{x \in \cO_r \mid \val(x)
	+\val(x+\tilde \beta)\geq j \}.$$
	This gives 
	$$\mathrm{h}_{\tilde{A}}^{j}=\begin{cases}
		\{ 0, \tilde \beta\}+	\pi^{j-\val(\beta)} \cO_r &   \mathrm{if}\,\, \val(\tilde\beta)<\lceil \frac{j}{2} \rceil ,\\
		\pi^{\lceil\frac{j}{2}\rceil} \cO_r &   \mathrm{if}\,\, \val(\tilde\beta)\geq\lceil \frac{j}{2} \rceil.
	\end{cases}$$
	
	Let $\cO \in \dvrtwozero$ and  $ j\geq 2\ee.$ 
	Then $2\in \pi^\ee\cO_r^\times$ and
	$2x= 0 \,\, \mathrm{mod}\, (\pi^{j} )$ imply $x\in \pi^{j-\ee}\cO_r.$ Therefore 
	\begin{eqnarray*}
		\mathrm{h}_{\tilde{A}}^{j}&=& \{  x\in \pi^{j-\ee}\cO_r \mid x(x+\tilde \beta)= 0 \,\, \mathrm{mod}\, (\pi^{j} )\}\\
		&=&\begin{cases}
			\{ x\in \pi^{j-\ee}\cO_r\mid \val(x)+\val(\tilde{\beta})\geq j \}  & \mathrm{if}\,\,  \val(\tilde{\beta})<j-\ee, \\
			\{ x\in \pi^{j-\ee}\cO_r \mid \val(x)+\val(x+\tilde{\beta})\geq j \} & \mathrm{if}\,\,   \val(\tilde{\beta})\geq j-\ee.
		\end{cases}
		\\
		&=&\begin{cases}
			\pi^{j-\min \{\ee, \val(\tilde\beta) \}} \cO_r & \mathrm{if}\,\,  \val(\tilde{\beta})<j-\ee, \\
			\pi^{j-\ee} \cO_r &  \mathrm{if}\,\,   \val(\tilde{\beta})\geq j-\ee.
		\end{cases}
	\end{eqnarray*} 
	This completes the proof for $j \geq 2e$. We now assume that $\beta \in \cO_{\ell'}^\times$ and $j > e$.  Then $2x= 0 \,\, \mathrm{mod}\, (\pi^{j} )$ for  $j>\ee$  implies $x\in \pi^{j-\ee}\cO_r$. Hence 
	$\mathrm{h}_{\tilde{A}}^{j}= \{ x \in\pi^{j-\ee}\cO_r\mid x(x+\tilde \beta)= 0 \,\, \mathrm{mod}\, (\pi^{j} )\}.$ Note that for $x \in \pi^{j-\ee} \cO_r,$ we have $x+\tilde \beta\in \cO_{r}^\times.$  Therefore 
	$\mathrm{h}_{\tilde{A}}^{j}= \{ x \in\pi^{j-\ee}\cO_r\mid x= 0 \,\, \mathrm{mod}\, (\pi^{j} )\}=\pi^{j} \cO_r  = \pi^{j-\min \{\ee, \val(\tilde\beta) \}}\cO_r.$
\end{proof}
The following corollary is immediate. 
\begin{corollary}
	\label{h_A/h_a^l cardinality} 
The following holds for $r\geq \mathbf R_{\cO}$ and cyclic $A \in M_2(\cO_{\ell'})$. 
\begin{enumerate}
 \item $[\mathrm{h}_{\tilde{A}}^{\ell'}:\mathrm{h}_{\tilde{A}}^{\ell}]\leq q$. 
  \item For $\cO \in \dvrtwozero,$ $|\mathrm{h}_{\tilde{A}}^{\ell'}|\asymp q^{\frac{r}{2}}.$
 \item For $\cO \in \dvrtwoplus,$ $$|\mathrm{h}_{\tilde{A}}^{\ell'}|\asymp \begin{cases}
	q^{\frac{r}{2}+\bol k} &  \mathrm{if}\,\, \bol k< \frac{\ell'}{2},\\
	q^{\frac{3r}{4}} &   \mathrm{if}\,\, \bol k\geq \frac{\ell'}{2}.
	\end{cases} $$
\end{enumerate} 
\end{corollary}

\section{Orbits of $\widehat{K^\ell}$ under $\SL_2(\cO_{ r})$ action} 
\label{sec:orbits-stabilizers}

In this section, we consider the action of group  $\SL_2(\cO_r)$ on the set of cyclic one-dimensional representations of $K^\ell$ and describe a few results regarding the orbit representatives of this action.

\begin{lemma}
	A  matrix $A \in M_2(\cO_{m})$  is cyclic if and only if
	$gAg^{-1} =  \left[\begin{matrix} 0 & a^{-1}\alpha \\ a & \beta \end{matrix}\right]$ for some $g \in \SL_2(\cO_{m}),$  $\alpha, \beta \in \cO_m$ and  $a \in \cO_m^{\times}.$

\end{lemma}
\begin{proof} For $a \in \cO_m^\times,$ it is straightforward to show that any matrix $A$ such that $ gAg^{-1} = \left[\begin{matrix} 0 & a^{-1}\alpha \\ a & \beta \end{matrix}\right]$ for some $\alpha, \beta \in \cO_m$ is cyclic. We proceed to prove the converse. 
	Recall that a matrix $A \in M_2(\cO_m)$ is cyclic if and only if there exists $v \in \cO_m \oplus \cO_m$ such that the set $\{v, Av\}$ is a basis of the free $\cO_m$-module $\cO_m \oplus \cO_m.$  Therefore there exists $g \in \GL_2(\cO_m)$ such that $gAg^{-1} =  \left[\begin{matrix} 0 & \alpha \\ 1 & \beta \end{matrix}\right],$ where $\beta = \mathrm{trace}(A)$ and $\alpha = -\det(A).$ 
	Consider $g' = \mat {\det(g)^{-1}}{0}{0}1 g.$ Then $g' \in \SL_2(\cO_m)$ and $g' A (g')^{-1} =  \left[\begin{matrix} 0 & \det(g)^{-1}\alpha \\ \det(g) & \beta \end{matrix}\right] .$
	This completes the proof of the lemma. 
\end{proof}

We denote the image of any subgroup $H\subseteq \GL_2(\cO_r)$ under the determinant map $\det : \GL_2(\cO_r) \rightarrow \cO_r^\times$ by $\det(H).$ 
\begin{proposition}
\label{equivalence}  Let $B_1 = \left[ \begin{matrix}
0 & a_1^{-1} \alpha_1 \\ a_1 & \beta_1 
\end{matrix} \right]$ and $B_2 = \left[ \begin{matrix}
0 & a_2^{-1}\alpha_2 \\ a_2 & \beta_2 
\end{matrix} \right]$ be two elements of $M_2(\cO_{ \ell'})$ and $\psi_{[B_1]}$ and $\psi_{[B_2]}$ be the corresponding one-dimensional representations of $K^\ell.$ Then $\psi_{[B_1]}^g = \psi_{[B_2]}$ for some $g \in \SL_2(\cO_r)$  if and only if there exists $s \in \cO_{\ell'}$ such that the following holds.
\begin{enumerate}
\item $\beta_2 = \beta_1 - 2s,$
\item  $\alpha_2 = \alpha_1- s^2  + s \beta_1,$
\item  $a_2a_1^{-1} \in \det(C_{\GL_2(\cO_{ \ell'})}(B_1)).$
\end{enumerate}

 \end{proposition}
 \begin{proof} The group $K^{\ell}$ acts trivially on $\widehat{K^\ell},$ therefore $\psi_{[B_1]}^g = \psi_{[B_2]}$ for some $g \in \SL_2(\cO_r)$ if and only if $g'[B_1]g'^{-1}= [B_2]$ for some $g' \in \SL_2(\cO_{\ell'}).$ We obtain conditions $(1)$ and $(2)$ by comparing the trace and determinant of $g'B_1g'^{-1}$ and $B_2+ s I_2.$ By comparing the $(2,1)^{th}$ entry of $g' B_1 g'^{-1} $ for $g' = \left[ \begin{matrix} x & y \\ z & w \end{matrix} \right]  \in \SL_2(\cO_{\ell'})$ and that of $B_2 + sI$ for $s \in \cO_{\ell'},$ we obtain 
 \[
 a_2 = a_1(w^2 - \beta_1 w(z/a_1) - \alpha_1(z/a_1)^2).
 \]
Now $(3)$ follows by observing that $w^2 - \beta_1 w(z/a_1) - \alpha_1(z/a_1)^2 = \det(wI - (z/a_1) B_1 )$ and $wI - (z/a_1) B_1  \in C_{\GL_2(\cO_{ \ell'})}(B_1).$ Conversely suppose  $(1)$-$(3)$ hold, that is, there exists $s \in \cO_{\ell'}$ such that
$\beta_2 = \beta_1 - 2s,$
 $\alpha_2 = \alpha_1- s^2  + s \beta_1$
 and there exists $u,v \in \cO_{\ell'}$ such that $uI + vB_1 \in C_{\GL_2(\cO_{ \ell'})}(B_1)$ and $a_ 2 = a_1 \det(uI + vB_1).$ The condition $uI + vB_1 \in C_{\GL_2(\cO_{ \ell'})}(B_1)$ further implies that  either $u$ or $v$ is invertible. We consider an element $g' \in \SL_2(\cO_{\ell'}),$
 as follows.
 \[
 g' = \begin{cases} \left[  \begin{matrix}
 u^{-1}(1- va_1a_2^{-1}(su-v \alpha_1)) & a_2^{-1}(su-v\alpha_1) \\
 -a_1v  & u
 \end{matrix}  \right]& \mathrm{if} \, \, u \in \cO_{\ell'}^{\times},\\
 \left[  \begin{matrix}
a_1 a_2^{-1} (v \beta_1 + u - vs) & a_1^{-1}v^{-1}- a_2^{-1} v^{-1} u (v \beta_1 + u - vs)  \\
  -a_1v  & u
  \end{matrix}  \right] & \mathrm{otherwise}.
 \end{cases}
\]
Then $g' B_1 g'^{-1} = B_2 + sI$.  This completes the proof of the proposition. 
\end{proof}

Motivated by  Proposition~\ref{equivalence}, we consider the set
\[
\Sigma = \{ (a, \alpha, \beta) \in \cO_{\ell'}^{\times} \times \cO_{\ell'} \times \cO_{\ell'} \}.
\] 
and define an equivalence relation on it as follows. Two elements $(a_1, \alpha_1, \beta_1)$ and  $(a_2, \alpha_2, \beta_2)$ of $\Sigma$ are equivalent if and only if there exists $u, v, s \in \cO_{\ell'}$ such that the following conditions are satisfied. 
\begin{eqnarray}
\beta_2 & = & \beta_1 - 2s, \label{condition-1}\\
 \alpha_2 & = & \alpha_1 - s^2 + s \beta_1, \label{condition-2}\\
a_2a_1^{-1} & = & u^2 + \beta_1 uv - \alpha_1 v^2. \label{condition-3}
\end{eqnarray}
The set of equivalence class representatives (also called set of orbit representatives) is in bijective correspondence with the orbits of cyclic one-dimensional representations of $\widehat{K^\ell}$ under the $\SL_2(\cO_r)$ action. In view of this, to describe the orbit representatives of $\widehat{K^\ell},$ we identify $A = \left[ \begin{matrix}
0 & a^{-1} \alpha \\ a & \beta 
\end{matrix} \right]$ also with corresponding tuple $(a, \alpha, \beta).$ We will use these notations most prominently in Section~\ref{sec:Examples}. 

We now proceed to  find orbit representatives of $\psi_{[A]}$ for cyclic $A\in M_2(\cO_{\ell'}).$ For this,  we need the following lemma first. 
We do not precisely know the set of all orbits representatives of $\widehat{K^\ell}$ and determine these in a few required
cases. 

\subsection{Orbit Representatives  for $\cO \in \dvrtwozero$} 
\label{orbit-representations-zero-ch}
\subsubsection{$\beta$ is invertible with $\ell' > \ee$}\label{invertible-orbits-number-field}
 Let $T$ be a fixed set of representatives of $\cO_{\ell'}^\times / (1 + \pi^\ee \cO_{\ell'}).$ By (\ref{condition-1}), (\ref{condition-2}), (\ref{condition-3}) and Lemma~\ref{lem:image-det-map}, there exist equivalence class representatives of the form $(1, \alpha, \beta ),$ where $\beta \in T$ and  $$(1, \alpha, \beta ) \sim (1, \alpha', \beta) \,\,\mathrm{if \,\, and \,\, only \,\,if} \,\, \alpha \in \alpha'  + \pi^{\ell' -\ee} \cO_{\ell'} .$$

\subsubsection{$\beta$ is not invertible with $\ell' \geq 2\ee$}\label{non-invertible-orbits-number-field} 
 Let $W$ be a fixed set of representatives of $\pi  \cO_{\ell'}/  \pi^\ee  \cO_{\ell'}.$ Then 
by (\ref{condition-1}), (\ref{condition-2}) and (\ref{condition-3}),  there exist equivalence class representatives of the form 
 $(a, \alpha, \beta ),$ where $\beta \in W$ and  $(a, \alpha, \beta ) \sim (a', \alpha', \beta')$ if and only if 
 \begin{enumerate} 
\item  $\beta=\beta',$ 
\item $\alpha \in \alpha'  + \pi^{\ell' -\ee}\beta  \cO_{\ell'},$
\item  $a'a^{-1} \in\{x^2+\beta xy -\alpha y^2 \mid x,y \in \cO_{\ell'}\}\cap \cO_{\ell'}^\times.$
\end{enumerate} 

\subsection{Orbit Representatives  for $\cO \in \dvrtwoplus$ with $\ell' \geq 1$}\label{invertible-orbits-function-field}
We first note from (\ref{condition-1}) that $(a, \alpha, \beta) \sim (a', \alpha' ,\beta')$ implies $ \beta= \beta'$.  
 \subsubsection{$\beta \in \cO_{\ell'}^\times$} The set $\{ s^2+ s \beta \mid s \in \cO_{ \ell'} \} $ is an additive subgroup of $\cO_{\ell'}$ of index two. Let $\alpha_\beta$ and $\alpha'_\beta$ be two distinct coset representatives of $\{ s^2+ s \beta \mid s \in \cO_{ \ell'} \} $ in $\cO_\ell'.$ By Lemma~\ref{lem:image-det-map}, and  (\ref{condition-1}), (\ref{condition-2}), (\ref{condition-3}),   the set of tuples $$\{ (1, \alpha_\beta, \beta), (1,\alpha'_\beta,\beta)\}_{\beta \in \cO_{ \ell'}^\times }$$ is the complete set of distinct equivalence class representatives in $\Sigma$ satisfying the condition that $\beta$ is invertible. 

\subsubsection{$\beta \in \pi \cO_{\ell'}$} For this case, we do not know the sets  $\{ s^2+ s \beta \mid s \in \cO_{ \ell'} \} $ and the image of $C_{\GL_2(\cO_{\ell'})}(A)$ under the determinant map in general. For $\cO=\mathbb F_2 \llbracket t \rrbracket $ and $1\leq \ell' \leq 3,$  we find the corresponding orbit representatives in section~\ref{sec:Examples}.

\begin{proof}[\bf Proof of Theorem~\ref{thm:number-of-orbits-split semisimple}]
Note that for a cyclic $A,$ the $\mathrm{trace}(A)$ is invertible if and only if $\bar{A}$ is either  split semisimple or irreducible. We first consider $\cO \in \dvrtwoplus$. 
From Section~\ref{invertible-orbits-function-field}, the orbit representatives for $A \in \Sigma$ such that $\mathrm{trace}(A)$ is invertible are  given by
$$\{[(1, \alpha_\beta, \beta)], [(1,\alpha'_\beta,\beta)] \}_{\beta \in \cO_{ \ell'}^\times },$$
where $\alpha_\beta$ and $\alpha'_\beta$ are two distinct coset representatives of $\{ s^2+ s \beta \mid s \in \cO_{ \ell'} \} $ in $\cO_{\ell'}.$
Now observe that exactly half of them are with $\bar{A}$  split-semisimple. 
So  $|\Sigma^{\SS}_\cO| = |\Sigma^{\IR}_\cO| = (q-1)(q^{\ell'-1}) .$

Next, we consider $\cO \in \dvrtwozero$. Note that for $ \ell' \leq \ee,$ $2=0$  in $\cO_{\ell'}.$  Therefore $\cO_{\ell'}\cong \cO'_{\ell'}$ where  $\cO' \in \dvrtwoplus.$  So  we assume $ \ell'> \ee.$ 
From Section~\ref{invertible-orbits-number-field}, the orbit representatives for $A \in \Sigma$ such that $\mathrm{trace}(A)$ is invertible are 
$\{[(1, \alpha, \beta)] \mid \beta \in T \,\, \mathrm{ and } \,\, \alpha \in B \}$, 
where $T$ (resp. $B$) is a fixed set of representatives of $\cO_{\ell'}^\times / (1 + \pi^{\ee} \cO_\ell')$ (resp. $\cO_{\ell'} / \pi^{\ell'-\ee} \cO_\ell'$). Now observe that exactly half of them are with $\bar{A}$  split semisimple and the other half are with $\bar{A}$  irreducible.
So  $|\Sigma^{\SS}_\cO| = |\Sigma^{\IR}_\cO| = \frac{(q-1)(q^{\ell'-1})}{2}.$ 

  We now focus on $|\Sigma_\cO^\NS|$ with $\cO \in \dvrtwozero$. Let $r\geq 4\ee .$ For any $\alpha,\beta \in \cO_{\ell'}, $ observe that $\{x^2 \mid x\in \cO_{\ell'}^\times\} \subseteq \{x^2+\beta xy -\alpha y^2 \mid x,y \in \cO_{\ell,}\}\cap \cO_{\ell'}^\times \subseteq \cO_{\ell'}.$  By Lemma~\ref{lem:SNS-char-o-proof},  $ \{x^2+\beta xy -\alpha y^2 \mid x,y \in \cO_{\ell,}\}\cap \cO_{\ell'}^\times \asymp q^{\ell'}.$ 
	A cyclic matrix $A = \mat 0 {a^{-1}\alpha} a \beta \in M_2(\cO_m)$ is of Type (SNS) if and only if $\beta$ is non-invertible. Hence we obtain the following from Section~\ref{non-invertible-orbits-number-field}.
	\begin{eqnarray*}
		|\Sigma_\cO^\NS|&\asymp & \frac{|\cO_{\ell'}^\times |}{q^{\ell'}}\times \sum_{\beta \in \pi \cO_{\ell'}/\pi^\ee\cO_{\ell'}} | \cO_{\ell'}/ (\pi^{\ell' -\ee}\beta \cO_{\ell'})|\\
		&\asymp&1 \times  \sum_{i=1}^{\ee}  (q^{\ell' -\ee+i} \times q)\\
		&=&q^{\ell'-\ee+1} \times \frac{  q^{\ee+1}-q}{q-1}.
	\end{eqnarray*}
	Hence 
	$|\Sigma_\cO^\NS|\asymp q^{\ell'}\asymp q^{\frac{r}{2}}.$	
\end{proof}
Recall $\Sigma_\cO^\NS$ denotes the set of equivalent classes of $\{\psi_{[A]} \in \widehat{K^\ell} \mid A\,\, \mathrm{ is}\,\, \mathrm{ of }\,\,\mathrm{Type}(\NS) \}$ under the conjugation action of $\SL_2(\cO_r)$. For $\cO \in \dvrtwoplus,$ let $\Sigma_\cO^\NS(\bol k,\bol s)$  be the set  of orbits in $\Sigma_\cO^\NS$ with the parameters $\bol k$ and $\bol s.$

\begin{proposition}	\label{prop:SNS-orbits-Char-2}
	For $\cO \in \dvrtwoplus$, $r\geq \mathbf R_{\cO}$ and given $\bol k$ and $\bol s ,$ 
		$$|\Sigma_\cO^\NS(\bol k,\bol s)|\asymp \begin{cases}
		q^{\frac{r}{2}}&  \mathrm{if}\,\,  1\leq \bol k <\frac{\ell'}{2},\\
		q^{\frac{3r}{4}-\bol k}&  \mathrm{if}\,\,  \frac{\ell'}{2}\leq \bol k <\ell',\\
		q^{\frac{r}{4}}&  \mathrm{if}\,\, \bol k =\ell'.
		\end{cases}$$ 

\end{proposition}
\begin{proof}

From \cite[Lemma 4.20, Lemma~5.3]{2017arXiv171009112H}, we have    
$|\Sigma_\cO^\NS(\bol k,\bol s)|\asymp q^{\lfloor\frac{\bol s}{2}\rfloor} \times \mathcal{B}(\bol k,\bol s),$ where $\mathcal{B}(\bol k,\bol s)$ is defined by the following: 
$$\mathcal{B}(\bol k,\bol s)= \begin{cases}
		2(q-1)q^{-1} D(\bol s)&  \mathrm{if}\,\,  1\leq \bol k <\frac{\ell'}{2},\\
		(q-1)q^{\lfloor\frac{\ell'}{2}\rfloor-\bol k-1}D(\bol s)&  \mathrm{if}\,\,  \frac{\ell'}{2}\leq \bol k <\ell',\\
		q^{-\lfloor\frac{\ell'}{2}\rfloor}D(\bol s) &  \mathrm{if}\,\, \bol k =\ell',
		\end{cases}
		$$
where $D(\bol s) $ is $(q-1)q^{\ell'-\lfloor \frac{\bol s}{2} \rfloor-1}$ for $\bol s < \bol k$ and is $q^{\ell'-\lfloor  \frac{\bol k}{2}\rfloor}$ for $\bol s \geq \bol k$. 
  Note that   $ D(\bol s)\asymp q^{\frac{r}{2}-\lfloor \frac{\bol s}{2}\rfloor} $ for $ \bol s < \bol k.$ Since $\bol s \leq \bol k+1$ (by definition of $\bol s$),   we must have  $ D(\bol s)\asymp q^{\frac{r}{2}-\lfloor \frac{\bol s}{2}\rfloor} $ even  for $\bol s \geq \bol k.$ Therefore 
$$\mathcal{B}(\bol k,\bol s)\asymp \begin{cases}
		q^{\frac{r}{2}-\lfloor \frac{\bol s}{2}\rfloor}&  \mathrm{if}\,\,  1\leq \bol k <\frac{\ell'}{2}, \\
		q^{\frac{3r}{4}-\bol k-\lfloor \frac{\bol s}{2}\rfloor}&  \mathrm{if}\,\,  \frac{\ell'}{2}\leq \bol k <\ell',\\
		q^{\frac{r}{4}-\lfloor \frac{\bol s}{2}\rfloor}&  \mathrm{if}\,\, \bol k =\ell'.
		\end{cases}$$ 
The result now follows from the fact that $|\Sigma_\cO^\NS(\bol k,\bol s)|\asymp q^{\lfloor\frac{\bol s}{2}\rfloor} \times \mathcal{B}(\bol k,\bol s).$
\end{proof}

\section{Characterization of elements of $\mathbb E_{\tilde{A}}$: Results} 
\label{sec:characterization-of-Ea-results}

In this section, we state results regarding characterization of $\mathbb E_{\tilde{A}}.$ These results are crucially used in Sections~\ref{sec:rep-growth} and \ref{sec:odd-level-reps}.
We have included proofs of a few results, that were easy to prove, in this section itself. The rest of the proofs are in Sections~\ref{sec:proof of thm:condition for extension}, \ref{sec:proof of thm:simplification of extension conditions} and
\ref{sec:proof of thm:quotient_abelian}. Throughout this section, we fix $A = \mat 0{a^{-1} \alpha}a{\beta} \in M_2(\cO_{\ell'})$ and its lift $\tilde{A} = \mat 0{\tilde{a}^{-1}\tilde \alpha}{\tilde{a}}{\tilde \beta} \in M_2(\cO_r).$ For $\lambda \in \cO_r$, the matrix $\mat{1}{\tilde{a}^{-1} \lambda}{0}{1}\in M_2(\cO_r)$ is denoted by by $e_\lambda.$

\begin{lemma} 
	\label{centralizer-operations} 
	For $x, y, \lambda \in \cO_{ r},$ 
	\begin{enumerate}
		\item $e_{\lambda} \tilde A e_\lambda^{-1} =\tilde A + \lambda I +\mat{0}{\tilde{a}^{-1} \lambda( \tilde \beta - \lambda ) }{0}{-2 \lambda}.$

		\item For invertible $xI + y\tilde A ,$
 \begin{enumerate}
\item $[ e_\lambda , xI + y\tilde A] =  I +\frac{\lambda y}{\det(x I + y \tilde A)} \mat {x+ \lambda y } {\tilde{a}^{-1}(x  (\tilde \beta-\lambda) -\tilde \alpha y )}{\tilde{a} y}{-x }.$
\item $[ e_\lambda , xI + y\tilde A] =  I \,\, \mathrm{mod}\,(\pi^\ell)$ if and only if $\lambda y =0  \,\, \mathrm{mod}\,(\pi^\ell).$
\end{enumerate}
		\item For $ T = \mat {t_{11}}{t_{12}}{t_{21}}{t_{22}} \in M_2(\cO_r),$ $$
		[e_{\lambda}, I + \pi^\ell T] = I + (\tilde{a}^{-1} \lambda )\pi^\ell  \mat {t_{21}}{ t_{22}-t_{11} -\tilde{a}^{-1} \lambda t_{21}}{0}{-t_{21}} .$$ 
		
	\end{enumerate}
\end{lemma} 
\begin{proof} We prove only (2) here. Rest of the proof follows by straightforward computations.
	For invertible $(xI + y\tilde{A})$,  $[e_\lambda, (xI + y\tilde{A})]$ equals 
	\[
	\frac{1}{\det(x I + y \tilde A)} \mat {\det(x I + y \tilde A) + \lambda y(x+ \lambda y) } {\tilde{a}^{-1}(x y \lambda(\tilde \beta-\lambda) -\lambda \tilde \alpha y^2 )}{\tilde{a}\lambda y^2}{\det(x I + y \tilde A) - \lambda xy } .
	\]
	This gives proof of  (2)(a). Observe that by (2)(a),   $\lambda y =0  \,\, \mathrm{mod}\,(\pi^\ell)$ implies $[ e_\lambda , xI + y\tilde A] =  I \,\, \mathrm{mod}\,(\pi^\ell).$
For invertible $xI + y\tilde A ,$ note that either $x\in \cO_r^\times$ or $\tilde{a}y\in \cO_r^\times.$ Therefore from  (2)(a), we obtain that $[ e_\lambda , xI + y\tilde A] =  I \,\, \mathrm{mod}\,(\pi^\ell)$ implies   $\lambda y =0  \,\, \mathrm{mod}\,(\pi^\ell).$ This gives proof of (2)(b).
\end{proof}

Recall that 
$\mathbb E_{\tilde{A}} = \{e_\lambda   \in \mathrm{H}_{\tilde{A}}^\ell \mid \psi_{[A]}\,\, \mathrm{extends}\,\, \mathrm{to}\,\, C_S^{\ell}(\tilde{A})\langle e_\lambda\rangle\}.$ 
 The set  
$\mathbb E_{\tilde{A}}$ is in bijective correspondence with the set $ E_{\tilde{A}} := \{\lambda \in \mathrm{h}_{\tilde{A}}^\ell \mid e_\lambda \in \mathbb E_{\tilde{A}} \},$ where $\mathrm{h}_{\tilde{A}}^\ell = \{ x \in \cO_{r } \mid 2x = 0  \,\, \mathrm{mod}\, (\pi^{\ell} ), \, \,  x(x+\tilde \beta) = 0 \,\, \mathrm{mod}\, (\pi^{\ell} ) \}.$

For $x, y, \lambda \in \cO_r,$ define
\begin{eqnarray*}
	f(\lambda, x, y) &=&xy \lambda (\tilde \beta- \lambda ) -\tilde \alpha \lambda y^2 + \lambda ( x^2-1) ,\\
	g(x, y) &=& x^2 + \tilde{\beta} xy - \tilde{\alpha} y^2.
\end{eqnarray*}
Note that $	f(\lambda, x, y)=\lambda ( g(x,y)-1) - \lambda^2  xy .$
We will keep these notations fixed throughout this article.

 For $\lambda \in \mathrm{h}_{\tilde{A}}^\ell $
 and fixed lifts $\tilde{\alpha}, \tilde{\beta} \in \cO_r$ of $\alpha, \beta \in \cO_{\ell'},$ we define the set  
\begin{equation} 
\label{eq: *} 
E_{\lambda, \tilde A} = \{ (x, y) \in \cO_r \times \cO_r \mid g(x,y) = 1 \,\, \mathrm{mod}\,(\pi^\ell) \,\, \mathrm{and} \,\, \lambda y \in \pi^\ell \cO_r  \}
\end{equation} 
and its subset $E_{\lambda, \tilde A}^\circ = \{ (x, y) \in E_{\lambda, \tilde{A}}  \mid  \psi(f(\lambda, x, y)) = 1 \}.  
$
Our first result, explores elements of $E_{ \tilde A}$ through various equivalent conditions.
\begin{thm}
	\label{thm:condition for extension} 
	For $\lambda \in \mathrm{h}_{\tilde{A}}^\ell,$ the following are equivalent. 
	\begin{enumerate}
		\item $E_{\lambda, \tilde A} = E_{\lambda, \tilde A}^\circ .$
		\item Any element of $K^\ell$ which is of the form $[e_\lambda, X]$ for $X \in C_{S}^\ell(\tilde{A})$ is in the kernel of $\psi_{[A]}.$ 
		\item $[\langle e_\lambda\rangle, C_{S}^\ell(\tilde{A}) ] \cap K^\ell$ is contained in the kernel of $\psi_{[A]}.$ 
		\item  $ \lambda \in E_{ \tilde A},$ i.e. there exists an extension of $\psi_{[A]} $ to $C_{S}^\ell(\tilde{A})\langle e_\lambda\rangle.$
	\end{enumerate}
\end{thm}
 We prove this result in Section~\ref{sec:proof of thm:condition for extension}.
The following corollary about $E_{\tilde{A}}$ will be useful.
\begin{corollary}
\label{cor:E_tilde A remark}
	The following hold for $\mathrm{h}_{\tilde{A}}^\ell $ and $E_{\tilde{A}}.$
	\begin{enumerate}
		\item 	 $\mathrm{h}_{\tilde{A}}^\ell \cap  \pi^{\ell'} \cO_r \subseteq E_{\tilde{A}} \subseteq \mathrm{h}_{\tilde{A}}^\ell .$ 
		\item If $\lambda \in E_{\tilde{A}},$ then $\lambda+\pi^\ell \cO_r \subseteq E_{\tilde{A}} .$ In particular, $\pi^\ell \cO_r \subseteq E_{\tilde{A}} .$
		\item For $z\in \pi^{\ell'} \cO_r $ and $\lambda \in  \mathrm{h}_{\tilde{A}}^\ell \setminus \pi^\ell \cO_r , $ we have $f(\lambda +  z, x, y ) = f(\lambda , x, y)$ for all $(x,y)\in E_{\lambda, \tilde{A}}.$ 
	\end{enumerate}
\end{corollary}

\begin{proof}
Let  $\lambda \in \mathrm{h}_{\tilde{A}}^\ell \cap  \pi^{\ell'} \cO_r .$ For $(x,y) \in E_{\lambda, \tilde{A}} ,$ we have $g(x,y) = 1 \,\, \mathrm{mod}\,(\pi^\ell)$ and $ \lambda y \in \pi^\ell \cO_r $.  Therefore
$$f(\lambda,x,y)	= \lambda ( g(x,y)-1) - \lambda^2xy = 0-\lambda(\lambda y)x=0.$$
This implies $(x,y) \in E_{\lambda, \tilde{A}}^\circ.$ Therefore 
$E_{\lambda, \tilde{A}} \subseteq E_{\lambda, \tilde{A}} ^\circ.$ By definition, we have $E_{\lambda, \tilde{A}}^\circ \subseteq E_{\lambda, \tilde{A}} .$ Hence $E_{\lambda, \tilde{A}} = E_{\lambda, \tilde{A}} ^\circ.$ Now (1) follows by Theorem~\ref{thm:condition for extension}. For (2), we observe that if  $\lambda =\lambda' \,\, \mathrm{mod}\,(\pi^\ell)$ then the group generated by 
	$ C_S^{\ell}(\tilde{A})$ and $ \langle e_\lambda\rangle$ is equal to the one generated by $C_S^{\ell}(\tilde{A})$ and $\langle e_{\lambda'}\rangle.$ Therefore 
	$\lambda+\pi^\ell \cO_r \subseteq E_{\tilde{A}} $ whenever $\lambda \in E_{\tilde{A}}.$ For (3),  we note that for $z \in  \pi^{\ell'} \cO_r $ and $\lambda \in  \mathrm{h}_{\tilde{A}}^\ell\setminus \pi^\ell \cO_r , $ we have
	\begin{eqnarray*}
		f(\lambda + z,x,y)	&=& (\lambda + z)  \left( g(x,y)-1 \right) - (\lambda + z)^2xy\\
			&=&f(\lambda ,x,y)+z \left( g(x,y)-1 \right)-(2\lambda z +z^2)xy.
	\end{eqnarray*}
	Note that for $(x,y) \in E_{\lambda, \tilde{A}} ,$ we must have $z \left( g(x,y)-1 \right)=0$ and $2\lambda z xy =2 (\lambda y)zx=0.$ For $\lambda \in  \mathrm{h}_{\tilde{A}}^\ell\setminus \pi^\ell \cO_r ,$ the fact $\lambda y \in \pi^\ell \cO_r $ implies $y\in \pi \cO_r .$ Therefore $z^2 xy =0.$ This gives $	f(\lambda + z,x,y)=	f(\lambda ,x,y).$
\end{proof}

 Let $\boldsymbol{\psi}:\mathbb F_q \rightarrow \mathbb C^\times$ be an additive one-dimensional representation given by
 $\boldsymbol{\psi}(\bar x)=\psi(\pi^{r-1} x),$  for all $x\in \cO_r$ such that $x=\bar x \,\, \mathrm{mod}\,(\pi).$ The following lemma is useful for us.
\begin{lemma}\label{existance-of-zeta}
	There exists a unique $\xi \in \mathbb F_q^\times$ such that $\mathrm{ker}(\boldsymbol{\psi})=\{ \xi x^2 +x \mid x \in \mathbb F_q\}.$
\end{lemma} 
\begin{proof}
For $q=2,$ the lemma follows because $\mathbb F_q^\times=\{1\} $ and  $\mathrm{ker}(\boldsymbol{\psi})=\{0\}=\{  x^2 +x \mid x \in \mathbb F_q\}.$
Let $q = 2^n$ be such that $n\geq 2.$ 
For $\eta \in \mathbb F_q^\times,$ the map $f_\eta : \mathbb F_q\rightarrow \mathbb F_q;$  $x \mapsto \eta x^2 +x$ is linear (over $\mathbb F_2$) and is of rank $n-1.$ Therefore the image set
	$V_\eta:=\{ \eta x^2 +x \mid x \in \mathbb F_q\}$ forms an $(n-1)$-dimensional subspace of $\mathbb F_q.$ Also, we note that $\mathrm{ker}(\boldsymbol{\psi})$ is an index two subgroup of  $\mathbb F_q$ and therefore it is an $(n-1)$-dimensional subspace of $\mathbb F_q .$ There are exactly $q-1\,$ $\mathbb F_2$-subspaces of $\mathbb F_q$ of dimension $n-1.$
	Since $ |\mathbb F_q^\times|=q-1,$ it is enough to prove $V_\eta\neq V_{\eta'}$ whenever $\eta\neq \eta'.$ 

Let $\eta, \eta' \in \mathbb F_q^\times$ such that  $\eta\neq \eta'.$ 
	Observe that $$V_\eta=\{ \eta x^2 +x \mid x \in \mathbb F_q\}=\{\eta^{-1}(( \eta x)^2 +(\eta x)) \mid x \in \mathbb F_q\}=\eta^{-1}V_1.$$
Similarly $V_{\eta'}=(\eta')^{-1}V_1.$
Hence to show $V_\eta\neq V_{\eta'},$ it is enough to show that $\eta' \eta^{-1} V_1\neq V_1.$   Suppose $\eta' \eta^{-1} V_1= V_1.$ Then $(\eta' \eta^{-1} )^m V_1= V_1$ for $m\geq 1.$ Therefore for every $x \in V_1,$ we obtain
$Hx\subseteq V_1, $ where $H$ is the subgroup of $\mathbb F_q^\times$ generated by $\eta' \eta^{-1}.$ Observe that 
$$V_1=\{0\} \sqcup (\cup_{x\in V_1 \setminus \{0\}}Hx).$$
Since $\{Hx \mid x \in V_1\}$ is a collection of cosets of $H$ in $\mathbb F_q^\times,$ we obtain that $|H|$ is a divisor of $|V_1|-1=2^{n-1}-1.$  
 $H$ is a subgroup of $\mathbb F_q^\times$ implies $|H|$ is a divisor of $|\mathbb F_q^\times|=2^{n}-1.$ Observe that 
$1= (2^{n}-1)-2\times(2^{n-1}-1).$
Therefore $|H|$ must divide $1$ and hence $|H|=1.$  It is not possible because $ \eta' \eta^{-1} \in H$ and $ \eta' \eta^{-1} \neq 1.$ Therefore we must have $\eta' \eta^{-1} V_1\neq V_1.$
\end{proof}
\begin{remark}
	({\bf Definition of $\xi$}) Now onwards, for  fixed $\psi : \cO_r \rightarrow \mathbb C^\times$ and $\boldsymbol{\psi}:\mathbb F_q \rightarrow \mathbb C^\times$ as above, we fix the notation $\xi$ for the unique element in $\mathbb F_q^\times$ such that
 $\mathrm{ker}(\boldsymbol{\psi})=\{ \xi x^2 +x \mid x \in \mathbb F_q\}.$
\end{remark}
Let $\cO \in \dvrtwozero. $
Let $\ee$ be the ramification index of $\cO.$ The characterization of  $E_{\tilde{A}}$ for $r \geq 4\ee$ is given by the following result. 
\begin{thm}\label{S_A-in-number}
Let $\cO \in  \dvrtwozero$. The following hold for $E_{\tilde{A}}$. 
	\begin{enumerate}
		\item  For $ r > 2\ee$ with $\beta \in \cO_{\ell'}^\times,$ 	$E_{\tilde{A}} =\pi^{\ell} \cO_r .$
		\item For $r = 4\ee$ with $\beta=\pi v \in \pi \cO_{\ell'},$ $E_{\tilde{A}}$ is either $ \{0, \pi^{\ell-1} ( \tilde \xi w^{2})^{-1}\}+\pi^\ell \cO_r $ or $\pi^\ell\cO_r ,$
		where $\tilde \xi \in \cO_r^\times$is a lift of $\xi$  and 
		$w  \in \cO_r^\times$ such that $2=\pi^{\ee} w.$ In particular for $\ee=1,$ we have   
		\[
		E_{\tilde{A}} = \begin{cases} \{ 0,  \pi (\tilde \xi w^{2})^{-1} \}+ \pi^\ell\cO_r &  \mathrm{for}\,\, \alpha =(\frac{\tilde v}{w})^2+ (\frac{1}{\tilde \xi w^3})^2 \,\, \mathrm{mod}\,(\pi),\\
		\pi^{\ell}\cO_r &  \mathrm{for}\,\,\alpha \neq(\frac{\tilde v}{w})^2+ (\frac{1}{\tilde \xi w^3})^2 \,\, \mathrm{mod}\,(\pi),\\
		\end{cases}
		\]
		
		where  $ \tilde v \in \cO_r$ is a lift of  $v.$ 
		
		\item  For $ r > 4\ee$ with $\beta\in \pi \cO_{\ell'},$	$E_{\tilde{A}} =\pi^{\ell'}\cO_r .$

	\end{enumerate}
\end{thm}
A proof of this result is included in Section~\ref{sec:proof of thm:condition for extension}.
The following corollary of this result will be used in
Section~\ref{sec:proof of thm:quotient_abelian}.

\begin{corollary}\label{cor: S_A-in-number} 
	For $r=2\ell\geq 4\ee,$ $E_{\tilde{A}} \subseteq \{0,  \lambda \}+  \pi^{\ell}\cO_{r}$ for some $\lambda \in \mathrm{h}_{\tilde{A}}^{\ell}.$
\end{corollary}
\begin{proof}
	This result  directly follows from Theorem~\ref{S_A-in-number}.
\end{proof}
We now proceed to state the results analogous to the above for $\cO \in \dvrtwoplus.$ From now onwards till  Theorem~\ref{thm:simplification of extension conditions}, we assume that $\cO \in \dvrtwoplus$.
To characterize the elements of $E_{\tilde{A}}$ for $\cO \in \dvrtwoplus$, we first fix a few notations.  These notations will be used throughout  the article for $\cO \in \dvrtwoplus$. 

For $v \in \cO_r,$  and $0\leq i \leq  r-1,$ we denote $(v)_i\ $ for the unique elements in $\mathbb F_q $ such that $v= (v)_0 + (v)_1 \pi +\cdots + 
(v)_{r-1} \pi^{r-1}.$ Note that  $v \in \cO_r^2$
if and only if $(v)_{2i+1} = 0$ for all $i$.  

\begin{definition} ({\bf Definition of $\psi$ for $\cO \in \dvrtwoplus$})
Let $\boldsymbol{\psi}:\mathbb F_q \rightarrow   \mathbb C^\times$ be a fixed additive one-dimensional representation such that $\boldsymbol{\psi}(1)\neq 1.$  Define $\psi: \cO_r  \rightarrow   \mathbb C^\times$ by  $\psi(v)=\boldsymbol{\psi}((v)_{r-1})$ for  $v\in \cO_r.$
\end{definition}
\begin{remark}
Note that for any additive one-dimensional representation $\psi':\cO_r  \rightarrow   \mathbb C^\times$  such that $\psi'(\pi^{r-1})\neq 1,$ there exist a unique $u \in \cO_r^\times $ such that $\psi'(v)=\psi(uv),$ for all $v\in \cO_r.$
\end{remark}

Recall that for $A = \mat 0 {a^{-1} \alpha} a \beta \in M_2(\cO_{\ell'}),$ $\bol s$ and $\bol k$ are defined as follows.   
\begin{eqnarray}
\bol  k & = & \val(\beta). \nonumber \\
\bol s & = &\begin{cases}
2 \lfloor \bol k/2 \rfloor + 1 &  \mathrm{if}\,\,   \alpha = v^2 \,\,\mathrm{mod}\,(\pi^ \bol k ), \\ m &  \mathrm{if}\,\, \alpha = v_1^2 + \pi^m v_2^2 \,\,\mathrm{mod}\,(\pi^ \bol k ) \,\,\mathrm{for}\,\,\mathrm{odd}\,\, m < \bol k \,\,\mathrm{and}\,\, v_2 \in \cO_{\ell'}^\times. 
\end{cases}\nonumber
\end{eqnarray}
We remark that for any lift $\tilde{ \beta}\in\cO_r$ of $\beta,$ we have $\bol k=  \min \{ \val({\tilde{\beta}}), \ell' \}.$ The following examples illustrate these notions. 

\begin{example}\label{example for k and s. No.1}
Assume $r=9.$ Then $\ell'=\lfloor \frac{9}{2}\rfloor=4.$ Let $\alpha=1+\pi+\pi^2$  and $\beta = \pi^2 $ be in $\cO_4.$ 
Then $\bol k=  \val(\beta)= 2.$ Note that  $\alpha = 1^2+\pi \times 1^2 \,\,\mathrm{mod}\,(\pi^2 ).$ Therefore $\bol s =1.$
\end{example}
\begin{example}\label{example for k and s. No.2}
Assume $r=8.$ Then $\ell'=\lfloor \frac{8}{2}\rfloor=4.$ Let $\alpha=1+\pi^2$  and $\beta = 0 $ be in $\cO_4.$ 
Then $\bol k= \val(\beta)= \val(0)=4.$ Note that  $\alpha = (1+\pi )^2 \,\,\mathrm{mod}\,(\pi^4 ).$ Therefore $\bol s = 2 \lfloor 4/2 \rfloor + 1 =5.$
\end{example}
\begin{lemma} 
	\label{defintion-of-s-and-w1} 
	For any lift $\tilde{\alpha}$ of $\alpha,$ there exists  $w_1, w_2 \in \cO_{r}$ such that $\tilde \alpha = w_1^2 + \pi^{\bol s} w_2^2 \in \cO_r.$ Further, if $\bol s < \bol k$ then $w_2 \in \cO_{r}^\times$. 
\end{lemma}
\begin{proof}
	The first part  follows   from the definition of  $\bol s$ and the fact that
	$v \in \cO_r^2$ if and only if $(v)_{2i+1} = 0$ for all $i.$
	For $\bol s <\bol k,$ the element $\alpha \in \cO_{\ell'}$ is itself of the form $v^2 + \pi^{\bol s} u^2$ for some $v \in \cO_{\ell'}$ and $u \in \cO_{\ell'}^\times.$ Therefore for  any lift $\tilde\alpha= w_1^2 + \pi^{\bol s} w_2^2 ,$ we have $ w_2=u \, \, \mathrm{mod}\, (\pi).$ Hence  $w_2 \in \cO_{r}^\times.$	
\end{proof}
\begin{remark}\label{proper-tilde-alpha-and-tilde-beta}
	({\bf Definition of $\tilde{\alpha},$ $\tilde{\beta},$ $w_1$ and $w_2.$}) Now onwards, we fix lifts $\tilde{\alpha}$ and $\tilde{\beta}$ of $\alpha$ and $\beta$ respectively. We also fix $w_1, w_2 \in \cO_r$  such that  $\tilde{\alpha}= w_1^2 + \pi^{\bol s} w_2^2 \in \cO_r.$ For $A = \mat 0{a^{-1} \alpha}a{\beta}$, we fix $\tilde{A}  = \mat 0{\tilde{a}^{-1}\tilde \alpha}{\tilde{a}}{\tilde \beta} $, where $\tilde \alpha$ and  $\tilde \beta$ correspond to the above fixed choices. 
\end{remark}
For $\lambda \in \cO_r$ and $\tilde{A}  = \mat 0{\tilde{a}^{-1}\tilde \alpha}{\tilde{a}}{\tilde \beta} , $ we define $\bol i_{\lambda, \tilde{A}}$, $\bol j_{\lambda, \tilde{A}}$ and $\delta_{\lambda, \tilde{A}}$.  
\begin{eqnarray}
\bol i_{\lambda, \tilde{A}}& = &\val({\lambda}), \nonumber \\
\bol j_{\lambda, \tilde{A}} & = & \min \{ \val({\lambda+ \tilde{\beta}}), \ell' \}, \nonumber \\
 \delta_{\lambda, \tilde{A}}&=& \bol j_\lambda - \bol s - \mathrm{max} \{ \ell - \bol i_\lambda, \ell - \bol k, \lceil (\ell - \bol s) /2 \rceil \}. \nonumber
\end{eqnarray}
We remark that $\tilde{A}  = \mat 0{\tilde{a}^{-1}\tilde \alpha}{\tilde{a}}{\tilde \beta} $ is fixed throughout, therefore we leave  $\tilde{A}$ notation from  $\bol i_{\lambda, \tilde{A}}, \bol j_{\lambda, \tilde{A}}, \delta_{\lambda, \tilde{A}}$ and write these as $\bol i_\lambda$, $\bol j_{\lambda}$, $\delta_\lambda$ respectively. We will use these notations throughout the article. 
\begin{example}
Choose $\alpha=1+\pi+\pi^2,$ $ \beta=\pi^2$ and $\tilde{\beta}=\pi^2+ \pi^4$ as in Example~\ref{example for k and s. No.1}. We have $\ell'=4,$ $\ell=5,$ $\bol k=2$ and $\bol s =1.$ 
For $\lambda=\pi^3\in \cO_9,$ we have $\bol i _\lambda=\val(\pi^3)=3,$ $\bol  j_\lambda  =  \min \{ \val({\pi^3+\pi^2+ \pi^4}), 4 \}=\min \{ 2, 4 \}=2.$ Therefore $\mathrm{max} \{ \ell - \bol i_\lambda, \ell - \bol k, \lceil (\ell - \bol s) /2 \rceil \}=\mathrm{max} \{ 5-3, 5-2, \lceil (5 - 1) /2 \rceil \}=3.$ Hence $\delta_\lambda= \bol j_\lambda - \bol s - \mathrm{max} \{ \ell - \bol i_\lambda, \ell - \bol k, \lceil (\ell - \bol s) /2 \rceil \}=2-1-3=-2.  $

For $\lambda=\pi^2+ \pi^4+\pi^6 \in \cO_9,$  $\bol i _\lambda=\val(\pi^2+ \pi^4+\pi^6)=2.$ Note that $\lambda+\tilde{ \beta}=\pi^6.$ Therefore  $\bol  j_\lambda  =\min \{ \val({\pi^6}), 4 \}=  \min \{ 6, 4 \}=4.$ So we have $\mathrm{max} \{ \ell - \bol i_\lambda, \ell - \bol k, \lceil (\ell - \bol s) /2 \rceil \}=\mathrm{max} \{ 3, 3, 2 \}=3.$ Hence $\delta_\lambda=4-1-3=0.  $
\end{example}

\begin{lemma}\label{i-j-k-s-lem}
	The following are true for $\lambda \in \cO_r$ and $\tilde{A} \in M_2(\cO_r)$.   
	\begin{enumerate}
		\item $\bol s \leq \bol k+1 .$
		\item If $ \bol i_\lambda \neq \bol k,$ then $\bol j_\lambda=  \mathrm{min}\{ \bol i_\lambda, \bol k\}.$
		\item If $\bol i_\lambda \leq \ell'$ and  $\bol j_\lambda=\ell',$ then $\bol k = \bol i_\lambda.$
		\item $\bol i_\lambda \geq \mathrm{min} \{ \bol j_\lambda, \bol k\},$ $\bol j_\lambda \geq \mathrm{min} \{ \bol i_\lambda, \bol k\}$ and $\bol k \geq \mathrm{min} \{ \bol i_\lambda, \bol j_\lambda\}.$
	\end{enumerate}
\end{lemma}
\begin{proof}
	The result follows directly from the definition of $\bol i_\lambda,$ $ \bol j_\lambda,$ $ \bol k$ and $ \bol s.$ 
\end{proof}

Define $\epsilon$ as follows:
\[
\epsilon =  \begin{cases}  1 &  \mathrm{if} \,\, r = \mathrm{even}, \\ 0 &  \mathrm{if} \,\, r = \mathrm{odd}. \end{cases}
\]
In other words $ r-1=2\ell' -\epsilon.$  The  next three results are quite useful and their proofs are given in Section~\ref{sec:proof of thm:simplification of extension conditions}. 

\begin{proposition}\label{prop:possible-valuations-of-lambda}
Let $\lambda \in  \mathrm{h}_{\tilde{A}}^\ell \setminus \pi^{\ell'}\cO_r $ such that $2\bol j_\lambda + \bol i_\lambda = 2 \ell' +\bol  s - \epsilon .$ 
	\begin{enumerate}
		\item If  $ 2\ell' +\bol  s -\epsilon< 3\bol k$ then we must have $\bol i_\lambda = \bol j_\lambda =\frac{2\ell' +\bol  s -\epsilon}{3}.$ 
		\item If $ 2\ell' +\bol  s -\epsilon \geq  3\bol k$ then $\bol i_\lambda \in \{ \bol k,\, 2\ell' + \bol s -\epsilon-2\bol k  \}$ and  $\bol j_\lambda =\begin{cases}
		\frac{2\ell' +\bol  s -\epsilon- \bol k}{2}&    \mathrm{for}\,\, \bol i_\lambda= \bol k,\\
		\bol k &  \mathrm{for}\,\,  \bol i_\lambda \neq \bol k.
		\end{cases} $
		\item   $\bol i_\lambda \geq \mathrm{min} \{\bol  k,	\frac{2\ell' +\bol  s -\epsilon}{3} \} .$  
		\item If $ 2\bol k-\bol s <  \ell$ then we must have $\bol i_\lambda =\bol k.$
		\item $ \mathrm{max} \{ \ell-\bol i_\lambda, \ell-\bol k, \lceil (\ell-\bol s) /2 \rceil \}=\begin{cases}
		\lceil (\ell-\bol s) /2 \rceil & \mathrm{for}\,\, 2 \bol k -\bol s \geq \ell,\\
		\ell-\bol k&  \mathrm{for}\,\, 2 \bol k -\bol s < \ell.
		\end{cases}$
		\item  If $\bol s < \bol k,$ then $	\bol j_\lambda -\bol s \geq   \mathrm{max} \{ \ell-\bol i_\lambda, \ell-\bol k, \lceil (\ell-\bol s) /2 \rceil \},$ i.e. $\delta\geq 0.$

	\end{enumerate}

\end{proposition}

The following corollary of this result will be used in
Section~\ref{sec:proof of thm:quotient_abelian}.
\begin{corollary}
\label{cor:trivial action}
Assume $2\bol k -\bol s \geq \ell.$
Let $\lambda \in  \mathrm{h}_{\tilde{A}}^\ell \setminus \pi^{\ell'}\cO_r $ such that $2\bol j_\lambda + \bol i_\lambda = 2 \ell' +\bol  s - \epsilon .$
For $x I + y \tilde A \in C_{\GL_2(\cO_r)}(\tilde{A})$ such that 
$\det(x I + y \tilde A) = 1 \,\, \mathrm{mod}\, (\pi^\ell ),$ we have 
$\lambda y = 0 \,\, \mathrm{mod}\, (\pi^\ell ).$ 
\end{corollary}

 Let $\cO_{r}^2 =\{ x^2 \mid x \in \cO_r\}.$  Note that from Corollary~\ref{cor:E_tilde A remark}, we have $\mathrm{h}_{\tilde{A}}^\ell \cap  \pi^{\ell'}\cO_r \subseteq E_{\tilde{A}}.$ 
For $\lambda \in  \mathrm{h}_{\tilde{A}}^\ell \setminus \pi^{\ell'}\cO_r ,$ the following  result discusses the conditions on $\lambda$ such that $ \lambda \in E_{\tilde{A}}.$
\begin{thm} 
	\label{thm:simplification of extension conditions} 
	
	For $\lambda \in  \mathrm{h}_{\tilde{A}}^\ell \setminus \pi^{\ell'}\cO_r ,$ $\lambda \in E_{\tilde{A}} $ if and only if the following conditions holds: 
	\begin{enumerate}[(I)]
		\item $ \lambda \in \pi^{\ell - \ell'}\cO_{r}^2 \,\, \mathrm{mod}\,  (\pi^{\ell'}).$
		\item $2\bol j_\lambda + \bol i_\lambda = 2 \ell' +\bol  s - \epsilon .$ 
		\item For  $ \bol j_\lambda < \ell',$ $\bol s < \bol k$ and $\delta_\lambda \geq 0,$
		we must have 
		$
		\xi u_1^2 u_2^2 =  u_1 w_2^2 \,\, \mathrm{mod}\, (\pi^{2 \delta_\lambda +1}), 
		$ 
		where $u_1,u_2 \in \cO_{r}^\times$ such that $\lambda = \pi^{\bol i_\lambda} u_1$
		and $ \lambda + \tilde{\beta} = \pi^{\bol j_\lambda} u_2.$
	\end{enumerate}
\end{thm} 
The following result gives very important information regarding $H_{\tilde A}(\chi)$. This is useful in the section regarding the representation growth of $\SL_2(\cO)$. 
\begin{proposition} 
	\label{prop:main-theorem}
	Let $\chi$ be an extension of $\psi_{[A]}$ to  $C_S^\ell(\tilde{A}).$ Recall that  
	$H_{\tilde{A}}(\chi) = \{ g \in  C_S^\ell(\tilde{A}) \mathrm{H}_{\tilde{A}}^{\ell} \mid \chi^g = \chi  \}. $
	For  $r \geq \mathbf R_\cO,$ 
\[
	|H_{\tilde{A}}(\chi)|/|C_S^\ell(\tilde{A})| \leq |\cO/\wp|^3.
	\]
	
\end{proposition} 
\begin{proof}
	For $\chi \in \mathcal{E}_{\tilde{A}},$ 
	let $M_\chi=\mathrm{H}_{\tilde{A}}^{\ell}\cap H_{\tilde A}(\chi)$ and $m_\chi=\{\lambda \in  \mathrm{h}_{\tilde{A}}^{\ell} \mid e_\lambda \in M_\chi \}.$
	Then $m_\chi$ is an abelian subgroup of $\mathrm{h}_{\tilde{A}}^{\ell}$ such that $\pi^\ell\cO_r  \subseteq m_\chi.$ This implies
	$$\frac{ |H_{\tilde{A}}(\chi)|}{|C_S^\ell(\tilde{A})|}=\frac{|m_\chi|}{|\pi^{\ell} \cO_r |}.$$
	
	For $\lambda \in m_\chi,$ $\chi$ extends to $C_S^\ell(\tilde{A})\langle e_\lambda\rangle.$ This combined with the definition of $E_{\tilde{A}} $ gives 
	$m_\chi\subseteq  E_{\tilde A} .$

	For  $\cO \in \dvrtwozero,$ we have $\mathbf R_{\cO}=4\ee.$ From Theorem~\ref{S_A-in-number}, observe that $E_{\tilde A} \subseteq \pi^{\ell-1}\cO_r .$ Thus $ m_\chi \subseteq \pi^{\ell-1}\cO_r $
	and hence
	$|m_\chi|/|\pi^\ell\cO_r |\leq | \pi^{\ell-1}\cO_r |/|\pi^\ell\cO_r |=q.$ Hence  the result follows in this case.

	For $\cO \in \dvrtwoplus,$ let us consider the following set of valuations: 
	\[
	\mathrm{Val}( E_{\tilde A}\setminus \pi^{\ell'}\cO_r ) = \{\bol i_\lambda  \mid \lambda \in  E_{\tilde A}\setminus \pi^{\ell'}\cO_r  \}. 
	\]
	From Theorem~\ref{thm:simplification of extension conditions} and Proposition~\ref{prop:possible-valuations-of-lambda} (1)-(2),  
	we obtain $|\mathrm{Val}( E_{\tilde A}\setminus \pi^{\ell'}\cO_r )| \leq 2.$
	Let $\mathrm{Val}( E_{\tilde A}\setminus \pi^{\ell'}\cO_r )\subseteq\{v_1,v_2\}$ for some $v_1<v_2<\ell'.$  Then we must have 
	\begin{equation}\label{M-eqn}
		E_{\tilde A} \cup \pi^{\ell'}\cO_r \subseteq \pi^{v_1}\cO_r^\times \cup \pi^{v_2}\cO_r^\times \cup \pi^{\ell'}\cO_r\, .
	\end{equation}
	
	Let  $m^{\prime}_\chi :=m_\chi \cap  \pi^{\ell'}\cO_r $ and $m^{\prime \prime}_\chi :=m_\chi \cap  \pi^{v_2}\cO_r .$ Then $\pi^\ell\cO_r \subseteq m^{\prime}_\chi\subseteq m^{\prime  \prime}_\chi \subseteq m_\chi\subseteq  E_{\tilde A}.$
	Note that  $[m_\chi^{\prime} :\pi^{\ell} \cO_r ]\leq[\pi^{\ell'}\cO_r  :\pi^{\ell} \cO_r ]\leq q.$
	We use that $m^{\prime  \prime}_\chi$ is a group and (\ref{M-eqn}) to observe that 
	if any two elements of $m^{\prime  \prime}_\chi$  are equal modulo $\pi^{v_2 +1}\cO_r $ then these must be equal modulo $\pi^{\ell'}\cO_r .$ Therefore $[m_\chi^{\prime \prime} : m_\chi^{\prime} ] \leq q.$ 
	By similar argument we  get  $[m_\chi : m_\chi^{\prime \prime} ] \leq q.$ Therefore 
	$$\frac{|m_\chi|}{|\pi^{\ell} \cO_r |}=[m_\chi : m_\chi^{\prime \prime}]\times [m_\chi^{\prime \prime} : m_\chi^{\prime} ]\times [m_\chi^{\prime} :\pi^{\ell} \cO_r ] \leq q^3. $$
\end{proof} 
Our next result plays a crucial role in Section~\ref{sec:odd-level-reps}. 

\begin{thm} \label{thm:quotient_abelian}
	Let $r =2\ell $, $r\geq \mathbf R_{\cO}$ and $A \in  M_2(\cO_{\ell})$ be cyclic. 
	\begin{enumerate}
		\item The set $E_{\tilde A}$ forms an additive subgroup of $\mathrm{h}_{\tilde{A}}^{\ell}$
		such that  $[E_{\tilde A}: \pi^{\ell}\cO_{2\ell}]\leq 4.$
		\item  Let $\mathbb E_{\tilde A}$ be the subgroup of $\mathrm{H}_{\tilde{A}}^{\ell}$ corresponding to $E_{\tilde A}.$ Then  there exists an extension of $\psi_{[A]} $ to $C_{\SL_2(\cO_{2\ell})} (\psi_{A})\mathbb E_{\tilde A}.$
		\item Either the group $C_{\SL_2(\cO_{2\ell})} (\psi_{A})\mathbb E_{\tilde A}/K^\ell$ is abelian or it is a semi-direct product of an abelian group with a group of order two. 
		\item If the group $C_{\SL_2(\cO_{2\ell})} (\psi_{A})\mathbb E_{\tilde A}/K^\ell$ is non-abelian, then $[E_{\tilde A}: \pi^{\ell}\cO_{2\ell}]= 2.$
	\end{enumerate}
\end{thm}
See Section~\ref{sec:proof of thm:quotient_abelian} for its proof.

\section{Representation growth of $\SL_2(\cO)$}  
\label{sec:rep-growth}
In this section, we results regarding the representation growth of $\SL_2(\cO)$. In particular, we prove Theorems~ \ref{thm:abs-of-convergence} and \ref{thm: Na-and-dimension-asymptomatics}. 
\subsection{Bounds on the numbers and dimensions of representations of $\SL_2(\cO_r)$ }
\label{sec:proof-of-main-thms}
In this section, we include a proof of Theorem~\ref{thm: Na-and-dimension-asymptomatics}. As mentioned before, this result plays an essential role in proving Theorem~\ref{thm:abs-of-convergence}. 

\begin{lemma}\label{lem:number of constituents}
	Let $\phi$ be a representation of a finite group $G$ such that $\phi \cong \phi_1^{m_1}\oplus \phi_2^{m_2} \oplus \cdots \oplus \phi_t^{m_t},$ where $\phi_i\,$ for $1\leq i\leq t$ are inequivalent irreducible representations of $G$ and $m_1 \leq m_2 \leq \cdots \leq m_t$. 
	Let $d=\min\{ \dim (\phi_i)\mid 1\leq i\leq t\}$ and $d'=\max\{ \dim (\phi_i)\mid 1\leq i\leq t\}.$ 
	Then the following inequality holds: 
	$$ \frac{\dim(\phi)}{m_t d'} \leq t \leq \frac{\dim(\phi)}{m_1 d}.$$
\end{lemma} 
\begin{proof}
	By definition of $d$ and $d'$, we obtain 
	$m_1t d \leq \dim(\phi)$ and $  m_tt d'\geq \dim(\phi).$
\end{proof}

\begin{proof}[\bf{Proof of Theorem~\ref{thm: Na-and-dimension-asymptomatics}}] 
	Let $A \in M_2(\cO_{\ell'})$ be cyclic and $\rho \in \mathrm{Irr}(\SL_2(\cO_{r}) \mid \psi_{[A]} ).$ By Clifford Theory, there exists a representation $\phi \in \mathrm{Irr}(C_{\SL_2(\cO_{r})} (\psi_{[A]})\mid \psi_{[A]} )$ such that $\rho \cong \mathrm{Ind}_{C_{\SL_2(\cO_{r})} (\psi_{[A]}) }^{\SL_2(\cO_{r})}\phi.$ Recall that 
	\[
	K^\ell \subseteq  C_S^\ell(\tilde{A}) \subseteq C_S^\ell(\tilde{A})\mathrm{H}_{\tilde{A}}^{\ell} \subseteq C_{\SL_2(\cO_{r})} (\psi_{[A]})
	\]
	and every representation of  $\mathrm{Irr}(C_S^\ell(\tilde{A}) \mid \psi_{[A]} )$ has dimension one. Therefore 
	\begin{eqnarray}
		\label{bound-1} 
		 \dim(\rho) \leq  \frac{|\SL_2(\cO_r)|}{ |C_S^\ell(\tilde{A})|}.
		\end{eqnarray} 
		Let  $\chi \in \mathrm{Irr}(C_S^\ell(\tilde{A}) \mid \psi_{[A]} )$ be such that $\langle\mathrm{Ind}_{C_S^\ell(\tilde{A})}^{C_{\SL_2(\cO_{r})} (\psi_{[A]})}\chi , \, \phi\rangle\neq 0.$ By definition of $H_{\tilde A}(\chi)$, we have $C_S^\ell(\tilde{A}) \subseteq H_{\tilde A}(\chi) \subseteq C_S^\ell(\tilde{A})\mathrm{H}_{\tilde{A}}^{\ell}$ and any constituent of $\mathrm{Ind}_{C_S^\ell(\tilde{A})}^{C_S^\ell(\tilde{A})\mathrm{H}_{\tilde{A}}^{\ell}}\chi$ has dimension 
	greater than or equal to $\frac{|C_S^\ell(\tilde{A})\mathrm{H}_{\tilde{A}}^{\ell}|}{|H_{\tilde A}(\chi)|}.$
	Hence
	\begin{equation*}
	\dim(\rho) \geq \frac{|C_S^\ell(\tilde{A})\mathrm{H}_{\tilde{A}}^{\ell}|}{|H_{\tilde A}(\chi)|} \times  \frac{|\SL_2(\cO_r)|}{ |C_{\SL_2(\cO_{r})} (\psi_{[A]})|}.
	\end{equation*}
	We now observe that $|C_S^\ell(\tilde{A})\mathrm{H}_{\tilde{A}}^{\ell}|=\frac{|C_S^\ell(\tilde{A})|\times  |\mathrm{h}_{\tilde{A}}^{\ell}|}{|\pi^\ell\cO_r |}$ and 
	$ |C_{\SL_2(\cO_{r})} (\psi_{[A]})|= \frac{|C_{\SL_2(\cO_{r})} (\psi_{A})|\times  |\mathrm{h}_{\tilde{A}}^{\ell'}|}{|\pi^{\ell'}\cO_r |}
	.$
	Therefore 
\begin{eqnarray*} 
\dim(\rho) \geq \frac{|C_S^\ell(\tilde{A})|}{|H_{\tilde A}(\chi)|} \times  \frac{|\SL_2(\cO_r)|}{ |C_{\SL_2(\cO_{r})} (\psi_{A})|}\times\frac{ |\mathrm{h}_{\tilde{A}}^{\ell}|}{ |\mathrm{h}_{\tilde{A}}^{\ell'}|}\times\frac{|\pi^{\ell'}\cO_r |}{|\pi^{\ell}\cO_r |}. 
\end{eqnarray*} 
This, along with Proposition~\ref{prop:main-theorem} and Corollary~\ref{h_A/h_a^l cardinality}(1), gives
\begin{eqnarray}
	\label{bound-2}  
\dim(\rho) \geq \frac{1}{ q^4}\times  \frac{|\SL_2(\cO_r)|}{ |C_{\SL_2(\cO_{r})} (\psi_{A})|}.
\end{eqnarray} 
From (\ref{bound-1}), (\ref{bound-2}), $|C_{\SL_2(\cO_{r})} (\psi_{A})| \asymp |C_{\SL_2(\cO_{\ell'})}(A)| \times q^{3\ell}$ 
and $|(C_S^\ell(\tilde{A}))| \asymp |C_{\SL_2(\cO_{r})} (\psi_{A})|,$ we obtain (2) of our result.  
We now prove (1). Denote  $\mathrm{Ind}_{K^\ell}^{\SL_2(\cO_{r})}\psi_{[A]}$ by $\varrho_A$. Let $\varrho_A \cong \rho_1^{m_1} \oplus \rho_2^{m_2} \oplus \cdots \rho_k^{m_k}$
such that $\rho_i \ncong \rho_j$ for $i \neq j$ and $m_1 \leq m_2 \leq \cdots m_k$. Note that $\mathrm{Irr}(\SL_2(\cO_{r}) \mid \psi_{[A]} )=\{\rho_i \mid 1\leq i \leq k \}.$
Let  $d=\min\{ \dim (\rho_i)\mid 1\leq i \leq k \}$ and  $d'=\max\{ \dim (\rho_i)\mid 1\leq i \leq k \}.$ 
By Lemma~\ref{lem:number of constituents}, 
	\begin{equation}\label{eqn:dim-num}
	\frac{\dim(\varrho_A)}{ m_k d'}	\leq |\mathrm{Irr}(\SL_2(\cO_{r}) \mid \psi_{[A]} )|\leq \frac{\dim (\varrho_A)}{ m_1 d}.
	\end{equation}
By Clifford Theory, every $\rho_i \in \mathrm{Irr}(\SL_2(\cO_{r}) \mid \psi_{[A]} )$ satisfies 
$$\langle\mathrm{Ind}_{K^\ell}^{\SL_2(\cO_{r})} \psi_{[A]}, 
\, \rho_i\rangle=\langle\psi_{[A]}, \rho_i|_{K^{\ell}}\rangle=\frac{\dim(\rho_i)}{|\mathrm{Orb}(\psi_{[A]})|},$$
where $\mathrm{Orb}(\psi_{[A]})$ is the orbit of $\psi_{[A]}$ under $\SL_2(\cO_{ r})$ action. For $1\leq i \leq k,$ we obtain 
	$$ \frac{d}{|\mathrm{Orb}(\psi_{[A]})|}
\leq
 \langle\mathrm{Ind}_{K^\ell}^{\SL_2(\cO_{r})} \psi_{[A]}, 
	\, \rho_i\rangle
\leq \frac{d'}{|\mathrm{Orb}(\psi_{[A]})|}.$$
Hence 
$ \frac{d}{|\mathrm{Orb}(\psi_{[A]})|}
\leq m_1 \leq m_k 
\leq \frac{d'}{|\mathrm{Orb}(\psi_{[A]})|}.$
This, along with (\ref{eqn:dim-num}), implies
	\begin{equation}\label{eqn:dim-num-1}
\frac{\dim (\varrho_A)}{d'^2\times |\mathrm{Orb}(\psi_{[A]})|}	\leq 	|\mathrm{Irr}(\SL_2(\cO_{r}) \mid \psi_{[A]} )| \leq \frac{\dim (\varrho_A)}{d^2\times |\mathrm{Orb}(\psi_{[A]})|}	 .
 \end{equation}
 We note that
  	$$|\mathrm{Orb}(\psi_{[A]})|=\frac{|\SL_2(\cO_{r})|}{|C_{\SL_2(\cO_{r})} (\psi_{[A]})|}=\frac{|\SL_2(\cO_{r})|\times |\pi^{\ell'}\cO_r | }{|C_{\SL_2(\cO_{r})} (\psi_{A})|\times |\mathrm{H}_{\tilde{A}}^{\ell'}|}$$
and $\dim (\varrho_A)=\frac{|\SL_2(\cO_{r})|}{|K^\ell|}$. From $(\ref{bound-1})$ and $(\ref{bound-2}),$  we have 
 $d' \leq \frac{|\SL_2(\cO_r)|}{ |C_S^\ell(\tilde{A})|}$ and $	\frac{1}{ q^4}\times  \frac{|\SL_2(\cO_r)|}{ |C_{\SL_2(\cO_{r})} (\psi_{A})|}\leq d.$  
  By substituting these 
   in (\ref{eqn:dim-num-1}), we obtain 
	\begin{eqnarray}
		\label{bound-3} 
		 \frac{|(C_S^\ell(\tilde{A}))|^2\times q^\ell}
		{|C_{\SL_2(\cO_{r})} (\psi_{A})|\times |\mathrm{H}_{\tilde{A}}^{\ell'}|\times |K^\ell|}\leq |\mathrm{Irr}(\SL_2(\cO_{r}) \mid \psi_{[A]} )| 
\leq  \frac{ |C_{\SL_2(\cO_{r})} (\psi_{A})|\times q^{\ell+8}}
		{|\mathrm{H}_{\tilde{A}}^{\ell'}|\times |K^\ell|}.
	\end{eqnarray}
Note that  $|(C_S^\ell(\tilde{A}))| \asymp |C_{\SL_2(\cO_{r})} (\psi_{A})|.$
 Therefore $|\mathrm{Irr}(\SL_2(\cO_{r}) \mid \psi_{[A]} )| \asymp \frac{ |C_{\SL_2(\cO_{r})} (\psi_{A})|\times q^{\ell}}
		{|\mathrm{H}_{\tilde{A}}^{\ell'}|\times |K^\ell|}.$ The result now follows from $|C_{\SL_2(\cO_{r})} (\psi_{A})| \asymp |C_{\SL_2(\cO_{\ell'})}(A)| \times q^{3\ell}$ and $|K^\ell|\asymp q^{3\ell}.$  
\end{proof}

\subsection{The abscissa of convergence of $\zeta_{\SL_2(\cO)}(s)$} 
\label{sec:abs-of-conv}
In this section, we prove Theorem~\ref{thm:abs-of-convergence}. 
Recall that the representation zeta function $\zeta_{\SL_2(\cO)}(s)$ of $\SL_2(\cO)$ is defined as  
\[
\zeta_{\SL_2(\cO)}(s) = \sum_{\rho \in \mathrm{Irr}(\SL_2(\cO)) } \frac{1}{\dim(\rho)^s}.
\]
Let $\mathrm{Irr}^\mathfrak{pr}(\SL_2(\cO_r))$ be the set of primitive (i.e. of level $r-1$) irreducible representations of $\SL_2(\cO_r)$ for $r \geq 2$.  Let $\zeta_{S_r}(s) =  \sum_{\rho \in \mathrm{Irr}^\mathfrak{pr}(\SL_2(\cO_r)) } \frac{1}{\dim(\rho)^s}.$ Then
\[ 
\zeta_{\SL_2(\cO)}(s) = \zeta_{\SL_2(\cO_1)}(s)+\sum_{r=2}^\infty \zeta_{S_r}(s). 
\]
The description of $ \zeta_{\SL_2(\cO_1)}(s)$ is already known in the literature, see~Adams~\cite{adams}, Zelevinsii-Narkunskaja~\cite{MR0369497}. Throughout this section, we will assume that $r \geq 2.$ 
Recall that a matrix $A=\mat{0 }{a^{-1}\alpha}{a}{\beta} \in M_2(\cO_{j})$ for $j \geq 1$ is said to be of  $\mathrm{Type}(\SS)$, $\mathrm{Type}(\IR)$,  $\mathrm{Type}(\NS)$ respectively if and only if 
$\bar{A}$ is split semisimple$(\SS)$, Irreducible$(\IR)$, Split non-semisimple$(\NS)$ respectively. 
We use this to partition the set $\mathrm{Irr}^\mathfrak{pr}(\SL_2(\cO_r))$ into three subsets as follows. Define $\zeta_{S_r}^\SS(s),$ $\zeta_{S_r}^\IR(s)$ and $\zeta_{S_r}^\NS(s)$ to be the corresponding zeta function of primitive irreducible representations of $\SL_2(\cO_r)$ of each type. Any cyclic  $\bar{A} \in M_2(\mathbb F_q)$ is either $\SS$, $\IR$, or $\NS$. Hence 
\[ 
\zeta_{\SL_2(\cO)}(s) = \zeta_{\SL_2(\cO_1)}(s) + \sum_{r=2}^\infty (\zeta_{S_r}^\SS(s) + \zeta_{S_r}^\IR(s)  + \zeta_{S_r}^\NS(s) ). 
\]
Recall that $\Sigma_\cO^\SS(\Sigma_\cO^\IR$ and $\Sigma_\cO^\NS)$ denote the  orbits of  $\{\psi_{[A]} \in \widehat{K^\ell} \}$ for  $A$  of Type($\SS$) (Type($\IR$), Type($\NS$)) under the conjugation action of $\SL_2(\cO_r)$.  The asymptotic behaviour of these sets' cardinalities were obtained in Theorem~\ref{thm:number-of-orbits-split semisimple} and Proposition~\ref{prop:SNS-orbits-Char-2}. These results are included under the heading ``$\#$orbits $\asymp$" in Table~\ref{table:number-dimension}.  For cyclic $A \in M_2(\cO_{\ell'})$,  Theorem~\ref{thm: Na-and-dimension-asymptomatics}, Theorem~\ref{thm:C-S-L asymp to..} and  Corollary~\ref{h_A/h_a^l cardinality} describe asymptotically both $n_A =  | \mathrm{Irr}(\SL_2(\cO_r) \mid \psi_{[A]}) |$ and $d_A = \dim(\rho)$ for $\rho \in \mathrm{Irr}(\SL_2(\cO_r) \mid \psi_{[A]})$. These results are  listed under the heading ``dimension $\asymp$" and ``$n_A \asymp$" in Table~\ref{table:number-dimension}.
\begin{table}	
	\begin{tabular}{|c|c|c|c|c|c|}
		\hline
		$\Char(\cO) $&Type &dimension $\asymp$ &$n_A \, \asymp $ & $ \#$Orbits $\asymp$&$n_A \times \#$Orbits $\asymp$\\
		\hline
		$0,2$&$(\SS)$ and $(\IR)$ &$q^r$&$q^{\frac{r}{2}}$&$q^{\frac{r}{2}}$&$q^r$\\
		\hline
		$0$&$(\NS)$&$q^r$&$q^{\frac{r}{2}}$&$q^{\frac{r}{2}}$&$q^r$\\
		\hline
		$2$&$(\NS)_{\bol{k}, \bol{s}}$&$q^{r-\lfloor\frac{\bol s}{2}\rfloor}$&$\frac{q^{r+\lfloor\frac{\bol s}{2}\rfloor}}{|\mathrm{h}_{\tilde{A}}^{\ell'}|}$&$\Sigma_\cO^\NS(\bol k,\bol s)$&$q^{r-\bol k +\lfloor\frac{\bol s}{2}\rfloor}$\\
		\hline
	\end{tabular}
	\caption{Dimensions and  numbers of the representations of each Type}
	\label{table:number-dimension}
\end{table}
From Table~\ref{table:number-dimension}, for $r \geq \mathbf{R}_\cO$, we have 
\begin{equation}
	\label{zeta-1 and 2}
	\zeta^\SS_{S_r}(s)+\zeta^\IR_{S_r}(s) \asymp  q^{r(1-s)}.
\end{equation}
The sum $\sum_{r=\mathbf{R}_\cO }^{\infty}q^{r(1-s)}$ diverges for all $s\leq 1.$ So we assume that $s >1.$
For $\cO \in \dvrtwozero,$ we further observe that $\zeta^\NS_{S_r}(s) \asymp  q^{r(1-s)}.$ This, along with $(\ref{zeta-1 and 2}),$ gives
$$(\zeta^{\SS}_{S_r}(s)+\zeta^{\IR}_{S_r}(s)+\zeta^{\NS}_{S_r}(s))\asymp 2q^{r(1 - s)}\asymp q^{r(1 - s)}.$$
Hence the abscissa of convergence of $\zeta_{\SL_2(\cO)}(s)$ for $\cO \in \dvrtwozero$ is  $1.$ Now onwards, we assume that $\cO \in \dvrtwoplus.$
Denote $\lfloor\frac{s}{2}\rfloor$ by $\delta.$ Again from Table~\ref{table:number-dimension},
\begin{eqnarray*}
	\zeta^\NS_{S_r}(s) &\asymp & \sum_{k=1}^{\ell'} \sum_{\delta=0}^{\lfloor \frac{k}{2} \rfloor} q^{r-k+ \delta -s (r-\delta)}=q^{r-sr} \sum_{k=1}^{\ell'} \sum_{\delta=0}^{\lfloor \frac{k}{2} \rfloor} q^{ \delta(1 +s)-k }.
\end{eqnarray*}	
Note that for $s> 1$, we have 
\begin{equation*}
	\sum_{\delta=0}^{\lfloor \frac{k}{2} \rfloor}  q^{\delta (1+s)}=\frac{q^{(1+s)(\lfloor \frac{k}{2} \rfloor+1)}-1}{q^{1+s} -1}\asymp q^{(1+s)(\lfloor \frac{k}{2} \rfloor+1)}-1\asymp q^{\frac{(1+s)k}{2}}.
\end{equation*}
Hence
\begin{eqnarray*}
	\zeta^\NS_{S_r}(s)	&\asymp&q^{r-s r}\sum_{k=1}^{\ell'}  
	q^{\frac{(s-1)k}{2} }.
\end{eqnarray*}
For $s >1,$ 
\begin{eqnarray*}
	\sum_{k=1}^{\ell'}  q^{\frac{(s-1)k}{2} }&=&
	\frac{q^{\frac{(s-1)(\ell'+1)}{2} }-q^{\frac{s-1}{2} }}{q^{ \frac{s-1}{2}}-1}\asymp q^{\frac{(s-1)(\ell'+1)}{2} } \asymp q^{\frac{(s-1)r}{4} }.
\end{eqnarray*}
Therefore    
$	\zeta^\NS_{S_r}(s)	\asymp
q^{\frac{3r}{4}(1-s )}
$ for large $r.$
Combining this with  $(\ref{zeta-1 and 2}),$ we obtain
$$(\zeta^\SS_{S_r}(s)+\zeta^\IR_{S_r}(s)+\zeta^\NS_{S_r}(s))\asymp q^{r(1 - s)}+q^{\frac{3r}{4}(1- s )}.$$
Observe that $\sum_{r=\mathbf{R}_\cO}^{\infty}( q^{r(1 - s)}+q^{\frac{3r}{4}(1- s )})$ converges for $s>1$ and diverges for $s<1.$
Hence the abscissa of convergence of $\zeta_{\SL_2(\cO)}(s)$ for $\cO \in \dvrtwoplus$ is  $1.$ This completes the proof of Theorem~\ref{thm:abs-of-convergence}.  
\section{Odd level representations of $\SL_2(\cO)$}
\label{sec:odd-level-reps}
In this section, we give a construction of all odd level representations of $\SL_2(\cO).$ Using this, we prove Theorem~\ref{thm:zeta poly not eaqal- gp algebra}. Throughout this section,  we assume $r=2\ell.$ 
\subsection{Construction of odd level representations of $\SL_2(\cO)$}
\label{sec:construction of odd level rep}
In this section, we prove Theorem~\ref{main_theorem-2}. The following  lemma is  elementary and  used in this section as well as in Sections~\ref{sec:proof of thm:condition for extension} and \ref{sec:proof of thm:quotient_abelian}.
\begin{lemma} (Diamond Lemma)
	\label{diamond-lemma} 
	Let $H$ and $K$ be two subgroups of a group $G$ such that $H$ is normal in $G$ and $G = HK.$ Let $\chi_1$ and $\chi_2$ be one-dimensional representations of $H$ and $K$ respectively such that $\chi_1(khk^{-1}) = \chi_1(h)$ for all $h \in H$ and $k \in K,$ and $\chi_1|_{H \cap K } = \chi_2|_{H\cap K}.$ Then there exists a unique one-dimensional representation of $G$ extending both $\chi_1$ and $\chi_2$ simultaneously.  
\end{lemma}
\begin{proof} For proof, see  Lemma~\cite[Lemma~5.4]{MR2684153}. 
\end{proof}
\begin{lemma}\label{lem:induction-irred}
	Let $\mathbb M_A = C_{\SL_2(\cO_{2\ell})} (\psi_{A})\mathbb E_{\tilde A} $.  The representation $\mathrm{Ind}_{\mathbb M_A}^{ C_{\SL_2(\cO_{2\ell})}(\psi_{[A]})}(\rho)$ is irreducible for every $\rho \in\mathrm{Irr}(\mathbb M_A \mid  \psi_{[A]}).$ 
\end{lemma}
\begin{proof}
	The group $C_{\SL_2(\cO_{2\ell})}(\psi_{A})$ is a normal subgroup of  $C_{\SL_2(\cO_{2\ell})}(\psi_{[A]})$ such that the corresponding  quotient group is abelian. Hence $\mathbb M_A$ is a normal subgroup of $C_{\SL_2(\cO_{2\ell})}(\psi_{[A]}).$ 
		Let $\rho \in\mathrm{Irr}(\mathbb M_A\mid  \psi_{[A]}) $ such that $\dim(\rho) = 1$.  From Lemma~\ref{lem:stabilizer-form},  $C_{\SL_2(\cO_{2\ell})}(\psi_{[A]})= C_{\SL_2(\cO_{2\ell})} (\psi_{A}) \mathrm{H}_{\tilde{A}}^{\ell},$ where $\mathrm{H}_{\tilde{A}}^{\ell}$ is an abelian group.  Therefore the stabilizer of $\rho$ in $C_{\SL_2(\cO_{2\ell})}(\psi_{[A]})$ is  $C_{\SL_2(\cO_{2\ell})} (\psi_{A}) N$ for some abelian subgroup $N$ of $ \mathrm{H}_{\tilde{A}}^{\ell}$ and  $\rho$ ( therefore $\psi_{[A]}$) extends to $C_{\SL_2(\cO_{2\ell})} (\psi_{A}) N.$ This combined with the definition of $\mathbb E_{\tilde A}$ implies  
	$N=\mathbb E_{\tilde A}.$ By
	Clifford theory, the induced representation $\mathrm{Ind}_{\mathbb M_A}^{ C_{\SL_2(\cO_{2\ell})}(\psi_{[A]})}(\rho)$ is irreducible. 
	
	Next, let $\rho \in\mathrm{Irr}(\mathbb M_A \mid  \psi_{[A]})$ such that $\dim(\rho) > 1$. This case arises only if the group $\mathbb M_A/K^\ell$ is non-abelian. From Theorem~\ref{thm:quotient_abelian} (4) we have $[E_{\tilde A}: \pi^{\ell}\cO_{2\ell}]= 2.$
	Let $\{0,\lambda_0\}$, with $\lambda_0 \in \mathrm{h}_{\tilde{A}}^\ell \setminus \pi^{\ell}\cO_{2\ell}$, be a set of representatives of $E_{\tilde A}/ \pi^{\ell}\cO_{2\ell}$. Let $\phi \in\mathrm{Irr}(C_{\SL_2(\cO_{2\ell})} (\psi_{A})\mid  \psi_{[A]})  $ be such that $\langle \phi , \rho\mid_{C_{\SL_2(\cO_{2\ell})} (\psi_{A})}\rangle \neq 0.$ The $\dim(\rho)>1$ and $|
	\mathbb M_A/ C_{\SL_2(\cO_{2\ell})} (\psi_{A})|=2$ give $e_{\lambda_0}$ does not stabilize $\phi$ and $\rho \cong\mathrm{Ind}_{C_{\SL_2(\cO_{2\ell})} (\psi_{A})}^{\mathbb M_A}(\phi)$. 
	By the definition of $\mathbb E_{\tilde A},$ the stabilizer of $\phi$ in $C_{\SL_2(\cO_{2\ell})} (\psi_{[A]})$ is contained in $\mathbb M_A$. Hence $C_{C_{\SL_2(\cO_{2\ell})} (\psi_{[A]})} (\phi)= C_{\SL_2(\cO_{2\ell})} (\psi_{A}).$ 
	By Clifford theory, the induced representation
	\[
	\mathrm{Ind}_{C_{\SL_2(\cO_{2\ell})} (\psi_{A})}^{ C_{\SL_2(\cO_{2\ell})}(\psi_{[A]})}(\phi)\cong \mathrm{Ind}_{\mathbb M_A}^{ C_{\SL_2(\cO_{2\ell})}(\psi_{[A]})}(\rho)
	\]
	is irreducible. 
\end{proof}
\begin{proof}[\bf{Proof of Theorem~\ref{main_theorem-2}}]
	By Theorem~\ref{thm:quotient_abelian},the set $\mathbb M_A=C_{\SL_2(\cO_{2\ell})} (\psi_{A})\mathbb E_{\tilde A}$  satisfies (1) and (2)  of Theorem~\ref{main_theorem-2}. 
	Now we prove (3) of Theorem~\ref{main_theorem-2}.  By Lemma~\ref{lem:induction-irred} and Clifford theory, we have 
	$$ \mathrm{Irr}( \SL_2(\cO_{2\ell})\mid  \psi_{[A]})=\{ \mathrm{Ind}_{\mathbb M_A}^{ \SL_2(\cO_{2\ell})}(\rho) \mid \rho \in  \mathrm{Irr}(\mathbb M_A\mid  \psi_{[A]}) \}.$$
Hence to show (3) of Theorem~\ref{main_theorem-2}, it is enough to prove that any $\rho_1 ,\rho_2 \in  \mathrm{Irr}(\mathbb M_A\mid  \psi_{[A]})$ satisfy   $\mathrm{Ind}_{\mathbb M_A}^{ \SL_2(\cO_{2\ell})}(\rho_1)\cong \mathrm{Ind}_{\mathbb M_A}^{ \SL_2(\cO_{2\ell})}(\rho_2)$ if and only if $\rho_1=\rho_2^g$ for some $g \in C_{\SL_2(\cO_{2\ell})}(\psi_{[A]}).$ This follows because of the 
	Clifford theory and the fact that $\mathbb M_A$ is a normal subgroup of $C_{\SL_2(\cO_{2\ell})}(\psi_{[A]}).$ 	
\end{proof}
\subsection{Group algebra of $\SL_2(\cO_{2 \ell})$} 
\label{sec:group-algebras}
This section contains a proof of Theorem~\ref{thm:zeta poly not eaqal- gp algebra}. We first collect a few explicit results regarding the numbers and dimensions of irreducible representations of $\SL_2(\cO_{2 \ell})$. 
\begin{lemma}\label{lem:high dim implies split semisimple}
Let $A \in M_2(\cO_\ell)$ be a cyclic matrix. Each $\rho \in \mathrm{Irr}( \SL_2(\cO_{2 \ell}) \mid \psi_{[A]})$ satisfies  $\dim(\rho) \leq (q+1) q^{2 \ell-1}.$ Further $\dim(\rho) = (q+1) q^{2 \ell-1}$ implies $\bar{A}$ is split semisimple. 
\end{lemma}
\begin{proof} 
From Lemma~\ref{lem:centralizer-sl}, every $\rho \in \mathrm{Irr}(C_S^\ell(\tilde{A}) \mid \psi_{[A]} )$ has dimension one. Therefore
$\dim(\rho) \leq \frac{|\SL_2(\cO_{2\ell})|}{ |C_S^\ell(\tilde{A})|}$  for all $\rho \in \mathrm{Irr}(C_S^\ell(\tilde{A}) \mid \psi_{[A]} )$. 
Also, $C_S^\ell(\tilde A)= (C_{\GL_2(\cO_{r})} (\tilde{A}) M^{\ell})\cap \SL_2(\cO_{r})$ gives
\begin{equation}\label{C_S L cardinality}
|C_S^\ell(\tilde{A})|=\frac{|C_{\GL_2(\cO_{r})} (\tilde{A}) M^{\ell}|}{|\det(C_{\GL_2(\cO_{r})} (\tilde{A}) M^{\ell})|}
=\frac{
	|C_{\GL_2(\cO_{\ell})}(A)| \times q^{4\ell}}{|\det(C_{\GL_2(\cO_{\ell})} (A))| \times q^\ell} 
  = \frac{ |C_{\GL_2(\cO_{\ell})}(A)|}{|\det(C_{\GL_2(\cO_{\ell})}(A))|}\times q^{3\ell}.
\end{equation}
By using $|\SL_2(\cO_{2\ell})|=(q^2 -1)q^{6\ell-2}$ and 
$|\det(C_{\GL_2(\cO_{\ell})}(A))|\leq |\cO_{\ell}^\times|=(q-1)q^{\ell-1},$ we obtain 
$$\dim(\rho) \leq  \frac{|\SL_2(\cO_{2\ell})|\times |\det(C_{\GL_2(\cO_{\ell})}(A))|}{ |C_{\GL_2(\cO_{\ell})}(A)|\times q^{3\ell}}\leq  \frac{ (q^2 -1)(q-1)q^{4\ell-3}}{ |C_{\GL_2(\cO_{\ell})}(A)|}.$$
The result now follows from Lemma~\ref{gl-centralizer-cardinality}.
\end{proof}
\begin{proposition}\label{prop:high dim iff split semisimple}
Assume $\cO \in \dvrtwozero$. Let $\ell>\ee$ and $A \in M_2(\cO_\ell)$ be a cyclic matrix.  
\begin{enumerate} 
	\item Any $\rho \in \mathrm{Irr}( \SL_2(\cO_{2 \ell}) \mid \psi_{[A]})$ satisfies $\dim(\rho) = (q+1) q^{2 \ell-1}$ if and only if $\bar{A}$ is split-semisimple.
	\item The group $\SL_2(\cO_{2 \ell})$ has $\frac{(q-1)^2 q^{2 \ell -2 }}{2}$ many inequivalent primitive irreducible representations of dimension $(q+1) q^{2 \ell-1}.$ 
\end{enumerate} 
\end{proposition} 
\begin{proof}
For (1), by Lemma~\ref{lem:high dim implies split semisimple}, we need to only prove that the result holds for every $A \in M_2(\cO_\ell)$ with split semisimple $\bar{A}$. Without loss of generality, we can take $A= \mat0{a^{-1}\alpha}a\beta.$ 
Since $\bar{A}$ is split-semisimple, $\beta=\mathrm{trace}(A)$ is invertible. From Proposition~\ref{prop:charactrization of h_Atilda^ell and h_Atilda^ell'}(2), we have 
$\mathrm{h}_{\tilde{A}}^\ell =\pi^\ell\cO_{2\ell}.$ Hence $C_{\SL_2(\cO_{2\ell})} (\psi_{[A]})= C_{\SL_2(\cO_{2\ell})}(\psi_{A}).$  By 
Clifford theory, for $\rho \in \mathrm{Irr}( \SL_2(\cO_{2 \ell}) \mid \psi_{[A]})$ there exists 
$\phi \in   \mathrm{Irr}( C_{\SL_2(\cO_{2\ell})}(\psi_{A}) \mid \psi_{[A]})$ such that $\rho \cong \mathrm{Ind}_{C_{\SL_2(\cO_{2\ell})}(\psi_{A})}^{ \SL_2(\cO_{2\ell})}(\phi) .$ Since $\phi$ is one-dimensional, 
$\dim(\rho) = \frac{|\SL_2(\cO_{2\ell})|}{|C_{\SL_2(\cO_{2\ell})}(\psi_{A})|}= (q+1) q^{2 \ell-1}.$ The last equality follows from  Lemmas~\ref{gl-centralizer-cardinality} and \ref{lem:image-det-map}.
We now prove (2). From Theorem~\ref{thm:number-of-orbits-split semisimple}, the number of orbits of $\psi_{[A]}$ with split semisimple $\bar{A}$  is $ \frac{(q-1)(q^{\ell-1})}{2}.$ Hence
by (1),  it is enough to show that for $A \in M_2(\cO_\ell)$ with split-semisimple $\bar{A}$, $| \mathrm{Irr}( \SL_2(\cO_{2 \ell}) \mid \psi_{[A]})|=(q-1)q^{\ell-1}.$ Let $A= \mat0{a^{-1}\alpha}a\beta \in M_2(\cO_\ell)$ be such that $\bar{A}$ is split-semisimple. Note that  $C_{\SL_2(\cO_{2\ell})} (\psi_{[A]})=C_{\SL_2(\cO_{2\ell})}(\psi_{A}).$ The result now follows because 
$$| \mathrm{Irr}( \SL_2(\cO_{2 \ell}) \mid \psi_{[A]})|=| \mathrm{Irr}(C_{\SL_2(\cO_{2\ell})}(\psi_{A}) \mid \psi_{[A]})|=|C_{\SL_2(\cO_{\ell})}(A)|=(q-1)q^{\ell-1}.$$
\end{proof}
\begin{proposition}\label{prop:strict inequality}
Assume 
$\cO \in \dvrtwoplus.$ Let $\ell \geq2 $ 
\begin{enumerate} 
\item Let $A=\mat0{a^{-1}\alpha}a\beta \in  M_2(\cO_\ell)$ be such that $\beta =x^2 $ for some $x \in \cO_{\ell}^\times.$ Then there exists $ \chi \in \mathcal{E}_{\tilde{A}}$ such that  $ H_{\tilde{A}}(\chi)\neq C_{\SL_2(\cO_{2\ell})}(\psi_A).$
\item For $A = \mat 001{1 + \pi} \in M_2(\cO_\ell)$, every $\rho \in \mathrm{Irr}( \SL_2(\cO_{2 \ell}) \mid \psi_{[A]})$ has dimension  $(q+1) q^{2 \ell-1}.$ 
\end{enumerate} 
\end{proposition} 
\begin{proof}
From Proposition~\ref{prop:charactrization of h_Atilda^ell and h_Atilda^ell'}(1), we obtain that $ \mathrm{h}_{\tilde{A}}^\ell =\{0, \tilde \beta\}+\pi^\ell\cO_{2\ell}.$  By using Corollary~\ref{cor:E_tilde A remark} and Theorem~\ref{thm:simplification of extension conditions}, $E_{\tilde A}=\{0, \tilde \beta\}+\pi^\ell\cO_{2\ell}.$ Therefore by Theorem~\ref{main_theorem-2}(1),
there exists an extension $\phi$ of $\psi_{[A]}$ to $C_{\SL_2(\cO_{2\ell})}(\psi_A)\mathbb E_{\tilde A}.$ Now (1) follows because $\chi:=\phi|_{C_{\SL_2(\cO_{2\ell})}(\psi_A)} \in \mathcal{E}_{\tilde{A}}$ and $H_{\tilde{A}}(\chi)=C_{\SL_2(\cO_{2\ell})}(\psi_A)\mathbb E_{\tilde A}\neq C_{\SL_2(\cO_{2\ell})}(\psi_A).$

For (2), we have $\bar{A}=\mat 0011$ implies that  $\bar{A}$ is split semisimple. Let  $\rho \in \mathrm{Irr}( \SL_2(\cO_{2 \ell}) \mid \psi_{[A]}).$ We claim that $\rho \cong \mathrm{Ind}_{C_{\SL_2(\cO_{2\ell})} (\psi_{A})}^{ \SL_2(\cO_{2 \ell}) }(\phi)$ for some $\phi \in \mathrm{Irr}( C_{\SL_2(\cO_{2\ell})} (\psi_{A}) \mid \psi_{[A]}).$ 
The trace of matrix $A = \mat 001{1 + \pi}$ is $1+\pi \in \cO_\ell^\times.$ From Proposition~\ref{prop:charactrization of h_Atilda^ell and h_Atilda^ell'},   we have $\mathrm{h}_{\tilde{A}}^\ell = 	\{0,1+\pi\} +\pi^\ell\cO_{2\ell},$ 
where  $1+\pi \in \cO_{r}$ is a lift of $1+\pi \in \cO_{\ell}.$  By Theorem~\ref{main_theorem-2}(3),
it is enough to show that $C_{\SL_2(\cO_{2\ell})} (\psi_{A})\mathbb E_{\tilde A}=C_{\SL_2(\cO_{2\ell})} (\psi_{A}).$ 
Note that  $1+\pi +\pi^\ell z \notin \cO_{2\ell}^2 $ modulo $\pi^\ell\cO_{2\ell}$ for any $z \in \cO_{2\ell}.$ Hence by Theorem~\ref{thm:simplification of extension conditions} and Corollary~\ref{cor:E_tilde A remark}, we obtain $E_{\tilde A}=\pi^\ell\cO_{2\ell}.$ Hence $C_{\SL_2(\cO_{2\ell})} (\psi_{A})\mathbb E_{\tilde A}=C_{\SL_2(\cO_{2\ell})} (\psi_{A}).$ 
We now proceed to complete the proof. Every $\phi \in \mathrm{Irr}( C_{\SL_2(\cO_{2\ell})} (\psi_{A}) \mid \psi_{[A]})$ is one-dimensional. Hence  we have
$\dim(\rho)=\frac{ |\SL_2(\cO_{2 \ell}) |}{|C_{\SL_2(\cO_{2\ell})} (\psi_{A})|}.$  
From Lemmas~\ref{lem:image-det-map} and \ref{gl-centralizer-cardinality},  we have  $|\det(C_{\GL_2(\cO_{\ell})}(A))|=|\cO_{\ell}^\times|=(q-1)q^{\ell-1}$ and $|C_{\GL_2(\cO_{\ell})}(A)|= (q-1)^2 q^{2\ell -2}$ respectively.
Therefore by (\ref{C_S L cardinality}),  $|C_{\SL_2(\cO_{2\ell})} (\psi_{A})|=(q-1)q^{4\ell-1}.$ Hence 	$\dim(\rho)=(q+1)q^{2\ell-1}.$
\end{proof}
\begin{lemma}\label{lem:max dim from cyclic}
Any irreducible representation of $\SL_2(\cO_{2\ell-1})$ has dimension strictly less than $(q+1) q^{2 \ell -1}.$
\end{lemma} 
\begin{proof} 
From Onn~\cite[Theorem~1.4]{MR2456275}, any irreducible representation of $\GL_2(\cO_{2\ell-1})$ has dimension less than or equal to $(q+1) q^{2 \ell -2}.$ The lemma follows because  $\SL_2(\cO_{2\ell-1})$ is a subgroup of $\GL_2(\cO_{2\ell-1})$ and $(q+1) q^{2 \ell -2}<(q+1) q^{2 \ell -1}.$
\end{proof}
\begin{definition}\label{def:rep zeta poly}
For a finite group $G,$ the polynomial 
\[
\calp_G(X) := \sum_{\rho \in \mathrm{Irr}(G) } X^{\dim(\rho)}
\]
is called the {\it representation zeta polynomial} of $G.$
\end{definition}
\begin{proof}[\bf{Proof of Theorem~\ref{thm:zeta poly not eaqal- gp algebra}}] Consider the set of cyclic irreducible representations of $\SL_2(\cO_{2 \ell}),$ that is the one that lie above $\psi_{[A]}$ for a cyclic $A.$ Recall from Section~\ref{sec:basic-framework}, this is exactly the set of primitive irreducible representations, denoted by $\mathrm{Irr}^\pr(\SL_2(\cO_{2 \ell})),$ of $\SL_2(\cO_{2 \ell}).$ The primitive representation zeta polynomial corresponding to the cyclic representations of $\SL_2(\cO_{2 \ell})$ is defined as below.  
\[
\calp^\pr_{\SL_2(\cO_{2 \ell})}(X) = \sum_{\rho \in \mathrm{Irr}^\pr(\SL_2(\cO_{2 \ell})) } X^{\dim(\rho)}.
\]
We note that in case $A$ is not cyclic then $\psi_{[A]} $ is a trivial one-dimensional representation of $K^{2\ell-1}.$ Therefore $\SL_2(\cO_{2 \ell})$ has the following expression for the representation zeta polynomial. 
\[
\calp_{\SL_2(\cO_{2 \ell})}(X)   = \calp^\pr_{\SL_2(\cO_{2 \ell})}(X) + \calp_{\SL_2(\cO_{2\ell-1})} (X).
\]
From Lemma~\ref{lem:high dim implies split semisimple}, Proposition~\ref{prop:high dim iff split semisimple} and Proposition~\ref{prop:strict inequality}, both $\calp^\pr_{\SL_2(\cO_{2 \ell})}(X)$ and $\calp^\pr_{\SL_2(\cO'_{2 \ell}) }(X)$ are polynomials of degree $(q+1) q^{2 \ell-1}.$ Therefore by Lemma~\ref{lem:max dim from cyclic}, to prove $\calp_{\SL_2(\cO_{2 \ell})}(X) \neq \calp_{\SL_2(\cO'_{2 \ell}) }(X),$ it is enough to show that the   leading coefficients of  $\calp^\pr_{\SL_2(\cO_{2 \ell})}(X)$ and $\calp^\pr_{\SL_2(\cO'_{2 \ell}) }(X)$ are not equal.
By Proposition~\ref{prop:high dim iff split semisimple},  $\calp^\pr_{\SL_2(\cO_{2 \ell})}(X)$ is a polynomial of degree $(q+1) q^{2 \ell-1}$ with leading coefficient $\frac{(q-1)^2 q^{2 \ell -2}}{2}.$ We prove that the leading coefficient of corresponding 
$\calp^\pr_{\SL_2(\cO'_{2 \ell}) }(X)$ is strictly less than $\frac{(q-1)^2 q^{2 \ell -2}}{2}.$ For this, we focus our attention on representations $\rho \in \mathrm{Irr}(\SL_2(\cO'_{2 \ell}))$ such that 
$\rho $ lies above $\psi_{[A]}$ and $\bar{A}$ is split-semisimple. By Lemma~\ref{lem:high dim implies split semisimple}, these are the only representations that contribute, if at all, non-trivially to the leading coefficient of the zeta polynomial.  Let $A = \mat 0 {a^{-1} \alpha} a \beta  \in M_2(\cO'_{\ell})$ such that $\bar{A}$ is split semisimple, so in particular $\beta$ is invertible. By Lemma~\ref{lem:stabilizer-form}, we have  
\[
C_{\SL_2(\cO'_{2 \ell})} (\psi_{[A]}) = C_{\SL_2(\cO'_{2 \ell})} (\psi_{A})\{ \mathrm{I}, \mat 1{\tilde{a}^{-1} \tilde \beta}01  \}. 
\]
If $\rho \in \mathrm{Irr}(\SL_2(\cO'_{2 \ell}) \mid \psi_{[A]} )$
such that $\dim(\rho ) = (q+1)q^{2\ell-1},$ then $\rho \cong \mathrm{Ind}_{C_{\SL_2(\cO'_{2\ell})}(\psi_A)}^{\SL_2(\cO'_{2\ell})}\chi$  for some $\chi \in \mathcal{E}_{\tilde{A}}.$  Note that for such $\chi \in \mathcal{E}_{\tilde{A}},$ 
$\chi\neq \chi^{e_{\tilde{\beta}}}$ and 
 $ \mathrm{Ind}_{C_{\SL_2(\cO'_{2\ell})}(\psi_A)}^{\SL_2(\cO'_{2\ell})}\chi \cong \mathrm{Ind}_{C_{\SL_2(\cO'_{2\ell})}(\psi_A)}^{\SL_2(\cO'_{2\ell})}\chi^{e_{\tilde{\beta}}}.$ 
 Hence the number of representations in  $\mathrm{Irr}(\SL_2(\cO'_{2 \ell}) \mid \psi_{[A]} )$ with dimension $(q+1)q^{2\ell-1}$ is equal to $\mathrm{n}_A :=  \frac{|\{ \chi \in \mathcal{E}_{\tilde{A}} \,\,\mid\,\, H_{\tilde{A}}(\chi)=C_{\SL_2(\cO'_{2\ell})}(\psi_A) \}| }{2}.$ 
Hence the leading coefficient of $\calp_{\SL_2(\cO'_{2 \ell}) }(X)$ is equal to $\underset{ \psi_{[A]} \in \Sigma_{\cO'}^{ss}}{\sum} \mathrm{n}_A,$
where   $\Sigma_{\cO'}^{ss}$ is the set of orbits of $\psi_{[A]}$ with $\bar{A}$  split-semisimple. 
For every $\psi_{[A]} \in \Sigma_{\cO'}^{ss},$ we have
\[ 
\mathrm{n}_A \leq \frac{|\mathcal{E}_A|}{2} = \frac{|C_{\SL_2(\cO_{\ell})}(A)|}{2}= \frac{(q-1)q^{\ell-1}}{2}. 
\]
By Proposition~\ref{prop:strict inequality}, the inequality $\mathrm{n}_A \leq \frac{|\mathcal{E}_A|}{2}$ is strict for any $\psi_{[A]} \in \Sigma_{\cO'}^{ss}$ such that $\mathrm{trace}(A) = x^2$ for some $x  \in\cO_{\ell}^\times.$ The existence of such a split semisimple $A$ is easy to see, for example $A=\mat 0011$ will do. This, along with Theorem~\ref{thm:number-of-orbits-split semisimple},  implies that the leading coefficient of  $\calp_{\SL_2(\cO'_{2 \ell}) }(X)$ is strictly less than $\frac{(q-1)^2 q^{2 \ell -2}}{2}.$ This completes the proof of Theorem~\ref{thm:zeta poly not eaqal- gp algebra}.
\end{proof}

\section{Proof of Theorems~\ref{thm:condition for extension} and ~\ref{S_A-in-number}}
\label{sec:proof of thm:condition for extension}
	In this section, we prove Theorems~\ref{thm:condition for extension} and ~\ref{S_A-in-number}. 
\begin{proof}  [{\bf Proof of  Theorem~\ref{thm:condition for extension} }]
	 We first prove the equivalence of $(1)$ and $(2).$ For $x,y \in \cO_r,$ observe that  $g(x,y)=1 \,\,\mathrm{mod}\, (\pi^\ell)$  if and only if there exists $X \in  C_{S}^\ell(\tilde{A})$ such that $X = (xI + y \tilde A)Z$ for some $Z \in M^\ell .$   	
	Let $X  = (xI + y\tilde{A}) Z \in C_{S}^\ell (\tilde{A})$. For $\lambda \in \mathrm{h}_{\tilde{A}}^{\ell},$ by Lemma~\ref{centralizer-operations}(2), we obtain  that  the commutator $[e_\lambda, X] \in K^\ell $ if and only if $\lambda y \in \pi^\ell \cO_r .$ 
	Therefore to prove the equivalence, we need only to prove that the following holds for $\lambda \in \mathrm{h}_{\tilde{A}}^{\ell}$ and $[e_\lambda, X] \in K^\ell $.
	\[
		\psi_{[A]}([e_\lambda, (xI + y \tilde A)Z]) = \psi(f(\lambda, x,y)).
	\]
Since  $xI + y \tilde{A} \in C_{\GL_2(\cO_{r})} (\psi_A),$ we have 
	\begin{eqnarray*}
		\psi_{[A]}([e_\lambda, (xI + y \tilde{A})Z])  & = & \psi_A ([e_\lambda, (xI + y \tilde{A})Z]) \\
		& = &  \psi_A((xI + y \tilde{A}) [(xI + y \tilde{A})^{-1}, e_\lambda] [ e_\lambda ,Z] (xI + y \tilde{A})^{-1}) \\ &  = & \psi_A(  [(xI + y \tilde{A})^{-1}, e_\lambda ]) \psi_A ([ e_\lambda ,Z]).
	\end{eqnarray*}
	Now observe that
	\[
	\psi_A(  [(xI + y \tilde{A})^{-1}, e_\lambda ])=	\psi_A(  [ e_\lambda , (xI + y \tilde{A})])=	\psi(  \lambda y(\tilde{\alpha}y+ x (\tilde \beta -\lambda)-\tilde \alpha y -\tilde \beta x)  ) =\psi(-\lambda^2 xy),
	\]
	where the second last equality follows from
	Lemma~\ref{centralizer-operations} (2), $\lambda y=0 \,\,\mathrm{mod}\, (\pi^\ell)$, $g(x,y)=1 \,\,\mathrm{mod}\, (\pi^\ell)$ implies $\lambda y / g(x,y) = \lambda y$, and the definition of $\psi_A.$  Similarly, for $Z = I + \pi^\ell(t_{ij}),$ by  Lemma~\ref{centralizer-operations} (3), we obtain the following:
	\begin{eqnarray*}
		\psi_A ([ e_\lambda, Z])&=&	\psi( \pi^\ell \tilde{a}^{-1}  \lambda (  \tilde{a} t_{22}-  \tilde{a}t_{11} - \lambda t_{21} - \tilde{\beta} t_{21}) ) \\
		&=&  \psi( \pi^\ell \lambda ( t_{22}-t_{11}  )-\pi^\ell  \tilde{a}^{-1} t_{21}  \lambda (   \lambda+ \tilde{\beta} ))\\
		&=& \psi(\lambda(\det(Z)^{-1}-1+2 \pi^\ell t_{22} )-\pi^\ell  \tilde{a}^{-1} t_{21}  \lambda (   \lambda+ \tilde{\beta} )).
	\end{eqnarray*}
	Note that $\lambda \in \mathrm{h}_{\tilde{A}}^{\ell}$ implies $2 \lambda \pi^\ell =0$ and  $ \pi^\ell \lambda (   \lambda+ \tilde{\beta} ) =0.$ Hence $\psi_A ([ e_\lambda, Z])=\psi(\lambda(\det(Z)^{-1}-1)).$
	By combining these terms and the fact that $\det(Z)^{-1} = \det(xI + y\tilde A),$ we obtain
	\begin{eqnarray*}
		\psi_{[A]}([e_\lambda, (xI + yA)Z]) &=& \psi( -\lambda^2 xy +\lambda  (x^2+ \tilde \beta xy - \tilde \alpha  y^2  -1 )) \\ 
		&=& \psi(xy \lambda (\tilde \beta- \lambda ) -\tilde \alpha \lambda y^2 + \lambda ( x^2-1) )\\
		&=&\psi(f(\lambda, x,y)). 
	\end{eqnarray*} 
This gives that $(1)$ and $(2)$ are equivalent. 
		Next  we prove the equivalence of $(2)$ and $(3).$ By definition, $(3)$ implies $(2).$  We proceed to prove $(2)$ implies $(3).$ 
	If $(3)$ is not true, then there exists $X_1,$ $X_2,$ \ldots, $X_n$ elements of  $C_{S}^\ell(\tilde{A})$ and integers $c_1,$ $c_2,$ \ldots, $c_n$ such that $\prod_{i=1}^n [e_\lambda^{c_i}, X_i] $ is an element of $K^\ell$ but is not in the kernel of $\psi_{[A]}.$ Let $\widehat{\psi_{A}}$ be an extension of $\psi_{[A]}$ to the group $C_{G}^\ell(\tilde{A})$(exists due to Lemma~\ref{lem:centralizer-form}(4)).  By Lemma~\ref{centralizer-operations}(1) and the fact that $\lambda \in \mathrm{h}_{\tilde{A}}^\ell,$ we obtain $e_\lambda^{c_i} X_i e_\lambda^{-c_i} \in C_{G}^\ell(\tilde{A}).$ Also, $\lambda \in \mathrm{h}_{\tilde{A}}^\ell$ gives $e_\lambda^{2} \in C_{G}^\ell(\tilde{A}).$ Therefore 
	$$\widehat{\psi_{A}} (e_\lambda^{c_i} X_i e_{\lambda}^{-c_i} )=\begin{cases} \widehat{\psi_{A}} (e_\lambda X_i e_{\lambda}^{-1} ) &  \mathrm{if}\,\, c_i \,\, \mathrm{is}\,\, \mathrm{odd},\\ \widehat{\psi_{A}} ( X_i ) &  \mathrm{if}\,\, c_i \,\, \mathrm{is}\,\, \mathrm{even}.
	\end{cases}$$
	This gives 
	\begin{equation}\label{eqn-assumptin}
	1 \neq \widehat{\psi_A} (\prod_{i=1}^n [e_\lambda^{c_i}, X_i]) = \prod_{i=1}^n \widehat{\psi_{A}} (e_\lambda^{c_i} X_i e_{\lambda}^{-c_i} )) \widehat{\psi_{A}}(X_i^{-1}) = \widehat{\psi_A}([e_\lambda, \prod_{i: \, c_i \,\,\mathrm{is}\,\, \mathrm{odd}} X_i ]).
	\end{equation}
	Next, we prove that $[e_\lambda, \prod_{i: \, c_i \,\,\mathrm{is}\,\, \mathrm{odd}} X_i ] \in K^\ell.$ For this we  note that $C_{G}^\ell(\tilde{A})/M^\ell$ is abelian. Therefore $[e_\lambda, \prod_{i: \, c_i  \,\,\mathrm{is}\,\, \mathrm{odd}} X_i ] M^\ell = \prod_{i=1}^n [e_\lambda^{c_i}, X_i] M^\ell.$ By our assumption $\prod_{i=1}^n [e_\lambda^{c_i}, X_i]$ is an element of $K^\ell.$ 
Therefore $[e_\lambda, \prod_{i: \, c_i  \,\,\mathrm{is}\,\, \mathrm{odd}} X_i ]\in K^\ell,$ because it has determinant one.
By the hypothesis  $$\widehat{\psi_A}([e_\lambda, \prod_{i: \, c_i  \,\,\mathrm{is}\,\, \mathrm{odd}} X_i ]) = 1.$$ This contradicts (\ref{eqn-assumptin}).  
	Therefore $(2)$ and $(3)$ are equivalent. 	
	Last, we prove the equivalence of (3) and (4). The implication of (4) from (3) is clear. Therefore we mainly focus on (3) implies (4). For this, we consider the hierarchy of groups  as given in Figure~\ref{hierarchy-of-groups-1}.
	\begin{tiny}
		\begin{figure}
			\xymatrix{
				&C_S^\ell(\tilde{A}) \langle e_\lambda\rangle& & & [C_{S}^\ell(\tilde{A}), \langle e_\lambda\rangle] \ar[dl] & \\ 
				\langle e_\lambda\rangle \ar[ur]&&C_S^\ell(\tilde{A}) \ar[ul] & [C_{S}^\ell(\tilde{A}), \langle e_\lambda\rangle] K^\ell  \ar[l] & & [C_{S}^\ell(\tilde{A}), \langle e_\lambda\rangle] \cap K^\ell \ar[dl] \ar[ul] \\ 
				&\langle e_\lambda\rangle \cap C_S^\ell(\tilde{A}) \ar[ur] \ar[ul] & & & K^\ell \ar[ul] & \\ 
			}
			\caption{}\label{hierarchy-of-groups-1}
		\end{figure}
	\end{tiny} 
We note that the group $[C_{S}^\ell(\tilde{A}), \langle e_\lambda\rangle]$ stabilizes $\psi_{[A]}.$ Therefore by Lemma~\ref{diamond-lemma} and the hypothesis, there exists a one-dimensional representation $\chi$ of $[C_{S}^\ell(\tilde{A}), \langle e_\lambda\rangle] K^\ell$ such that restriction of $\chi$ to $[C_{S}^\ell(\tilde{A}), \langle e_\lambda\rangle]$ is trivial. By Lemma~\ref{lem:centralizer-sl}, the quotient group $C_S^\ell(\tilde{A})/ K^\ell$ is abelian and there exists an extension of $\psi_{[A]}$ to $C_S^\ell(\tilde{A}).$ Therefore there exists an extension $\tilde{\chi}$ of $\chi$ to  $C_S^\ell(\tilde{A})$. The one-dimensional representation $\tilde{\chi}$ is stable under the action of $\langle e_\lambda\rangle$ because $\tilde{\chi}|_{[C_{S}^\ell(\tilde{A}),\langle e_\lambda\rangle]} $ is trivial. Therefore (4)  follows by Lemma~\ref{diamond-lemma} and the fact that the group $\langle e_\lambda\rangle$ is abelian. This completes the proof of  Theorem~\ref{thm:condition for extension}
\end{proof} 
\begin{proof}[{\bf Proof of Theorem~\ref{S_A-in-number}}] 
	
\noindent {\bf (1): Let $r> 2 \ee $ and $\beta \in \cO_{\ell'}^\times.$} 
Note that $\ell > \ee.$	By Proposition~\ref{prop:charactrization of h_Atilda^ell and h_Atilda^ell'}(2), we have
$	\mathrm{h}_{\tilde{A}}^{\ell}= \pi^{\ell}\cO_r .$ Therefore by  Corollary~\ref{cor:E_tilde A remark}(1), we obtain that
	$E_{\tilde{A}}=\pi^\ell\cO_r .$
	
\noindent {\bf (2): Let $r =4 \ee$ and $\beta \in \pi \cO_{\ell'}.$} 
Let $m_0= \min \{\ee, \val(\tilde\beta)\}.$ From Proposition~\ref{prop:charactrization of h_Atilda^ell and h_Atilda^ell'}(2), we have $\mathrm{h}_{\tilde{A}}^{\ell}=\pi^{\ell-m_0}\cO_r  .$
By Corollary~\ref{cor:E_tilde A remark}(2), we have $\pi^\ell\cO_r\subseteq E_{\tilde{A}}.$ Also, for $\lambda \in E_{\tilde{A}},$ $\lambda+\pi^\ell\cO_r\subseteq E_{\tilde{A}}.$ Therefore
 to show  that $E_{\tilde{A}}$ is either $ \{0, \pi^{\ell-1} ( \tilde \xi w^{2})^{-1}\}+\pi^\ell\cO_r $ or $\pi^\ell\cO_r ,$
 it is enough to prove the following: 
\begin{enumerate}[(a)]
\item $(\mathrm{h}_{\tilde{A}}^{\ell}\setminus \pi^{\ell-1}\cO_r ) \cap E_{\tilde{A}} =\emptyset .$ 
\item If $ \pi^{\ell-1} u \in E_{\tilde{A}}$ 
 for some $u \in \cO_{r}^\times,$ then
 $  u=(\tilde \xi w^{2})^{-1} \,\, \mathrm{mod}\,(\pi).$
\end{enumerate}

We first show (a).
 Let 
$\lambda \in\mathrm{h}_{\tilde{A}}^{\ell}\setminus \pi^{\ell-1}\cO_r .$ Then $ \ell-m_0\leq \ell-2$ and  $\lambda=\pi^i u,$
 for some  $u \in \cO_{r}^\times$ and $\ell-m_0\leq  i \leq \ell-2.$   Now choose $x=1+\pi^{r-\ee-i-1}(wu)^{-1}$ and $y=0.$ Then we obtain  $x^2 + \tilde{\beta} xy - \tilde{\alpha} y^2 = 1 \,\, \mathrm{mod}\, (\pi^\ell),$ $\lambda y = 0 \,\, \mathrm{mod}\,(\pi^\ell)$ and   
\begin{eqnarray*}
 \psi(f( \lambda,x,y))&=&\psi(xy \lambda (\tilde \beta- \lambda ) -\tilde \alpha \lambda y^2 + \lambda ( x^2-1) )\\
&=&  \psi(\pi^i u(2\pi^{r-\ee-i-1}(wu)^{-1}+\pi^{2r-2\ee-2i-2}(wu)^{-2}) ) \\
 &=& \psi(\pi^{r-1}) \neq 1. 
\end{eqnarray*}
Here we used the facts that $2=\pi^{\ee}w$ and $2r-2\ee-i-2\geq 2r-2\ee-(\ell-2)-2= r .$
%
%
%
Therefore $(x,y ) \in E_{\lambda, \tilde A}\setminus E_{\lambda, \tilde A}^{\circ}.$ By Theorem~\ref{thm:condition for extension}, $\lambda \notin E_{\tilde{A}} .$ Hence (a) holds.

Let 
 $ \lambda=\pi^{\ell-1} u \in E_{\tilde{A}}, $
 for some $u \in \cO_{r}^\times$ be such that  $u\neq (\tilde \xi w^2)^{-1}\,\, \mathrm{mod}\,(\pi).$ To show (b), it is enough to prove that $\lambda \notin  E_{\tilde{A}}.$ Note that $\tilde\xi \neq (u w^2)^{-1}\,\, \mathrm{mod}\,(\pi).$ By definition of $\tilde\xi,$ there exists $z\in \cO_r$ such that $\psi(\pi^{r-1}(z+(u w^2)^{-1} z^2 ))\neq 1.$
	Now observe that the pair $(x,y):=(1+\pi ^{\ee}(uw)^{-1}z, 0)\in   E_{\lambda, \tilde A}.$ But 
\begin{eqnarray*}
\psi(f(\lambda,x,y))&=&\psi(xy \lambda (\tilde \beta- \lambda ) -\tilde \alpha \lambda y^2 + \lambda ( x^2-1) ) \\
&=&\psi(\pi^{\ell-1} u (2\pi ^{\ee}(uw)^{-1}z+\pi ^{2\ee}(uw)^{-2}z^2))\\
&=&\psi(\pi^{\ell-1} u (\pi ^{2\ee}u^{-1}z+\pi ^{2\ee}(uw)^{-2}z^2))\\
&=&\psi(\pi^{r-1}(z+(u w^2)^{-1} z^2 )
\neq 1.
\end{eqnarray*}
	Therefore $(x,y)\notin  E_{\lambda, \tilde A}^{\circ}.$ By Theorem~\ref{thm:condition for extension}, $\lambda \notin  E_{\tilde{A}}.$
	Hence (b) holds.

	Now assume $\ee=1$ and so $r=4.$
		Let $\beta =\pi v$ and $\tilde \beta =\pi \tilde v.$ It is easy to notice that 
		an element $ (x,y)\in  E_{\pi ( \tilde \xi w^{2})^{-1}, \tilde A}$ if and only if $x^2=1\,\, \mathrm{mod}\, (\pi^2)$ and $y=0 \,\, \mathrm{mod}\, (\pi).$ These  conditions on $x$ and $y$ hold if and only if 
		$ (x,y)=(1,0)\,\, \mathrm{mod}\,(\pi).$ Thus  arbitrary  $(x,y)=(1+\pi x_1 , \pi y_1 )\in  E_{\pi ( \tilde \xi w^{2})^{-1}, \tilde A}$ satisfies 
		\begin{eqnarray}
		\psi(f(\pi ( \tilde \xi w^{2})^{-1}, 1+\pi x_1,\pi y_1) ) &=&\psi(\pi^{3}(
		\frac{\tilde v}{\tilde \xi w^2}y_1- \frac{1}{\tilde \xi^2 w^4}y_1-\frac{\tilde \alpha}{\tilde \xi w^2}y_1^2
		)).\label{eqn-mmm}
		\end{eqnarray} 
		Here we used the facts that $xy\pi ( \tilde \xi w^{2})^{-1}(\tilde \beta -\pi ( \tilde \xi w^{2})^{-1})=\pi^{3}(	\frac{\tilde v}{\tilde \xi w^2}y_1- \frac{1}{\tilde \xi^2 w^4}y_1)$ 
and 
		\begin{eqnarray*}
\psi(\pi ( \tilde \xi w^{2})^{-1} ( x^2-1))&=&\psi(\pi ( \tilde \xi w^{2})^{-1} ( 2\pi x_1+\pi^2 x_1^2))\\
	&=&\psi(\pi^3( (\tilde \xi^{-1} w^{-1}x_1) +\tilde \xi(\tilde \xi^{-1} w^{-1}x_1)^2))\\
&=&1.
\end{eqnarray*}
		   Now let $\mu \in \cO_r$ be such that $\mu^2=\tilde \alpha \,\, \mathrm{mod}\,(\pi).$ By definition of $\xi,$ we have $ \psi(\pi^{3}(
		\frac{\mu}{\tilde \xi w}y_1+\tilde \xi \frac{\tilde \alpha}{(\tilde \xi w)^2}y_1^2
		))=1 .$ Hence  (\ref{eqn-mmm}) implies 
		\begin{eqnarray*}
			\psi(f(\pi ( \tilde \xi w^{2})^{-1}, 1+\pi x_1,\pi y_1) ) &=&\psi(\pi^{3}(\tilde \xi w)^{-1}(
			\frac{\tilde v}{w}- \frac{1}{\tilde \xi w^3}+ \mu)y_1
			).
		\end{eqnarray*} 
		
		Since $\psi(\pi^{3})\neq 1$ and $y_1 \in \cO_r$ is arbitrary, we conclude that $\psi(f(\pi ( \tilde \xi w^{2})^{-1}, 1+\pi x_1,\pi y_1) )=1$ if and only if $\frac{\tilde v}{w}- \frac{1}{\tilde \xi w^3}+ \mu=0 \,\, \mathrm{mod}\,(\pi).$
This  condition on $\tilde v$ and $ \mu$ holds if and only if $\alpha=(\frac{\tilde v}{w})^2+(\frac{1}{\tilde \xi w^3})^2 \,\, \mathrm{mod}\,(\pi).$ Therefore by Theorem~\ref{thm:condition for extension}, we must have 
		$E_{\tilde{A}} = \{0, \pi ( \tilde \xi w^{2})^{-1}\}+\pi^2\cO_r $ if and only  if $\alpha=(\frac{\tilde v}{w})^2+(\frac{1}{\tilde \xi w^3})^2 \,\, \mathrm{mod}\,(\pi).$
		
\noindent {\bf (3): Let $r > 4 \ee$ and $\beta \in \pi \cO_{\ell'}.$}
From Proposition~\ref{prop:charactrization of h_Atilda^ell and h_Atilda^ell'}(2), we have 
$
\mathrm{h}_{\tilde{A}}^{\ell}=\pi^{\ell-m_0}\cO_r ,$ where $ m_0= \min \{\ee, \val(\tilde\beta)\}.$ Note that $  \mathrm{h}_{\tilde{A}}^{\ell}\cap \pi^{\ell'}\cO_r=\pi^{\ell'}\cO_r.$
Hence by Corollary~\ref{cor:E_tilde A remark}(1), we  have $\pi^{\ell'}\cO_r \subseteq E_{\tilde{A}}.$ Therefore by Theorem~\ref{thm:condition for extension}, to show $E_{\tilde{A}}=\pi^{\ell'}\cO_r ,$ it is enough to prove that for any $\lambda \in 	\mathrm{h}_{\tilde{A}}^{\ell} \setminus \pi^{\ell'}\cO_r  ,$ there exists 
	$(x,y ) \in E_{\lambda, \tilde A}\setminus E_{\lambda, \tilde A}^{\circ}.$
Let  $\lambda \in 	\mathrm{h}_{\tilde{A}}^{\ell} \setminus \pi^{\ell'}\cO_r .$ Then  $\lambda = \pi^{i}u,$ for some $u \in \cO_{r}^\times$ and $\ell-m_0\leq  i \leq \ell'-1.$
Let us choose $x=1+\pi^{r-\ee-i-1}(wu)^{-1}$ and $y=0.$ Then by using the fact that $ \ell>2\ee,$   we obtain  $x^2 + \tilde{\beta} xy - \tilde{\alpha} y^2 = 1 \,\, \mathrm{mod}\, (\pi^\ell),$ $\lambda y = 0 \,\, \mathrm{mod}\,(\pi^\ell)$ and   
\begin{eqnarray*}
\psi(f( \lambda,x,y))&=&\psi(xy \lambda (\tilde \beta- \lambda ) -\tilde \alpha \lambda y^2 + \lambda ( x^2-1) )\\
&=& \psi(\pi^{i}u( 2\pi^{r-\ee-i-1}(wu)^{-1}+\pi^{2r-2\ee-2i-2}(wu)^{-2}) )\\
& =& \psi(\pi^{r-1}) \neq 1. 
\end{eqnarray*}
%
%
	Therefore $(x,y ) \in E_{\lambda, \tilde A}\setminus E_{\lambda, \tilde A}^{\circ}.$  This completes the proof of Theorem~\ref{S_A-in-number}.
\end{proof}

\section{Proof of Proposition~\ref{prop:possible-valuations-of-lambda}, Corollary~\ref{cor:trivial action} and Theorem~\ref{thm:simplification of extension conditions}}\label{sec:proof of thm:simplification of extension conditions}
In this section, we first prove Proposition~\ref{prop:possible-valuations-of-lambda} and Corollary~\ref{cor:trivial action}. We next prove a few results required for the proof of Theorem~\ref{thm:simplification of extension conditions}.  At last, we prove Theorem~\ref{thm:simplification of extension conditions}. 
\begin{proof}[\bf{Proof of Proposition~\ref{prop:possible-valuations-of-lambda}}]
	
	Let $\lambda \in  \mathrm{h}_{\tilde{A}}^\ell \setminus \pi^{\ell'}\cO_r $ such that $2\bol j_\lambda + \bol i_\lambda = 2 \ell' +\bol  s - \epsilon.$ Then we have the following observations: 
	\begin{enumerate}[(a)]
		\item For $\bol i_\lambda< \bol k,$ by Lemma~\ref{i-j-k-s-lem}(2), we obtain that $ \bol j_\lambda =\bol  i_\lambda$.
		Therefore $\bol  i_\lambda = \frac{2\ell' +\bol  s -\epsilon}{3}.$ This forces that $ \frac{2\ell' +\bol  s -\epsilon}{3} $ is an integer and $ 2\ell' +\bol  s -\epsilon< 3\bol k.$ 
		\item For $ \bol i_\lambda =\bol  k,$ Lemma~\ref{i-j-k-s-lem}(4) gives $\bol j_\lambda \geq \bol k$. Therefore $\bol j_\lambda = \frac{2\ell' +\bol  s -\epsilon- \bol k}{2}$ must be an integer and $2\ell' + \bol s -\epsilon \geq 3\bol k.$ 
		\item For $\bol i_\lambda > \bol k,$ by Lemma~\ref{i-j-k-s-lem}(2) we obtain $\bol j_\lambda =\bol  k$. Hence $\bol  i_\lambda = 2\ell' +\bol  s -\epsilon-2\bol k.$ Therefore $2\ell' + \bol s -\epsilon=\bol  i_\lambda+2\bol k>3\bol k.$ 
		\item Note that $ 2\ell' + \bol s -\epsilon-2\bol k< \ell'$ if and only if $ 2\bol k-\bol s\geq \ell$ (by using $\ell=\ell'-\epsilon+1$).
		
	\end{enumerate}

	Note that $(1),$ $(2)$  and $(3)$ follows from (a)-(c). To show $(4),$ observe that for $ 2\bol k-\bol s <  \ell,$ we obtain 
	$2\ell'+\bol s -\epsilon=\ell'+(\ell-1)+\bol s \geq \ell' +(2\bol k - \bol s )+\bol s =3 \bol k +(\ell' -\bol k )\geq 3 \bol k.$
	Therefore by $(2),$ $\bol i_\lambda \in \{ \bol k, 2\ell' + \bol s -\epsilon-2\bol k  \}.$	By (d), we obtain that $2\ell' + \bol s -\epsilon-2\bol k\geq \ell'.$ For $\lambda \in  \mathrm{h}_{\tilde{A}}^\ell \setminus \pi^{\ell'}\cO_r $, we have $\bol i_\lambda <\ell'$. Therefore we must have $\bol i_\lambda=\bol k.$
	
	We now prove $(5)$ and $(6)$ together. For this we consider $2 \bol k -\bol s \geq \ell$ and $2 \bol k -\bol s < \ell$ separately. \\
	
	\noindent
	{\bf (Case 1, $2 \bol k -\bol s \geq \ell$):}
	First of all, by using the facts that $2\bol j_\lambda + \bol i_\lambda = 2 \ell' +\bol  s - \epsilon$ and   $\ell=\ell'-\epsilon+1,$ we obtain 
	$$\bol j_\lambda - \bol s - \lceil (\ell-\bol s) /2 \rceil \geq \frac{2\bol j_\lambda -2\bol s - (\ell-\bol s +1 )}{2}=\frac{\ell'-\bol i_\lambda -2}{2}\geq \frac{-1}{2}.$$
	Since both $\bol j_\lambda - \bol s$ and $ \lceil (\ell-\bol s) /2 \rceil $ are integers, we must have $\bol j_\lambda - \bol s \geq \lceil (\ell-\bol s) /2 \rceil $. Hence $(6)$ follows from $(5)$ in this case. 
To show $(5),$ note that by $(3),$ we have $\bol i_\lambda \geq\mathrm{min} \{\bol  k,	\frac{2\ell' +\bol  s -\epsilon}{3} \} .$
Since $\bol i_\lambda$ is an integer, we can even assume that 
$\bol i_\lambda \geq\mathrm{min} \{\bol  k,	\lceil\frac{2\ell' +\bol  s -\epsilon}{3}\rceil \} .$ Therefore to show $(5),$
 it is enough to show that $\lceil (\ell-\bol s) /2 \rceil \geq \ell- \bol k$ and $\lceil (\ell-\bol s) /2 \rceil \geq \ell-	\lceil\frac{2\ell' +\bol  s -\epsilon}{3}\rceil.$ It follows because 
	$\lceil (\ell-\bol s) /2 \rceil -(\ell- \bol k)\geq \frac{2\bol k- \bol s -\ell}{2}\geq 0$ 
	and 
	\begin{eqnarray*}
		\lceil (\ell-\bol s) /2 \rceil - (\ell-\Big\lceil\frac{2\ell' +\bol  s -\epsilon}{3}\Big\rceil) &\geq& \frac{\ell-\bol s} {2} - (\ell-\frac{2\ell' +\bol  s -\epsilon}{3})\\
		&=&\frac{\ell' +3(\ell'-\ell) -\bol s - 2 \epsilon}{6}\\
	&=&\frac{\ell' +3(\epsilon -1) -\bol s - 2 \epsilon}{6}\\
		&=&\frac{\ell'  -\bol s + \epsilon-3}{6}\geq \frac{-4}{6},
	\end{eqnarray*}
	therefore $	\lceil (\ell-\bol s) /2 \rceil - ( \ell-\lceil\frac{2\ell' +\bol  s -\epsilon}{3}\rceil ) \geq 0.$
	Here we used the fact that $\bol s\leq \ell'+1.$\\

	\noindent
	{\bf (Case 2, $2 \bol k -\bol s < \ell$):} First of all observe that from $(4),$ we have $\bol i_\lambda = \bol k$ in this case.
	Therefore to show $(5)$ and $(6),$ it is enough to show that $\ell- \bol k \geq \lceil (\ell-\bol s) /2 \rceil $ and for $\bol s< \bol k,$ $ \bol j_\lambda-\bol s  \geq \ell- \bol k .$ But it follows from 
	$\ell- \bol k - \lceil (\ell-\bol s) /2 \rceil \geq \frac{2\ell -2\bol k-(\ell -\bol s +1)}{2} =\frac{\ell-(2\bol k-\bol s)-1}{2}\geq 0$ and for  $\bol s< \bol k,$ 
	\begin{eqnarray*}
		\bol j_\lambda-\bol s  -( \ell- \bol k) &=&\frac{2\ell' +\bol  s -\epsilon-\bol k}{2} -\bol s  -( \ell- \bol k)\\
		&=&\frac{2(\ell'-\ell)+\bol k -\bol  s -\epsilon}{2} \\
		&=&\frac{\bol k -\bol  s +\epsilon-2}{2} \geq \frac{\epsilon-1}{2}\geq \frac{-1}{2},
	\end{eqnarray*}
therefore $\bol j_\lambda-\bol s  -( \ell- \bol k)\geq 0.$
	This completes the proof of $(5)$ and $(6)$ for this case.
\end{proof}

\begin{proof}[\bf{Proof of Corollary~\ref{cor:trivial action}}]
	Let $\lambda \in  \mathrm{h}_{\tilde{A}}^\ell \setminus \pi^{\ell'}\cO_r $ such that $2\bol j_\lambda + \bol i_\lambda = 2 \ell' +\bol  s - \epsilon .$
	We first claim that $\bol s < \bol k.$ As if $\bol s \geq \bol k,$ then our assumption $2\bol k -\bol s \geq \ell$ implies 
	$$  \ell \leq 2\bol k -\bol s \leq \bol k \leq \ell' \leq \ell.$$
	Therefore we must have $\ell = \ell'=\bol k =\bol s.$ Since $r=\ell + \ell'=2\ell,$  we have $\epsilon =1.$ 
	Observe that $2 \ell' +\bol  s - \epsilon=3 \ell -1<3\ell=3\bol k.$ Therefore by Proposition~\ref{prop:possible-valuations-of-lambda}(1), we must have $\bol i_\lambda=\frac{2 \ell' +\bol  s - \epsilon}{3}=\frac{3 \ell -1}{3}.$ This is not possible because $\frac{3 \ell -1}{3}$ is not an integer. Hence the claim.
	
	Let  $x I + y \tilde A \in C_{\GL_2(\cO_r)}(\tilde{A})$ such that 
	$\det(x I + y \tilde A) = 1 \,\, \mathrm{mod}\, (\pi^\ell ).$
	By Proposition~\ref{prop:possible-valuations-of-lambda} (3), we have $\bol i_\lambda \geq \mathrm{min}\{\bol k, 	\frac{2\ell'+\bol s-\epsilon}{3}\} .$ 
	Therefore to show $\lambda y = 0 \,\, \mathrm{mod}\, (\pi^\ell )  ,$  it is enough to prove that $ \val(y)+ \mathrm{min}\{\bol k, 	\frac{2\ell'+\bol s-\epsilon}{3}\}> \ell-1 .$ Note that  $\det(x I + y \tilde A) = 1 \,\, \mathrm{mod}\, (\pi^\ell ) $ implies $ x^2 + \tilde \beta xy -\tilde \alpha y^2 = 1 \,\, \mathrm{mod}\, (\pi^\ell ).$ By substituting $\tilde \alpha = w_1^2 + \pi^s w_2^2,$  we get,
	\[
	\pi^{\bol s} w_2^2 y^2 = 1 + x^2 + w_1^2 y^2 + \tilde\beta xy \,\, \mathrm{mod} \, (\pi^\ell).
	\]
	
	Note that  $ 1 + x^2 + w_1^2 y^2$ is a perfect square in $\cO_r.$ This implies $\pi^\bol s w_2^2 y^2$ is also a perfect square modulo $(\pi^{\mathrm{min}\{\val( \tilde\beta xy),\ell\}}).$
	But the valuation of $\pi^\bol s w_2^2 y^2$  is odd. Hence we must have $\val(\pi^\bol s w_2^2 y^2)\geq \mathrm{min}\{\val( \tilde\beta xy),\ell\}.$
	Since $\bol s < \bol k,$  from Lemma~\ref{defintion-of-s-and-w1}, we have $w_2\in \cO_r^\times.$ 
	Therefore 
	$$2\, \val(y) + \bol s =\val(\pi^\bol s w_2^2 y^2) \geq \mathrm{min}\{\val( \tilde\beta xy),\ell\}\geq \mathrm{min}\{\bol k+\val(y),\ell\}.$$
	This implies  $\val(y)\geq \mathrm{min}\{\bol k-\bol s, \lceil\frac{\ell-\bol s}{2} \rceil  \}.$ Therefore to show $ \val(y)+ \mathrm{min}\{\bol k, 	\frac{2\ell'+\bol s-\epsilon}{3}\} > \ell-1 ,$ it is enough to show that $(\bol k-\bol s)+\bol k > \ell-1,$ $(\bol k-\bol s)+ 	\frac{2\ell'+\bol s-\epsilon}{3} > \ell-1 ,$ $  \lceil\frac{\ell-\bol s}{2} \rceil +\bol k > \ell-1$ and $   \lceil\frac{\ell-\bol s}{2} \rceil + 	\frac{2\ell'+\bol s-\epsilon}{3}> \ell-1.$ 
	All of these inequalities follow from our assumption $  2\bol k-\bol s\geq \ell$ and the facts that $\bol s < \bol k\leq \ell' $ and $\ell'=\ell-1+\epsilon.$ This completes the proof of Corollary~\ref{cor:trivial action}. 
\end{proof}

\begin{lemma}\label{lem:elements in E_lambda, A}
	If $\lambda \in  \mathrm{h}_{\tilde{A}}^\ell \setminus \pi^{\ell'}\cO_r $ such that $2\bol j_\lambda + \bol i_\lambda = 2 \ell' +\bol  s - \epsilon ,$ then
	$$
	E_{\lambda, \tilde{A} } =	\left\{(x, y)\in \cO_r \times \cO_r \mid
	\begin{array}{c}
	\val(y) \geq \mathrm{max}\{\ell- \bol i_\lambda, \ell- \bol k, \lceil \frac{\ell- \bol s}{2}\rceil  \}\\
\mathrm{and}\,\, x = 1+ w_1 y \,\, \mathrm{mod}\, (\pi^{\lceil \frac{\ell}{2}\rceil}) 
	\end{array}\right\}.
	$$
\end{lemma}
\begin{proof}
Note that Proposition~\ref{prop:possible-valuations-of-lambda}(4) implies  that either $\bol i_\lambda \leq \bol k,$ or $\bol i_\lambda> \bol k$ with $ 2\bol k-\bol s\geq \ell.$ So we consider these two cases separately. \\

\noindent	{\bf (Case 1, $\bol i_\lambda \leq \bol k$):} Suppose $(x,y)\in E_{\lambda, \tilde{A}}.$ Since  $\lambda y = 0 \,\,\mathrm{mod} \,(\pi^\ell ),$ we have $\val(y) \geq \ell-\bol i_\lambda .$ Hence $\tilde \beta xy=0 \,\,\mathrm{mod}\,(\pi^\ell),$ because $\val(\tilde\beta xy)\geq \bol k + \val(y)\geq \bol k+\ell -\bol i_\lambda\geq \ell.$ Therefore  $g(x,y) = 1 \,\,\mathrm{mod}\,(\pi^\ell)$ implies $ x^2 + (w_1^2 + \pi^\bol s w_2^2) y^2 = 1 \,\,\mathrm{mod}\,(\pi^\ell).$ This gives $  \pi^\bol s w_2^2y^2 = 1+x^2 +  w_1^2y^2  \,\,\mathrm{mod}\,(\pi^\ell).$  Note that the valuation of $\pi^\bol s w_2^2 y^2$  is odd and 
	$1+ x^2 + w_1^2y^2 \in \cO_r^2.$ So we must have $ \val(\pi^\bol s w_2^2 y^2)\geq \ell.$ Note that for $\bol s < \bol k,$ we have $w_2 \in \cO_r^\times$(by Lemma~\ref{defintion-of-s-and-w1}) and hence $2\val(y)+\bol s=\val(\pi^\bol s w_2^2 y^2)\geq \ell.$ 
Therefore for $\bol s < \bol k,$ we have $ \val(y)\geq  \lceil \frac{\ell-\bol s}{2}\rceil.$ Similarly for $\bol s \geq \bol k,$ we have  $ \val(y)\geq  \lceil \frac{\ell-\bol s}{2}\rceil,$ because 
$$ \val(y)\geq \ell-\bol i_\lambda \geq \ell -\bol k \geq \ell -\bol s \geq   \lceil \frac{\ell-\bol s}{2}\rceil.$$
 Therefore by 
comparing this with $\val(y)\geq \ell- \bol i_\lambda$ and $\bol i_\lambda \leq \bol k,$ we obtain that 
$\val(y) \geq \mathrm{max}\{\ell-\bol i_\lambda,  \lceil \frac{\ell-\bol s}{2}\rceil \}= \mathrm{max}\{\ell-\bol i_\lambda, \ell-\bol k, \lceil \frac{\ell-\bol s}{2}\rceil \}$ in this case. Further, this means $x^2 + w_1^2 y^2 = 1  \,\, \mathrm{mod}\,(\pi^\ell)$ which is equivalent to  $x = 1 + w_1 y \,\, \mathrm{mod}\,(\pi^{\lceil \frac{\ell}{2}\rceil}).$ Converse of this is easily seen to be true. \\

\noindent	{\bf (Case 2, $\bol i_\lambda> \bol k$ with $ 2\bol k-\bol s\geq \ell$):} 
First of all note that $\bol s < \bol k.$ This follows because $ \bol k - \bol s=(2\bol k - \bol s)-\bol k > \ell - \bol i_\lambda>0.$  Therefore by Lemma~\ref{defintion-of-s-and-w1}, we must have $w_2 \in \cO_r^\times.$
Suppose $(x,y)\in E_{\lambda, \tilde{A} }.$  Note that  $\lambda y = 0 \,\,\mathrm{mod} \,(\pi^\ell ) $ gives $\val(y) \geq \ell-\bol i_\lambda.$ Now $g(x, y) = 1 \,  \mathrm{mod} \,(\pi^\ell)$ implies  $ x^2  + w_1^2y^2 + 1 =  \pi^\bol s w_2^2 y^2 +\tilde{ \beta} xy \,\,\mathrm{mod}\,(\pi^\ell).$ Since the valuation of $\pi^\bol s w_2^2 y^2$  is odd and $x^2 + w_1^2y^2+1 \in \cO_r^2,$  we must have  
	\begin{equation}\label{eqn-s-val(y)}
	\bol s+ 2 \,   \val(y) =\val(\pi^\bol s w_2^2 y^2)\geq\mathrm{min}\{  \val(\tilde{ \beta} xy) , \ell \} \geq  \mathrm{min}\{  \val(y) + \bol k, \ell \}.
	\end{equation}
	We claim that  $\mathrm{min}\{  \val(y) + \bol k, \ell \}=\ell.$ Then as in  Case $1$ we obtain that $\val(y) \geq \mathrm{max}\{\ell-\bol i_\lambda,  \lceil \frac{\ell-\bol s}{2}\rceil \}.$ Also, the claim implies $\val(y) \geq  \ell- \bol k.$ Therefore the result follows.
	
	To show  the claim, suppose contrary that $ \val(y) + \bol k< \ell.$ Then by  (\ref{eqn-s-val(y)}) we will get $\bol s+ 2  \val(y) \geq \val(y) + \bol k ,$ which implies  $  \val(y) \geq  \bol k-\bol s .$ This gives 
$  \val(y)+\bol k \geq  2\bol k-\bol s .$ Since  $2\bol k-\bol s \geq \ell,$ we must have $ \val(y)+\bol k \geq \ell.$ This is a contradiction. Hence the claim.
\end{proof}
\begin{proposition}\label{proper z}
 Let $\cO \in \dvrtwoplus$ and $\xi$ be the unique element in $\mathbb F_q^\times$ as in Lemma~\ref{existance-of-zeta}. 
 Let $b,c \in \cO_{ r}^\times$ be such that $\xi  b^2 \neq  c^2 \, \,  \mathrm{mod}\,(\pi^{2m+1})$ for some $2m+1\in [1,r].$ Then there exists $z \in \cO_r$ such that $( bz)_m +(c^2z^2)_{2m}\notin \mathrm{ker}(\boldsymbol{\psi}).$ 
\end{proposition}
\begin{proof} 
	From direct calculations, 
	\begin{equation}\label{mateqn}
	\left[ \begin{matrix}
	(b)_0^2 &(b)_1^2& \cdots&(b)_m^2 \\ (c)_0^2 &(c)_1^2& \cdots&(c)_m^2 \end{matrix}\right] \left[ \begin{matrix}
	(z)_m^2  \\(z)_{m-1}^2  \\ \vdots\\ (z)_0^2 \end{matrix}\right]= \left[ \begin{matrix}
	(bz)_{m}^2  \\(c^2z^2)_{2m} \end{matrix}\right].
	\end{equation}
	If the matrix $ B:=\left[ \begin{matrix}
	(b)_0^2 &(b)_1^2& \cdots&(b)_m^2 \\(c)_0^2 &(c)_1^2& \cdots&(c)_m^2 \end{matrix}\right] \in M_{2\times (m+1)}( \mathbb F_q)$ has rank two, then we can find $z\in \cO_r$ such that $(bz)_{m} +(c^2z^2)_{2m} \notin \mathrm{ker}(\boldsymbol{\psi}).$ Therefore we assume that $B$ has rank either $1$ or $0.$
Since $b,c \in \cO_r^{\times},$ $B$ must have rank $1.$ Hence there exists $\xi' \in \mathbb F_q^\times$ such that $\xi'(b)_i^2=(c)_i^2$ for $0\leq i \leq m.$ This implies 
$\xi' b^2 = c^2 \,\, \mathrm{mod}\,(\pi^{2m+1}).$ For any $z\in \cO_r,$ 
	\begin{equation}
	\label{eq-9.2}
	(bz)_m +\xi'(bz)_{m}^2=( bz)_m +(\xi'b^2z^2)_{2m}=( bz)_m +(c^2z^2)_{2m}. 
		\end{equation}
	We will show the result by proving that $( bz)_m +(c^2z^2)_{2m}\in \mathrm{ker}(\boldsymbol{\psi})$ for all $z\in \cO_r$ implies
	$\xi  b^2 =  c^2 \,\, \mathrm{mod}\, (\pi^{2m+1}).$
	Suppose  $( bz)_m +(c^2z^2)_{2m}\in\mathrm{ ker}(\boldsymbol{\psi})$ for all $z\in \cO_r.$ By $b\in \cO_r^\times$ and  (\ref{eq-9.2}), we can conclude that 
	$  \{ x+ \xi' x^2  \mid x \in \mathbb  F_q\}\subseteq \mathrm{ker}(\boldsymbol{\psi}).$ Since both
	$  \{ x+ \xi' x^2  \mid x \in \mathbb  F_q\}$ and $ \mathrm{ker}(\boldsymbol{\psi})$
are $(n-1)$-dimensional subspaces (over $\mathbb F_2$) of $\mathbb F_q,$ where $n = [ \mathbb F_q : \mathbb F_2].$ We have $  \{ x+ \xi' x^2  \mid x \in \mathbb F_q\}=\mathrm{ker}(\boldsymbol{\psi}).$ But $\xi$ is the unique element in $\mathbb F_q^\times$ which satisfies $\mathrm{ker}(\boldsymbol{\psi}) =  \{ x+ \xi x^2  \mid x \in \mathbb F_q\}$. Therefore  $\xi'=\xi,$ hence $\xi b^2 = c^2 \, \,  \mathrm{mod}\, (\pi^{2m+1}).$
\end{proof}

\begin{proof} [{\bf Proof of  Theorem~\ref{thm:simplification of extension conditions} }]
By Theorem~\ref{thm:condition for extension}, $ \lambda \in  \mathrm{h}_{\tilde{A}}^\ell$ satisfies $\lambda \in E_{\tilde{A}} $ if and only if  $ E_{\lambda, \tilde{A}}= E_{\lambda, \tilde{A}}^\circ .$
By definition of $ E_{\lambda, \tilde{A}}$ and $ E_{\lambda, \tilde{A}}^\circ ,$ we have $ E_{\lambda, \tilde{A}}^\circ\subseteq E_{\lambda, \tilde{A}} .$
We prove our result by showing that $ \lambda \in  \mathrm{h}_{\tilde{A}}^\ell \setminus \pi^{\ell'}\cO_r $ satisfies $ E_{\lambda, \tilde{A}}\subseteq E_{\lambda, \tilde{A}}^\circ $ if and only if $\lambda$ satisfies conditions \textit{(I)}, \textit{(II)} and \textit{(III)} of Theorem~\ref{thm:simplification of extension conditions}. 
We first show the converse. 
Let $\lambda \in  \mathrm{h}_{\tilde{A}}^\ell \setminus \pi^{\ell'}\cO_r $ be such that $\lambda $  satisfies 
	\textit{(I)}, \textit{(II)} and \textit{(III)}. 
From Corollary~\ref{cor:E_tilde A remark}(3), observe that 
$f(\lambda + z, x, y ) = f(\lambda , x, y)$ for all $(x,y)\in E_{\lambda, \tilde{A}}$ and $z\in \pi^{\ell'}\cO_r.$ Since $\lambda \in \pi^{\ell- \ell'}\cO_{r}^2 \,\, \mathrm{mod}\,(\pi^{\ell'}),$  as far as the value of $f(\lambda, x, y) $ is concerned, we can even assume $\lambda \in \pi^{\ell-\ell'}\cO_{r}^2.$
 We show $ E_{\lambda, \tilde{A}}\subseteq E_{\lambda, \tilde{A}}^\circ $ for   $\bol s \geq \bol k$ and $\bol s < \bol k$ separately. Recall that  since $\cO \in \dvrtwoplus,$ any $z\in \cO_r$ has a unique expression of the form $z = (z)_0 + (z)_1 \pi + \cdots + (z)_{r-1} \pi^{r-1}$ for some $(z)_i \in \cO/\wp $
\\

\noindent
	{\bf (Case 1, $\bol s \geq \bol k$):} We first show that $\bol j_\lambda = \ell'$ and $\bol k = \bol i_\lambda.$
Note that by Lemma~\ref{i-j-k-s-lem}(3) and the fact that $\bol i_\lambda < \ell',$  $\bol j_\lambda = \ell'$ implies $\bol k = \bol i_\lambda.$   
To show $\bol j_\lambda = \ell',$ suppose the contrary that $\bol j_\lambda \leq \ell'-1.$
Since $\bol k\leq \bol s,$ we obtain 
$$ 2\ell' + \bol k-\epsilon \leq 2\ell' + \bol s-\epsilon = 2\bol j_\lambda + \bol i_\lambda \leq 2\ell' -2 +\bol i_\lambda . $$
%
%
This implies $\bol k \leq \bol i_\lambda-2+\epsilon<\bol i_\lambda.$ Hence by Lemma~\ref{i-j-k-s-lem}(2), we obtain   $  \bol j_\lambda=\bol k.$ 
 Therefore we have  $2\bol k + \bol i_\lambda = 2\ell' + \bol s-\epsilon.$ This together with $\bol k <\bol i_\lambda< \ell'$ implies  $\epsilon=(\ell'-  \bol i_\lambda)+(\ell'-  \bol k)+ (\bol s-  \bol k) \geq 2.$ 
But this is a contradiction to the fact that $\epsilon\in \{0, 1\}.$ 
Therefore $\bol j_\lambda = \ell'.$
	
	To show $E_{\lambda, \tilde{A}}\subseteq E_{\lambda, \tilde{A}}^\circ,$
	let $(x,y)\in E_{\lambda, \tilde{A}} .$ 
Recall that $f(\lambda, x, y)=xy \lambda (\tilde \beta- \lambda ) -\tilde \alpha \lambda y^2 + \lambda ( x^2-1).$  Since $2=0$ and $\tilde \alpha=w_1^2 + \pi^\bol s w_2^2, $  we obtain 
%
%
%
%
	\[
	f(\lambda, x, y) = xy \lambda(\lambda + \tilde \beta) + (w_1^2 + \pi^\bol s w_2^2 ) \lambda y^2 + \lambda(1 + x^2). 
	\] 
 By Lemma~\ref{lem:elements in E_lambda, A}, we have 
$\val(y) \geq \mathrm{max}\{\ell- \bol i_\lambda, \ell- \bol k, \lceil \frac{\ell- \bol s}{2}\rceil  \}$ and $x = 1+ w_1 y \,\, \mathrm{mod}\, (\pi^{\lceil \frac{\ell}{2}\rceil}) .$ 
Since $\bol j_\lambda = \ell',$ $ \bol k =  \bol i_\lambda $ and $\bol s\geq \bol k,$ we obtain that 
$xy \lambda(\lambda + \tilde \beta) = 0$ and $\pi^\bol s w_2^2 \lambda y^2 = 0.$ Because $\lambda \in \pi^{\ell-\ell'}\cO_{r}^2,$ we have $f(\lambda, x, y)= \lambda  w_1^2 y^2 +\lambda(1 + x^2)\in \pi^{\ell-\ell'} \cO_{r}^2.$
 We claim that for any element $v \in  \pi^{\ell-\ell'} \cO_{r}^2,$  $(v)_{r-1}=0.$
Therefore by definition of $\psi,$ we have $\psi(f(\lambda, x, y))=1$ and hence $E_{\lambda, \tilde{A}}\subseteq E_{\lambda, \tilde{A}}^\circ.$ Observe that the claim can be easily shown by using the fact that $r-1=\ell-\ell' +(2\ell'-1).$\\

\noindent
	{\bf (Case 2, $\bol s < \bol k$):} We claim that $\bol  j_\lambda <\ell'$ for this case. For $\bol  j_\lambda = \ell',$ $2\bol j_\lambda + \bol i_\lambda = 2\ell'+ \bol s -\epsilon$ implies $\bol  i_\lambda = \bol s-\epsilon< \bol k.$ Therefore, by
Lemma~\ref{i-j-k-s-lem}(2), we obtain $\bol j_\lambda= \bol  i_\lambda<\ell'.$ This is a contradiction to $\bol j_\lambda=\ell'.$ Hence the claim.
	
To show $E_{\lambda, \tilde{A}}\subseteq E_{\lambda, \tilde{A}}^\circ,$
	let $(x,y)\in E_{\lambda, \tilde{A}} .$ 
	By Lemma~\ref{lem:elements in E_lambda, A},  we have $x=1+w_1y \,\, \mathrm{mod}\, (\pi^{\lceil \frac{\ell}{2}\rceil})$ and $\val(y)\geq \mathrm{max}\{\ell- \bol i_\lambda, \ell- \bol k, \lceil \frac{\ell- \bol s}{2}\rceil  \}.$ Therefore by Lemma~\ref{i-j-k-s-lem}(4), we obtain that
	$\bol i_\lambda +\bol j_\lambda +2\val(y) \geq\bol i_\lambda + 	\mathrm{min} \{ \bol i_\lambda, \bol k\}+2\val(y)\geq 2\ell \geq r.$
	Similarly 
	\begin{eqnarray*} 
	\lceil \frac{\ell}{2}\rceil + \bol i_\lambda + \bol j_\lambda + \val(y) & \geq &  \frac{\ell+ (2\bol i_\lambda +2 \bol j_\lambda) +2 \val(y)}{2} \\  &= &  \frac{\ell+ (2\ell'+\bol s -\epsilon +\bol i_\lambda) +2 \val(y)}{2} \\  &\geq & \frac{\ell+2\ell'+ \bol i_\lambda+\val(y)}{2} \\ &\geq & \frac{\ell+2\ell'+\ell }{2}=r.
	 \end{eqnarray*}
	Hence $\pi^{\lceil \frac{\ell}{2}\rceil}y \lambda(\lambda +\tilde{ \beta})=w_1y^2 \lambda(\lambda +\tilde{ \beta})=0.$  This gives 
$xy \lambda(\lambda + \tilde \beta)=y\lambda(\lambda +\tilde{ \beta}) .$ Therefore
	\[
	f(\lambda, x, y) = y\lambda(\lambda +\tilde{ \beta}) + \pi^\bol s w_2^2 \lambda y^2 + \lambda v^2 ,\,\,\mathrm{where}\,\, v =1+x+w_1 y .
	\]
 Further since $ \lambda v^2 \in \pi^{\ell-\ell'} \cO_r^2,$ as in the above case, we have $(\lambda v^2)_{r-1}=0.$  Therefore 
	$(f(\lambda, x, y))_{r-1}=(y\lambda(\lambda + \tilde \beta) )_{r-1} + (\pi^\bol s w_2^2 \lambda y^2 )_{r-1}.$
Now by definition of $\delta_\lambda$,  we have  $\bol j_\lambda-\bol s-\delta_\lambda = \mathrm{max} \{ \ell-\bol i_\lambda, \ell-\bol k, \lceil (\ell-\bol s) /2 \rceil \}.$ Therefore
$ y = \pi^{\bol j_\lambda-\bol s-\delta_\lambda } z$ for some $z \in \cO_{r}.$ Hence 
	$ y\lambda(\lambda + \tilde{ \beta})  = \pi^{2 \bol j_\lambda+\bol i_\lambda-\bol s-\delta_\lambda }  u_1 u_2 z = \pi^{r-1 - \delta_\lambda} u_1 u_2 z,
	$
 where $u_1,u_2\in \cO_{r}^\times$ such that $\lambda = \pi^{\bol i_\lambda} u_1$
		and $ \lambda + \tilde{\beta} = \pi^{\bol j_\lambda} u_2.$
	Similarly,  
	$
	\pi^\bol s w_2^2 \lambda y^2 = \pi^{ 2\bol j_\lambda+\bol i_\lambda-\bol s-2\delta_\lambda} w_2^2 u_1 z^2 = \pi^{r-1-2\delta_\lambda} u_1w_2^2  z^2. 
	$
If $\delta_\lambda<0,$ then we see that both $\lambda(\lambda + \tilde{ \beta}) y$ and $ 
	\pi^\bol s  w_2^2 \lambda y^2 $ are zero. Therefore $(f(\lambda, x, y))_{r-1}=0$ and hence $\psi(f(\lambda, x, y))=1.$ So in this case, we must have $E_{\lambda, \tilde{A}}\subseteq E_{\lambda, \tilde{A}}^\circ.$

 Now we assume that $\delta_\lambda \geq 0.$ Then 
	$$(f(\lambda, x, y))_{r-1}=(\lambda(\lambda + \beta) y)_{r-1} + (\pi^\bol s  w_2^2 \lambda y^2 )_{r-1}=( u_1 u_2 z)_{\delta_\lambda} + ( u_1w_2^2  z^2)_{2\delta_\lambda}.$$ 
Since $\bol s < \bol k$ (our assumption) and  $\bol  j_\lambda <\ell', $ by condition \textit{(III)}, we have
$\xi u_1^2 u_2^2  =  u_1 w_2^2  \,\, \mathrm{mod} \, (\pi^{2 \delta_\lambda +1}).$ 
This implies
	$$\xi u_1^2 u_2^2 z^2 =  u_1 w_2^2 z^2 \,\, \mathrm{mod} \, (\pi^{2 \delta_\lambda +1}).$$ Therefore 
	$( u_1w_2^2  z^2)_{2\delta_\lambda}=\xi (u_1^2 u_2^2 z^2)_{2\delta_\lambda}=\xi( u_1 u_2 z)_{\delta_\lambda} ^2,$  and hence  
	$$(f(\lambda, x, y))_{r-1}=(u_1 u_2 z)_{\delta_\lambda} + ( u_1w_2^2  z^2)_{2\delta_\lambda}=( u_1 u_2 z)_{\delta_\lambda}+\xi( u_1 u_2 z)_{\delta_\lambda} ^2.$$
	But by the definition of $\xi,$ we obtain that $( u_1 u_2 z)_{\delta_\lambda}+\xi( u_1 u_2 z)_{\delta_\lambda} ^2 \in \mathrm{ker}(\boldsymbol{\psi}).$  Therefore $\psi(f(\lambda, x, y))=\boldsymbol{\psi}((f(\lambda, x, y))_{r-1})=1.$ Hence $E_{\lambda, \tilde{A}}  \subseteq  E_{\lambda, \tilde{A}} ^\circ.$

We next prove that  for $ \lambda \in  \mathrm{h}_{\tilde{A}}^\ell \setminus \pi^{\ell'}\cO_r ,$ $ E_{\lambda, \tilde{A}}\subseteq E_{\lambda, \tilde{A}}^\circ $ implies conditions \textit{(I)}, \textit{(II)} and \textit{(III)} are satisfied. 
We achieve this with the help of Claims 1, 2 and 3.

\begin{claim1}Let $ \lambda \in  \mathrm{h}_{\tilde{A}}^\ell \setminus \pi^{\ell'}\cO_r .$ If  $\lambda$ does not satisfy   \textit{(I)} of Theorem~\ref{thm:simplification of extension conditions}, that is    $ \lambda \notin \pi^{\ell - \ell'}\cO_{r}^2 \, \, \mathrm{mod} \, (\pi^{\ell'}),$ then there exists $(x, y) \in E_{\lambda, \tilde{A} } \setminus E_{\lambda, \tilde{A} }^\circ.$ 
\end{claim1}
\begin{proof}
Let  $\lambda = (\lambda)_0 + (\lambda)_1 \pi + \cdots + (\lambda)_{r-1} \pi^{r-1}$ with $(\lambda)_i \in \mathbb F_q$ be the  unique expression of $\lambda.$ Recall that $\ell-\ell'=1-\epsilon.$
As $ \lambda \notin \pi^{\ell - \ell'}\cO_{r}^2 \, \, \mathrm{mod} \, (\pi^{\ell'}),$ there exists non-negative integer $i$ such that  $ 0 \leq 2i + \epsilon < \ell'$ and  $(\lambda)_{2i + \epsilon} \neq 0.$ Let $\mu \in \mathbb F_q$ be such that $\mu^2=(\lambda)_{2i + \epsilon}^{-1}$. Choose   $x = 1 + \pi^{\ell'-i-\epsilon} \mu $  and $y = 0.$ We now proceed to prove that this choice of $(x,y)$ satisfies $(x,y) \in E_{\lambda, \tilde{A} } \setminus E_{\lambda, \tilde{A} }^\circ.$
 
 \noindent {\bf \underline{$\lambda y = 0 \,\, \mathrm{mod}\,(\pi^\ell) $:}} This is obvious. 
 
 \noindent {\bf \underline{$g(x, y) = 1 \,\, \mathrm{mod}\,(\pi^ \ell)$: }}  For this we note that 
 $g(x, y) =1 +\pi^{2\ell'-2i-2\epsilon}\mu^2 .$ 
 So it remains to show that $2\ell'-2i-2\epsilon \geq \ell.$ But this follows because 
$$2\ell'-2i-2\epsilon=\ell'+(\ell'-2i-\epsilon)-\epsilon \geq\ell'+1-\epsilon =\ell.$$
 
 \noindent{\bf \underline{ $\psi (f(\lambda, x, y))\neq 1$: }} By the definition of $\psi,$ it is enough to show that 
 $(f(\lambda, x, y))_{r-1} =1 .$
 For the given values of $x$ and $y,$
 $	f(\lambda, x, y) =\lambda\pi^{2\ell'-2i-2\epsilon}\mu^2. $
Hence  $(f(\lambda, x, y))_{r-1}=\mu^2(\lambda)_{2i+\epsilon}=1.$
 Therefore $(x,y) \in E_{\lambda, \tilde{A} } \setminus E_{\lambda, \tilde{A} }^\circ.$
\end{proof}	
\begin{claim2}
	Let $ \lambda \in  \mathrm{h}_{\tilde{A}}^\ell \setminus \pi^{\ell'}\cO_r $ be such that $\lambda$ satisfies \textit{(I)}  and not  \textit{(II)} of Theorem~\ref{thm:simplification of extension conditions}. i.e.,
	$$\lambda \in \pi^{\ell - \ell'}\cO_{r}^2 \, \, \mathrm{mod} \,(\pi^{\ell'})\,\, \mathrm{and}\,\, 2\bol j_\lambda + \bol i_\lambda \neq  2 \ell' + \bol s - \epsilon. $$
	Then there exists $(x, y) \in  E_{\lambda, \tilde{A} } \setminus E_{\lambda, \tilde{A}}^\circ.$ 
\end{claim2}
\begin{proof} We first show that  $ \bol i_\lambda  + \epsilon$ is odd.
Since $ \bol i_\lambda <\ell'$ and  $\lambda \in \pi^{\ell - \ell'}\cO_{r}^2 \, \, \mathrm{mod} \,(\pi^{\ell'}),$ we obtain that the number $ \bol i_\lambda -(\ell-\ell')$ is even. This together with $\ell -\ell'=1-\epsilon$  implies
 $ \bol i_\lambda  + \epsilon$ is odd.
	We prove the claim for  $2\bol j_\lambda + \bol i_\lambda < 2 \ell'  + \bol s - \epsilon $ and $2\bol j_\lambda + \bol i_\lambda >  2 \ell'  + \bol s - \epsilon $ separately.\\

\noindent
	{\bf (Case 1: $2\bol j_\lambda + \bol i_\lambda < 2 \ell'  + \bol s - \epsilon $):} 
	For this case, we first show that $ \bol j_\lambda < \ell'.$ By Lemma~\ref{i-j-k-s-lem}(3), $ \bol j_\lambda =\ell' $  implies $ \bol i_\lambda = \bol k.$ Therefore for $ \bol j_\lambda =\ell' ,$ $ 2 \ell' + \bol k=2 \bol j_\lambda+ \bol i_\lambda  <  2 \ell' + \bol s - \epsilon.$ This implies $\bol k + \epsilon < \bol s .$ Note that both $\bol k + \epsilon \, \, (=\bol i_\lambda +\epsilon)$ and $\bol s$ are odd integers.
 Therefore $\bol k + \epsilon < \bol s $ implies  $\bol k + \epsilon \leq \bol s -2.$ This is not possible because by Lemma~\ref{i-j-k-s-lem}(1), we have $\bol s \leq \bol k + 1.$ Therefore we must have $\bol j_\lambda < \ell'.$
	Let $u \in \cO_r^\times$ such that $\lambda(\lambda + \til \beta)= \pi^{\bol i_\lambda+ \bol j_\lambda} u .$ Let $ y = u^{-1}\pi^{r-\bol i_\lambda-\bol j_\lambda-1}$ and  $x = 1+ y w_1 .$ 
 Now we show that this choice of $(x,y)$ is in $E_{\lambda, \tilde{A} } \setminus E_{\lambda, \tilde{A} }^\circ.$
 
 \noindent {\bf \underline{$\lambda y = 0 \,\, \mathrm{mod}\,(\pi^\ell) $:}}
This amounts to proving that $r - \bol j_\lambda -1 \geq \ell $ or equivalent to prove that $\bol j_\lambda \leq \ell' -1,$ which we have already shown above to be true. 
 
 \noindent {\bf \underline{$g(x, y) = 1 \,\, \mathrm{mod}\,(\pi^ \ell)$: }} 
For this we note that 
		$$g(x, y) = 1 +\tilde{ \beta} x u^{-1}\pi^{r-\bol i_\lambda-\bol j_\lambda-1} +w_2^2  u^{-2} \pi^{2(r-\bol i_\lambda-\bol j_\lambda-1) + \bol s}.$$ 
		So it remains to show that $r-\bol i_\lambda-\bol j_\lambda-1+\bol k \geq \ell$ and $2(r-\bol i_\lambda-\bol j_\lambda-1) + \bol s \geq \ell.$ Since $\bol k \geq \min \{\bol i_\lambda, \bol j_\lambda\}$ (by Lemma~\ref{i-j-k-s-lem}(4)), $\bol j_\lambda < \ell'$ and  $ \bol i_\lambda < \ell' ,$  we are done with the first case. For the second we recall that $r  = 2 \ell'+1 - \epsilon$ and therefore we obtain the following.
		\[
		2(r-\bol i_\lambda-\bol j_\lambda-1) + \bol s = (2\ell' + \bol s -\epsilon ) - (2\bol j_\lambda + \bol i_\lambda )-1 + r -\bol i_\lambda  \geq \ell. 
		\]


 \noindent{\bf \underline{ $\psi (f(\lambda, x, y))\neq 1$: }} 
By definition of $\psi,$ it is enough to show that 
		$f(\lambda, x, y) =\pi^{r-1}.$
		For the given values of $x$ and $y,$ we have, 
		\begin{eqnarray*}
			f(\lambda, x, y) &=&(1+w_1 y)y\lambda(\lambda+\tilde{ \beta})+\lambda \pi^\bol s w_2^2 y^2
		\end{eqnarray*}
	Since $\lambda(\lambda+\tilde{ \beta} ) = \pi^{\bol i_\lambda+ \bol j_\lambda}u$ and $y=u^{-1} \pi^{r-\bol i_\lambda-\bol j_\lambda-1} \in \pi \cO_{ r},$ we must have $(1+w_1 y)y\lambda(\lambda+\tilde{ \beta})=\pi^{r-1}.$ Further $2( r - \bol i_\lambda - \bol j_\lambda -1) + \bol i_\lambda + \bol s=r-1+(r-1+\bol s)-(2\bol j_\lambda+\bol i_\lambda)\geq r$ because of our assumption $2\bol j_\lambda + \bol i_\lambda < 2\ell' + \bol s -\epsilon=r-1+\bol s.$ This implies $\lambda \pi^\bol s w_2^2 y^2=0.$ Therefore $	f(\lambda, x, y) =\pi^{r-1}.$
Hence $(x, y) \in  E_{\lambda, \tilde{A}}  \setminus E_{\lambda, \tilde{A}} ^\circ.$\\

	\noindent{\bf (Case 2: $2\bol j_\lambda + \bol i_\lambda > 2 \ell'  + \bol s - \epsilon $):}
	For this case, we first show that $\bol i_\lambda+\epsilon-s\geq 2$ and $\bol s < \bol k.$ Our assumption $2\bol j_\lambda+\bol  i_\lambda > 2\ell' + \bol s -\epsilon$ implies $\bol i_\lambda+\epsilon-\bol s> 2\ell'-2\bol j_\lambda\geq 0.$  Recall that both $\bol i_\lambda+\epsilon$ and  $\bol s$ are odd integers. Therefore we must have $\bol i_\lambda+\epsilon-\bol s\geq 2.$ To show $\bol s < \bol k,$ suppose contrary that $\bol s \geq \bol k.$ Then we have 
$ \bol k \leq \bol s \leq \bol i_\lambda+\epsilon- 2<\bol i_\lambda.$
By Lemma~\ref{i-j-k-s-lem}(2), we obtain $\bol j_\lambda=\bol k.$ Therefore by our assumption, we have $ 2\bol k +\bol i_\lambda\geq 2 \ell'  + \bol s - \epsilon +1 .$ This implies $\bol i_\lambda\geq 2 \ell' -\bol k  + (\bol s-\bol k ) +1- \epsilon\geq 2\ell' -\ell'+1- \epsilon\geq \ell' .$ This contradicts the fact that $\bol i_\lambda<\ell'.$ Therefore $\bol s < \bol k.$

	Let $v\in \cO^\times$ be such that  $\lambda =\pi ^{\bol i_\lambda} v^{2}\, \, \mathrm{mod} \, (\pi^{\ell'}),$ this is possible because $\lambda \in \pi^{\ell - \ell'}\cO_{r}^2 \, \, \mathrm{mod} \,(\pi^{\ell'}).$ Since $\bol s < \bol k,$ from Lemma~\ref{defintion-of-s-and-w1}, we have $w_2\in \cO_r^\times.$  Let $ y =  v^{-1}w_2^{-1}\pi^{\frac{2\ell'-\bol s-\bol i_\lambda-\epsilon}{2}}$ and  $x = 1+ y w_1 .$ 
Now we show that  $(x,y)\in E_{\lambda, \tilde{A} } \setminus E_{\lambda, \tilde{A} }^\circ.$
%
 
 	\noindent{\bf \underline{$\lambda y = 0 \,\, \mathrm{mod}(\pi^\ell)$:}} For this  we need to prove that $\bol i_\lambda +\frac{2\ell'-\bol s-\bol i_\lambda -\epsilon}{2}\geq \ell$
	or equivalent to prove that $\bol i_\lambda-\epsilon - \bol s \geq 2(\ell -\ell').$ The last inequality follows from the facts that $\bol i_\lambda+\epsilon-\bol s\geq 2$ and   $\ell -\ell'=1-\epsilon.$
	
	\noindent{\bf \underline {$g(x, y) = 1 \,\, \mathrm{mod} (\pi^\ell )$:} } For our choice of $(x,y)$ we have, $g (x, y) = 1 + \tilde{ \beta} xy + \pi^\bol s w_2 ^2 y^2.$ First of all we show that $\tilde{ \beta} y \in \pi^\ell\cO_r.$ For that it is enough to show that $ \bol k + \frac{2\ell' - \bol s - \bol i_\lambda- \epsilon} {2} \geq  \ell $ or equivalent to $2 \bol k+2\ell' -\bol s-\bol i_\lambda-\epsilon  \geq2\ell.$ Now we consider the cases $\bol i_\lambda\leq \bol k$ and $\bol i_\lambda > \bol k$ separately. For $ \bol i_\lambda \leq \bol k,$  $\bol i_\lambda+\epsilon-\bol s\geq 2$ implies  $\bol k-\bol s-\epsilon\geq 2-2\epsilon.$ Hence 
	$$2 \bol k+2\ell' -\bol s-\bol i_\lambda-\epsilon =2\ell' +(\bol k-\bol i_\lambda) +(\bol k-\bol s-\epsilon)\geq 2\ell' +2-2\epsilon= 2\ell$$
	For $ \bol i_\lambda > \bol k,$ by Lemma~\ref{i-j-k-s-lem}(2), we have $ \bol j_\lambda = \bol k.$ Therefore by  using our assumption (i.e. $ 2\bol j_\lambda + \bol i_\lambda  >  2\ell' + \bol s - \epsilon$), and the fact that  $\bol i_\lambda< \ell',$ we obtain 
	$$2 \bol k+2\ell' -\bol s-\bol i_\lambda-\epsilon=(2\bol j_\lambda+\bol i_\lambda-\bol s+\epsilon) -2\epsilon+2(\ell'-\bol i_\lambda) > 2\ell'-2\epsilon+2 = 2\ell.$$
	Therefore $\tilde{ \beta} y \in \pi^\ell\cO_r.$
	 Next we proceed to show that $ \pi^\bol s w_2 ^2 y^2 \in \pi^\ell \cO_r.$ This follows because $(2\ell'-\bol s-\bol i_\lambda-\epsilon) + \bol s=(2\ell' -\epsilon)-\bol i_\lambda \geq (r-1)-(\ell'-1)=\ell.$ Therefore $g(x, y) = 1 \,\, \mathrm{mod} (\pi^\ell ).$

	\noindent{\bf \underline{$\psi (f(\lambda, x, y))\neq 1$:} } By definition of $\psi,$ it is enough to show that 
	$f(\lambda, x, y) = \pi^{r-1} .$  For the given values of $x$ and $y,$ we have 
	\[
	f(\lambda, x, y) = (1 + y w_1)y\lambda(\lambda + \til \beta) + \pi^\bol s w_2^2 \lambda y^2.  
	\]
	It is easy to see that $(1 + y w_1)y\lambda(\lambda +\til \beta) = 0 $ because 
	$$\frac{2\ell'-\bol s-\bol i_\lambda-\epsilon}{2} + \bol i_\lambda +\bol j_\lambda =\frac{(2\bol j_\lambda+\bol i_\lambda)+2\ell'-\bol s-\epsilon}{2}> 2\ell'-\epsilon=r-1$$
	Further $\pi^\bol s w_2^2 \lambda y^2 = \pi^{2\ell' -\epsilon}=\pi^{r-1}$ proves the result. 
	Therefore  $(x, y) \in  E_{\lambda, \tilde{A}}  \setminus E_{\lambda, \tilde{A}} ^\circ.$	
\end{proof}

\begin{claim3}	Let $ \lambda \in  \mathrm{h}_{\tilde{A}}^\ell \setminus \pi^{\ell'}\cO_r $ be such that $\lambda$ satisfies
	\textit{(I)}, \textit{(II)} and not \textit{(III)} of Theorem~\ref{thm:simplification of extension conditions}.
 Then  there exists $(x, y) \in E_{\lambda, \tilde{A}}  \setminus E_{\lambda, \tilde{A}} ^\circ.$ 
\end{claim3}
\begin{proof} 
	Since $\lambda$ does not satisfy \textit{(III)} of Theorem~\ref{thm:simplification of extension conditions}, we must have $\bol j_\lambda<\ell',$ $\bol s<\bol k,$  $\delta_\lambda\geq 0$  and $	\xi u_1^2 u_2^2 \neq u_1 w_2^2 \,\, \mathrm{mod} \, (\pi^{2 \delta_\lambda +1}),$ where $u_1,$ $u_2 \in \cO_{r}^\times$ such that $\lambda = \pi^{\bol i_\lambda} u_1$
		and $ \lambda + \tilde{\beta} = \pi^{\bol j_\lambda} u_2.$ 

First of all we show that $ u_1 \in \cO_r^2\, \, \mathrm{mod} \, (\pi^{2 \delta_\lambda +1}).$ Since $\lambda \in \pi^{\ell - \ell'}\cO_{r}^2 \, \, \mathrm{mod} \, (\pi^{\ell'})$ 
(by \text{(I)}), it is enough to show that $(2 \delta_\lambda + 1)\leq (\ell'-\bol i_\lambda) .$ Now by definition of $\delta_\lambda,$ we have $2 \delta_\lambda\leq2(\bol j_\lambda -\bol s)-  (\ell-\bol s).$ Also, $\lambda $ satisfies \textit{(II)} implies
  $2\bol j_\lambda+\bol i_\lambda=2\ell'+\bol s-\epsilon=\ell'+\ell+\bol s-1.$ Therefore  
	$$(\ell'-\bol i_\lambda) - (2 \delta_\lambda + 1) \geq (\ell'-\bol i_\lambda) - 2(\bol j_\lambda-\bol s) + (\ell-\bol s) - 1 = 0.$$
Hence $(2 \delta_\lambda + 1)\leq (\ell'-\bol i_\lambda) .$

	By definition of $\delta_\lambda,$ we have
	\begin{eqnarray}
	\label{eq:j-s-h} 
	\bol j_\lambda-\bol s-\delta_\lambda =  \mathrm{max} \{ \ell-\bol i_\lambda, \ell-\bol k, \lceil (\ell-\bol s) /2 \rceil \}.
	\end{eqnarray}
	Therefore $\bol j_\lambda-\bol s-\delta_\lambda\geq\ell-\bol i_\lambda>0.$ 
We choose $y = z\pi^{\bol j_\lambda-\bol s-\delta_\lambda}$ and $x = 1+ yw_1$ 
and show that there exists a choice of $z \in \cO_r$ such that the corresponding pair $(x,y)$ satisfies $(x,y) \in E_{\lambda, \tilde{A}}  \setminus E_{\lambda, \tilde{A}} ^\circ.$ 
	
	\noindent {\bf \underline{$\lambda y = 0 \,\, \mathrm{mod}\,(\pi^\ell) $:}} This follows from (\ref{eq:j-s-h}).

	\noindent {\bf \underline{$g(x, y) = 1 \,\, \mathrm{mod}\,(\pi^\ell)$:} }  For the given $(x,y),$  we note that 
	$$g(x, y) = 1 + \tilde{ \beta}xz \pi^{\bol j_\lambda-\bol s-\delta_\lambda} + w_2^2 z^2\pi^{2(\bol j_\lambda-\bol s-\delta_\lambda) + \bol s}.$$
By using (\ref{eq:j-s-h}), it is easy to show that 
$ \bol k +\bol j_\lambda-\bol s-\delta_\lambda\geq \ell$ and $2(\bol j_\lambda-\bol s-\delta_\lambda)+\bol s\geq\ell.$ 
Therefore $g(x, y) = 1 \,\, \mathrm{mod}\,(\pi^\ell).$ 
	
	\noindent {\bf \underline{$\psi (f(\lambda, x, y))\neq 1$:} } For the given values of $x$ and $y,$ we have, 
	\[
	f(\lambda, x, y) = (1 + y w_1) y\lambda (\lambda + \tilde{ \beta}) + \pi^\bol s w_2^2 \lambda y^2 .
	\]
	By Lemma~\ref{i-j-k-s-lem}(4), we have $\bol j_\lambda\geq \mathrm{ Min}\{\bol i_\lambda,\bol k\}.$ Therefore by  (\ref{eq:j-s-h}),	
$2(\bol j_\lambda-\bol s-\delta_\lambda)+\bol i_\lambda+\bol j_\lambda\geq r.$ Hence
$ y^2 w_1 \lambda (\lambda + \tilde{ \beta}) = 0.$ By substituting $y,$ $\lambda,$ $\lambda + \tilde{ \beta}$ and using  $2\bol j_\lambda+ \bol i_\lambda = 2\ell'+\bol s-\epsilon=r+\bol s-1,$ we obtain $f(\lambda, x, y) = \pi^{(r-1) - \delta_\lambda} u_1 u_2 z + \pi^{(r-1)-2\delta_\lambda} w_2^2 u_1 z^2.$ 
Therefore 
	$$(f(\lambda, x, y) )_{r-1}= ( u_1 u_2 z)_{\delta_\lambda} + ( u_1 w_2^2 z^2)_{2 \delta_\lambda}.$$

Let $b= u_1 u_2$ and $c\in  \cO_r^\times$ be such that  $c^2= u_1 w_2^2 \,\, \mathrm{mod}\,(\pi^{2\delta_\lambda+1}),$ which is possible because $ u_1 \in \cO_r^2\, \, \mathrm{mod} \,(\pi^{2 \delta_\lambda +1}).$ 
Now observe that $(f(\lambda, x, y) )_{r-1}= (b z)_{\delta_\lambda} + ( c^2 z^2)_{2 \delta_\lambda}.$ 
Since $	\xi u_1^2 u_2^2 \neq u_1 w_2^2 \,\, \mathrm{mod} \, (\pi^{2 \delta_\lambda +1}),$ we have
 $	\xi b^2 \neq c^2 \,\, \mathrm{mod} \,(\pi^{2 \delta_\lambda +1}).$ 
Therefore by Proposition~\ref{proper z}, there exists $z \in \cO_r$ such that  $( b z)_{\delta_\lambda} + ( c^2 z^2)_{2 \delta_\lambda}\notin \mathrm{ker}(\boldsymbol{\psi}).$ Hence for $(x,y)$ corresponding to  that choice of $z \in \cO_r,$ 
$$\psi (f(\lambda, x, y))=\boldsymbol{\psi} ( (f(\lambda, x, y))_{r-1})=\boldsymbol{\psi} ((b z)_{\delta_\lambda} + ( c^2 z^2)_{2 \delta_\lambda})\neq 1.$$
\end{proof} 
  This completes the proof of Theorem~\ref{thm:simplification of extension conditions}.
\end{proof} 
\section{Proof of Theorem~\ref{thm:quotient_abelian}}\label{sec:proof of thm:quotient_abelian}
In this section, we give a proof of Theorem~\ref{thm:quotient_abelian}.
Throughout this section,we assume $r=2\ell.$  
\begin{lemma}\label{lem:E_A in single Lambda}
	If $E_{\tilde A}\subseteq\{0, \lambda \} +  \pi^{\ell}\cO_{2\ell}$ for some $\lambda \in \mathrm{h}_{\tilde{A}}^{\ell},$ then  Theorem~\ref{thm:quotient_abelian} follows.
\end{lemma}
\begin{proof}
	Let $E_{\tilde A}\subseteq\{0, \lambda \} + \pi^{\ell}\cO_{2\ell}$ for some $\lambda \in \mathrm{h}_{\tilde{A}}^{\ell}.$
By Corollary~\ref{cor:E_tilde A remark}, we obtain that $E_{\tilde A}$ is either $\pi^{\ell}\cO_{2\ell}$ or  $\{0, \lambda \} + \pi^{\ell}\cO_{2\ell}.$ Since $2\lambda \in \pi^{\ell}\cO_{2\ell},$ both $\pi^{\ell}\cO_{2\ell}$ and $\{0, \lambda \} + \pi^{\ell}\cO_{2\ell}$ are additive subgroups of $\mathrm{h}_{\tilde{A}}^{\ell}.$ Also, we have
	 $[E_{\tilde A}: \pi^{\ell}\cO_{2\ell}]\leq 2.$
	Therefore $(1)$ follows.
We note that $C_{\SL_2(\cO_{2\ell})} (\psi_{A})\mathbb E_{\tilde A}$ is either $C_{\SL_2(\cO_{2\ell})} (\psi_{A})$ or $C_{\SL_2(\cO_{2\ell})} (\psi_{A})\langle e_\lambda\rangle.$ Therefore  $(2)$ follows from the definition of $\mathbb E_{\tilde A}.$ We obtain  $(3)$ and $(4)$ from the facts that $C_{\SL_2(\cO_{2\ell})} (\psi_{A})/K^\ell$ is abelian and $[C_{\SL_2(\cO_{2\ell})} (\psi_{A})\mathbb E_{\tilde A}:C_{\SL_2(\cO_{2\ell})} (\psi_{A})]\leq 2.$
\end{proof}

\begin{lemma}\label{action in abelian}

Assume	 $\cO \in \dvrtwoplus.$  
Let $A \in  M_2(\cO_{\ell})$ be cyclic with $2\bol k -\bol s \geq \ell.$ Then  
$e_\lambda $ for  $\lambda \in E_{\tilde A}$ acts trivially on $C_{\SL_2(\cO_{2\ell})} (\psi_{A})$ modulo $K^\ell.$ 
\end{lemma}
\begin{proof}
Let $\lambda \in E_{\tilde A}$ and $X\in C_{\SL_2(\cO_{2\ell})} (\psi_{A})$. 
To show this lemma, it is enough to prove $[e_\lambda, X]\in K^\ell.$ This is equivalent to prove 
 $[e_\lambda, X] = I \,\, \mathrm{mod}\,(\pi^\ell).$ As $X\in C_{\SL_2(\cO_{2\ell})} (\psi_{A})\subseteq C_{\GL_2(\cO_{2\ell})} (\psi_{A}),$  by Lemmas~\ref{lem:stabilizer-form}(1) and \ref{lem:centralizer-form}(2), there exists $x,y\in \cO_r$ such that $X=xI + y \tilde A \,\, \mathrm{mod}\,(\pi^\ell).$ 
So  $[e_\lambda, X] = [e_\lambda, xI + y \tilde A]  \,\, \mathrm{mod}\,(\pi^\ell).$ 
By Lemma~\ref{centralizer-operations} (2), we have $[e_\lambda, xI + y \tilde A] = I \,\, \mathrm{mod}\,(\pi^\ell)$ if and only $\lambda y \in \pi^\ell\cO_r .$ 
Therefore 
it is enough to show that $\lambda y \in \pi^\ell \cO_r .$ For $\lambda \in E_{\tilde A}\cap \pi^\ell\cO_r , $ this follows trivially. For $\lambda \in  E_{\tilde A}\setminus \pi^\ell\cO_r  ,$ by
 Theorem~\ref{thm:simplification of extension conditions}, we have  $2\bol j_\lambda + \bol i_\lambda = 2 \ell +\bol  s - 1 .$ Also, note that $\det( xI + y \tilde A)=\det(X)=1 \,\, \mathrm{mod}\,(\pi^\ell)$ and $2\bol k -\bol s \geq \ell.$ Therefore, by Corollary~\ref{cor:trivial action}, $\lambda y \in \pi^\ell \cO_r .$ 
\end{proof}

\begin{proof}[{\bf Proof of Theorem~\ref{thm:quotient_abelian}}]

We first assume $\cO \in \dvrtwozero.$ By Corollary~\ref{cor: S_A-in-number}, we have $E_{\tilde A}\subseteq\{0, \lambda \} + \pi^{\ell}\cO_{2\ell}$ for some $\lambda \in \mathrm{h}_{\tilde{A}}^{\ell}.$ In this case, the theorem follows from Lemma~\ref{lem:E_A in single Lambda}. 
Next we assume $\cO \in \dvrtwoplus.$  We prove the theorem for  $2\bol k -\bol s < \ell$ and $2\bol k -\bol s \geq \ell$ separately \\

\noindent
{\bf (Case 1: $2\bol k -\bol s < \ell$):}
In this case, we claim that $E_{\tilde A}\subseteq\{0, \lambda \} + \pi^{\ell}\cO_{2\ell}$ for some $\lambda \in \mathrm{h}_{\tilde{A}}^{\ell}.$ The result then follows from Lemma~\ref{lem:E_A in single Lambda}. 
If $E_{\tilde A}= \pi^{\ell}\cO_{2\ell},$ then the claim follows vacuously. So assume $E_{\tilde A}\neq \pi^{\ell}\cO_{2\ell}.$ Let $\lambda_1,\lambda_2 \in  E_{\tilde A}\setminus \pi^{\ell}\cO_{2\ell}.$ To show the claim, it is enough to show that $\lambda_1=\lambda_2 \,\, \mathrm{mod}\,(\pi^{\ell}).$ By Proposition~\ref{prop:possible-valuations-of-lambda}(4), we have $\bol i_{\lambda_1}=\bol i_{\lambda_2}=\bol k <\ell.$ Therefore $\bol j_{\lambda_1}=\bol j_{\lambda_2}=\frac{ 2\ell+ \bol s -1-\bol k}{2},$ because $\lambda_1$ and $\lambda_2$ satisfies 
\textit{(II)} of Theorem~\ref{thm:simplification of extension conditions}. We denote   $\bol j_{\lambda_1}$ and  $\bol j_{\lambda_2}$ by  $\bol j.$ Let $u\in \cO_{ 2\ell}^\times$ such that $\tilde \beta =\pi^\bol k u.$ For $\bol j=\ell,$ by definition of $\bol j,$ we obtain  $\lambda_1=\lambda_2=\tilde \beta \,\, \mathrm{mod}\,(\pi^{\ell}).$ So we are done in this case.
 Let  $\bol j< \ell$ and  $v_1,v_2 \in \cO_{ 2\ell}^{\times}$ be such that $\lambda_1=\tilde \beta + \pi^{\bol j}v_1 $ and $\lambda_2=\tilde \beta + \pi^{\bol j}v_2 .$ Further $ \bol j = \frac{ 2\ell+ \bol s -1-\bol k}{2}$ gives $\bol s=\bol k + 2(\bol j -\ell)+ 1<\bol k.$
Therefore by Proposition~\ref{prop:possible-valuations-of-lambda}(6), we have $\delta_\lambda\geq 0.$ 
Hence by \textit{(III)} of Theorem~\ref{thm:simplification of extension conditions}, $v_1$ and  $v_2$ must satisfy 
\begin{equation*}
\xi ( u +\pi^{\bol j -\bol k} v_i )v_i ^2 =w_2^2 \,\, \mathrm{mod}\,(\pi^{2\delta_\lambda +1})\,\, \mathrm{for}  \,\, i=1,2.
\end{equation*} 
In particular, we have $	uv_i^2= w_2^2 \xi^{-1} +\pi^{\bol j -\bol k} v_i ^3 \,\, \mathrm{mod}\,(\pi^{2\delta_\lambda +1})$ for $ i=1,2.$  Therefore 
\begin{eqnarray}
u(v_1+v_2)^2 &=&uv_1^2+uv_2^2   \label{eqn:2k-s<l ; cond-3}\\
&=&\pi^{\bol j-\bol k}(v_1^3+v_2^3)\nonumber \\
&=&\pi^{\bol j-\bol k}(v_1+v_2)(v_1^2+v_1v_2+v_2^2) \,\, \mathrm{mod}\,(\pi^{2\delta_\lambda +1}).\nonumber 
\end{eqnarray}

Note that $\val(u(v_1+v_2)^2)=2 \, \val(v_1+v_2).$ If $2 \, \val(v_1+v_2)<2\delta_\lambda+1,$ then by comparing the valuations of both sides of  (\ref{eqn:2k-s<l ; cond-3}), we must have $ \val(v_1+v_2)\geq \bol j -\bol k.$ In such case, since $2\bol j -\bol k= 2\ell-(2\bol k- \bol s) -1\geq \ell,$ we obtain
$\pi^\bol j (v_1+ v_2)=0 \,\, \mathrm{mod}\,(\pi^\ell).$ Hence $\lambda_1=\lambda_2\,\, \mathrm{mod}\,(\pi^{\ell}).$ 
If $2 \val(v_1+v_2)\geq 2\delta_\lambda+1,$ then  we have $\val(v_1+v_2)\geq \delta_\lambda+1.$  So
$$\bol j+\delta_\lambda+1=\bol j+(\bol j-\bol s-(\ell-\bol k))+1=2
\bol j-\bol s-\ell+\bol k+1=\ell,$$ 
where the first equality is by the definition of $\delta_\lambda$ and  Proposition~\ref{prop:possible-valuations-of-lambda} (5). Therefore $\pi^\bol j (v_1+ v_2)=0 \,\, \mathrm{mod}\,(\pi^\ell).$ Hence $\lambda_1=\lambda_2\,\, \mathrm{mod}\,(\pi^{\ell}).$\\

\noindent
{\bf (Case 2: $2\bol k -\bol s \geq \ell$):}
Let $\langle \mathbb{E}_{\tilde A}\rangle$ be the subgroup of $\mathrm{H}_{\tilde{A}}^{\ell}$ generated by $\mathbb{E}_{\tilde A}.$ First of all we claim that  each  $\chi\in  \mathrm{Irr}(C_{\SL_2(\cO_{2\ell})} (\psi_{A})\mid  \psi_{[A]})$ extends to $C_{\SL_2(\cO_{2\ell})} (\psi_{A})\langle \mathbb{E}_{\tilde A}\rangle.$ Note that $\langle \mathbb{E}_{\tilde A}\rangle$ is abelian. Therefore  by lemma~\ref{diamond-lemma},
to show the claim, it is enough to show that  $\chi^{e_\lambda}=\chi$ for all $\chi\in  \mathrm{Irr}(C_{\SL_2(\cO_{2\ell})} (\psi_{A})\mid  \psi_{[A]})$, and for all  $e_\lambda \in \mathbb E_{\tilde A}.$ By Lemma~\ref{action in abelian},  $[e_\lambda,X]\in K^\ell$  for all $e_\lambda \in \mathbb E_{\tilde A}$, and for all $ X\in C_{\SL_2(\cO_{2\ell})} (\psi_{A}).$
By definition, we have $\chi|_{K^\ell}=\psi_{[A]}$ for  $\chi\in  \mathrm{Irr}(C_{\SL_2(\cO_{2\ell})} (\psi_{A})\mid  \psi_{[A]}).$  This combined with  Theorem~\ref{thm:condition for extension} gives
$ \chi([e_\lambda,X])=\psi_{[A]}([e_\lambda,X])=1.$
Hence the claim follows. Therefore we obtain
the following results:
\begin{itemize}
	\item[(a)] The one-dimensional representation $ \psi_{[A]}$ extends to $C_{\SL_2(\cO_{2\ell})} (\psi_{A})\langle \mathbb{E}_{\tilde A}\rangle$ and  $C_{\SL_2(\cO_{2\ell})} (\psi_{A})\langle \mathbb{E}_{\tilde A}\rangle/K^\ell$ is abelian.
	\item[(b)] Every $g\in \langle \mathbb{E}_{\tilde A}\rangle$ is of the form $e_\lambda$ for some $\lambda \in \mathrm{h}_{\tilde{A}}^\ell.$ Note that $ \psi_{[A]}$ extends to $C_{\SL_2(\cO_{2\ell})} (\psi_{A})\langle g\rangle\subseteq C_{\SL_2(\cO_{2\ell})} (\psi_{A})\langle \mathbb{E}_{\tilde A}\rangle.$ Therefore by definition of $\mathbb{E}_{\tilde A},$ we obtain $g \in \mathbb{E}_{\tilde A}.$ Summing up $\langle \mathbb{E}_{\tilde A}\rangle=\mathbb{E}_{\tilde A}.$
	
\end{itemize}
Thus, except 
$[E_{\tilde A}: \pi^{\ell}\cO_{2\ell}]\leq 4$,  Theorem~\ref{thm:quotient_abelian} follows in this case.  We now proceed to prove $[E_{\tilde A}: \pi^{\ell}\cO_{2\ell}]\leq 4$.
For $E_{\tilde A} =\pi^\ell\cO_r ,$ it follows trivially. For  $E_{\tilde A} \neq \pi^\ell\cO_r ,$ we prove the result for  $2\ell+ \bol s -1\leq 3 \bol k$ and $2\ell+ \bol s -1>3 \bol k$ separately. \\

\noindent
{\bf (Sub case 1: $2\ell+ \bol s -1\leq 3 \bol k$):} Let $\lambda \in E_{\tilde A} \setminus \pi^\ell\cO_r .$ Let $d=\frac{ 2\ell+ \bol s -1}{3}.$ Clearly $d\leq \bol k.$ By Proposition~\ref{prop:possible-valuations-of-lambda} (1)-(2), we  obtain 
$\bol i_\lambda=\bol j_\lambda=d.$ Therefore $d$ is an integer and $d<\ell,$ because $ \bol i_\lambda<\ell.$ 
We claim that $\bol s<\bol k$ in this case.  For $\bol s \geq \bol k,$  our assumption $2\bol k -\bol s \geq \ell$ gives 
$\bol s\geq \ell$. This implies 
$d= \frac{ 2\ell+ \bol s -1}{3}\geq \ell- \frac{1}{3}$ and therefore $d\geq \ell.$ This  contradicts the fact that $d<\ell.$ Therefore  we must have $\bol s<\bol k$. Also, by \textit{(II)} of Theorem~\ref{thm:simplification of extension conditions}, we have $2 \bol j_\lambda+ \bol i_\lambda=2\ell+ \bol s -1.$ Hence by  Proposition~\ref{prop:possible-valuations-of-lambda} (6),  $\delta_\lambda \geq 0.$ 
Let $u  \in \cO_{ 2\ell}^\times$ and $v  \in \cO_{ 2\ell}$ be such that $\lambda= \pi^{d}u $ and $\tilde \beta =\pi^\bol k v.$ Let $T(x) = \xi x (x+ \pi^{\bol k-d} v )^2 +w_2^2$. 
By \textit{(III)} of Theorem~\ref{thm:simplification of extension conditions}, 
$T(u) = 0\,\, \mathrm{mod}\,(\pi^{2\delta_\lambda +1})$.
Since a third degree polynomial over a field has at most three distinct solutions, there are at most three possible values for  $u$ modulo $(\pi)$ such that $T(u) =0\,\, \mathrm{mod}\,(\pi).$ Note that $T'(u)=3\xi u^2 +\xi \pi^{2(\bol k-d)} v^2.$ We claim that 
$T'(u)  \neq 0  \,\, \mathrm{mod}\,(\pi).$ Recall that $\bol k\geq d.$
For $\bol k>d,$ the claim follows easily. For $\bol k= d,$ 
by definition of $\bol j_\lambda$ and $\bol j_\lambda = d<\ell,$ we must have $j_\lambda=\val(\lambda+\tilde \beta).$
Therefore  $u+v\neq 0  \,\, \mathrm{mod}\,(\pi).$ Hence $T'(u) =\xi (u  + v)^2 \neq 0  \,\, \mathrm{mod}\,(\pi)$ for $\bol k=d $ also.  
By Hensel's lemma, we obtain that there are at most three possible values for  $u  \in \cO_{ 2\ell}$ such that $T(u)=0\,\, \mathrm{mod}\,(\pi^{2\delta_\lambda +1}).$ This altogether gives that $E_{\tilde A}\setminus\pi^\ell\cO_r $ contains  at most three distinct  elements modulo $(\pi^{d+2\delta_\lambda +1}).$   By definition of $\delta_\lambda,$  our assumption $2\bol k -\bol s \geq \ell$ and Proposition~\ref{prop:possible-valuations-of-lambda} (5), we obtain $\delta_\lambda=d-\bol s-\lceil (\ell-\bol s) /2 \rceil.$ Hence  
$$d+2\delta_\lambda +1=2\ell- \bol s -2\lceil (\ell-\bol s) /2 \rceil =\begin{cases}
\ell&  \mathrm{for}\,\,\ell \,\,\mathrm{odd},\\
\ell-1  &  \mathrm{for}\,\,\ell \,\, \mathrm{even}.
\end{cases}$$
Therefore  $[E_{\tilde A}: \pi^{\ell}\cO_{2\ell}]\leq 4$ for odd $\ell.$  For even $\ell,$ result follows by using $ \lambda \in \cO_{r}^2 \, \, \mathrm{mod} \, (\pi^{\ell})$ (\textit{(I)} of Theorem~\ref{thm:simplification of extension conditions}) and  $\ell-1$ is odd.  \\

\noindent
{\bf (Sub case 2: $2\ell+ \bol s -1> 3 \bol k$):}
For $\lambda \in E_{\tilde A} \setminus \pi^\ell\cO_r ,$ by Proposition~\ref{prop:possible-valuations-of-lambda} (2),  we have  $\bol i_\lambda$ is either $\bol k$ or $2\ell+ \bol s -1-2\bol k.$ Let $d'=2\ell+ \bol s -1-2\bol k.$ Note that $\bol k< d'=2\ell-(2\bol k- \bol s) -1<\ell.$ 
Now we will show that there exists a unique  $\lambda \in E_{\tilde A}\setminus \pi^\ell\cO_r $ modulo $(\pi^\ell)$ with  $\bol i_\lambda=d'.$ For that let $\lambda  \in E_{\tilde A} \setminus \pi^\ell\cO_r $ be such that $\lambda=\pi^{d'}u$ for some $u \in \cO_{ 2\ell}^\times.$ By Lemma~\ref{i-j-k-s-lem}(2) and $\bol k< d'=\bol i_\lambda,$ we must have $\bol j_\lambda=\bol k.$ 
Since $\bol k< d'<\ell,$ we obtain $\bol s=d'+1+2\bol k -2\ell \leq 2\bol k -\ell<\bol k.$
By Proposition~\ref{prop:possible-valuations-of-lambda} (6),  $\delta_\lambda \geq 0.$
Therefore by \textit{(III)} of Theorem~\ref{thm:simplification of extension conditions}, $u$  satisfies
\begin{equation}\label{eqn:uv}
T(u):=\xi u (v+ \pi^{d'-\bol k}u)^2 +w_2^2 =0\,\, \mathrm{mod}\,(\pi^{2\delta_\lambda +1}),
\end{equation} 
where $v\in \cO_{ 2\ell}^\times$ such that  $\tilde \beta =\pi^\bol k v.$ Here $v\in \cO_{ 2\ell}^\times$ because $\bol k <\ell$ implies $\val(\tilde \beta)=\val(\beta)=\bol k.$ 
Since $\bol k< d',$ (\ref{eqn:uv}) implies $\xi u v^2 =w_2^2 \,\, \mathrm{mod}\,(\pi).$ Hence 
$u=\xi^{-1}v^{-2}w_2^2  \,\, \mathrm{mod}\,(\pi).$  Also, $	T'(u)=\xi v^2 \neq 0  \,\, \mathrm{mod}\,(\pi).$ By Hensel's lemma,  we obtain  that $u$ is uniquely determined modulo $(\pi^{2\delta_\lambda +1}).$ 
 By 
Proposition~\ref{prop:possible-valuations-of-lambda} (5), we obtain $\delta_\lambda=\bol k-\bol s-\lceil (\ell-\bol s) /2 \rceil.$ Hence
%
%
$$d'+2\delta_\lambda +1=2\ell- \bol s -2\lceil (\ell-\bol s) /2 \rceil =\begin{cases}
\ell &  \mathrm{for}\,\, \ell\,\, \mathrm{odd},\\
\ell-1 & \mathrm{for}\,\, \ell \,\, \mathrm{even}.
\end{cases}$$
By using the same arguments used in the above case, we can conclude that $\lambda=\pi^{d'}v$ is uniquely determined modulo $(\pi^\ell).$
We proceed to show the existence of such $\lambda \in E_{\tilde A} \setminus \pi^\ell\cO_r .$ Note that $x=\xi^{-1}\bar{v}^{-2}\bar{w_2}^2 $ is the unique solution in $\cO_1$  for $T(x)=0 \,\, \mathrm{mod}\,(\pi)$ in $\cO_1.$ Also, we have 
 $	T'(\xi^{-1}\bar{v}^{-2}\bar{w_2}^2 )=\xi v^2 \neq 0  \,\, \mathrm{mod}\,(\pi).$
  By Hensel's lemma, there exists a unique $ u_{*}\in  \cO_{ 2\ell}^\times$  such that  $T(u_{*})=0 \,\, \mathrm{mod}\,(\pi^{2\ell}).$ It is easy to show that $\lambda_{*}:=\pi^{d'}u_{*}$ satisfies  \textit{(I)}, \textit{(II)} and \textit{(III)} of  Theorem~\ref{thm:simplification of extension conditions}. Hence $\lambda_{*} \in  E_{\tilde A} \setminus \pi^\ell\cO_r .$

For $\lambda_1,\lambda_2 \in  E_{\tilde A} $ such that $\bol i_{\lambda_1}=\bol i_{\lambda_2}=\bol k,$ by  Proposition~\ref{prop:possible-valuations-of-lambda} (2), we have  $\bol j_{\lambda_1}=\bol j_{\lambda_2}=\frac{2\ell+ \bol s -1- \bol k}{2}>\bol k.$ Therefore $\lambda_1=\lambda_2=\tilde \beta \,\, \mathrm{mod}\,(\pi^{\bol k+1}).$ Hence $\lambda_1+\lambda_2=0 \,\, \mathrm{mod}\,(\pi^{\bol k+1}),$ that is, $\val(\lambda_1+\lambda_2)\geq \bol k +1.$ But we know that 
$E_{\tilde A}$ forms an additive group and each $\lambda \in E_{\tilde A}\setminus \pi^\ell\cO_r $ has valuation  either $\bol k$ or $ d'.$ 
Therefore we must have $\lambda_1+\lambda_2 \in \{ 0,\lambda_{*}\} \,\, \mathrm{mod}\,(\pi^{\ell}).$ So  $[ E_{\tilde A}: \pi^{\ell}\cO_{2\ell}]\leq 4.$ 
This completes the proof of Theorem~\ref{thm:quotient_abelian}.
\end{proof}
\section{Examples}\label{sec:Examples}
In this section, we discuss the examples  of $\SL_2(\cO_{2\ell})$ for $\cO = \mathbb Z_2$ and $\mathbb F_2 \llbracket t \rrbracket $ with $1 \leq \ell \leq 3$ in detail. We fix $\pi=2$ and $\pi=t$ for $\cO = \mathbb Z_2$ and  $\cO=\mathbb F_2 \llbracket t \rrbracket $ respectively. We note that $\cO/(\pi^r) \cong \mathbb Z/2^r \mathbb Z$ for $\cO = \mathbb Z_2$ and $\cO/(\pi^r) \cong \mathbb F_2[t]/(t^r)$ for $\cO = \mathbb F_2 \llbracket t \rrbracket $. We will use these isomorphisms and corresponding notations throughout this section. We first discuss a few general results regarding $\SL_2(\cO_{2 \ell})$ representations that do not use the assumption $1\leq \ell \leq 3.$ We will consider specific values of $\ell$ from Section~\ref{subsec:example-1} onwards.

Recall that $\mathbb M_A = C_{\SL_2(\cO_{2\ell})}(\psi_{A}) \mathbb E_{\tilde{A}}$  for $A = \mat 0 {a^{-1} \alpha} a \beta \in M_2(\cO_{\ell}).$
For $i\geq 1,$ let $\Delta_i$ be the set of all equivalence classes of $\{\phi \in \mathrm{Irr}(\mathbb {M}_A\mid  \psi_{[A]}) \mid \dim(\phi)=i \}$
under the conjugation action of $C_{\SL_2(\cO_{2\ell})}(\psi_{[A]}).$
By Theorem~\ref{main_theorem-2}, we have the following observations:

\begin{enumerate}
	\item For $i\geq 3,$ $\Delta_i=\emptyset.$
	\item Every representation $\rho \in  \mathrm{Irr}( \SL_2(\cO_{2\ell})\mid  \psi_{[A]}) $ is either of dimension $\frac{|\SL_2(\cO_{2\ell})|}{|\mathbb M_A|}$ or $2\times \frac{|\SL_2(\cO_{2\ell})|}{|\mathbb M_A|}.$  
	\item There are exactly $|\Delta_1|$ (resp. $|\Delta_2|$) many representations in $ \mathrm{Irr}( \SL_2(\cO_{2\ell})\mid  \psi_{[A]}) $ of dimension $\frac{|\SL_2(\cO_{2\ell})|}{|\mathbb M_A|}$  (resp. $2\times \frac{|\SL_2(\cO_{2\ell})|}{|\mathbb M_A|}$).
\end{enumerate}
For $\lambda \in \mathrm{h}_{\tilde{A}}^{\ell},$ let $\theta_\lambda = |\{ (x,y) \in \cO_\ell \times \cO_\ell : \lambda y =0 \,\, \mathrm{mod}\,(\pi^\ell) \,\, \mathrm{ and }\,\, x^2+\tilde{\beta} x y +\tilde{\alpha} y^2 =1 \,\, \mathrm{mod}\, (\pi^\ell)  \}|.$
\begin{lemma}\label{lem:repres. of CSL/KL}
	If $\mathbb M_A = C_{\SL_2(\cO_{2\ell})} (\psi_{A})\langle e_\lambda\rangle$ for some  $\lambda \in \mathrm{h}_{\tilde{A}}^{\ell}\setminus \pi^\ell\cO_r ,$ then there are exactly $2\theta_\lambda$ (resp. $\frac{|C_{\SL_2(\cO_{\ell})} (A)|-\theta_\lambda}{2}$) many representations in $ \mathrm{Irr}( \mathbb M_A \mid  \psi_{[A]}) $ of dimension one (resp. two).
\end{lemma}	
\begin{proof}
	Observe that every representation in $ \mathrm{Irr}( C_{\SL_2(\cO_{2\ell})} (\psi_{A}) \mid  \psi_{[A]}) $ is a one-dimensional representation. Therefore  by Clifford Theory, to show the result, it is enough to prove
	\begin{equation}\label{eqn:theta=}
	\theta_\lambda=
	|\{\phi \in \mathrm{Irr}( C_{\SL_2(\cO_{2\ell})} (\psi_{A}) \mid  \psi_{[A]})  \mid \phi^{e_\lambda}=\phi \}|.
	\end{equation}	
	The character $ \psi_{[A]}$ extends to $\mathbb M_A = C_{\SL_2(\cO_{2\ell})} (\psi_{A})\langle e_\lambda\rangle.$ Therefore by Clifford theory, the right hand side of (\ref{eqn:theta=}) is equal to 
	$|\{\chi \in \mathrm{Irr}( C_{\SL_2(\cO_{2\ell})} (\psi_{A})/K^\ell )  \mid \chi^{e_\lambda}=\chi \}|.$
From Lemma~\ref{lem:stabilizer-form}(1), we have $C_{\SL_2(\cO_{2\ell})} (\psi_{A})=
(C_{\GL_2(\cO_{2\ell})} (\tilde{A}) M^{\ell})\cap \SL_2(\cO_{2\ell}).$ It is easy to observe that $  C_{\SL_2(\cO_{2\ell})} (\psi_{A})/K^\ell \cong  C_{\SL_2(\cO_{\ell})} (A),$ which is an abelian group.
	Consider the conjugation action of $\langle e_\lambda\rangle$ on $C_{\SL_2(\cO_{2\ell})} (\psi_{A})/K^\ell .$  Observe that  for $\chi \in \mathrm{Irr}( C_{\SL_2(\cO_{2\ell})} (\psi_{A})/K^\ell ),$ 
$$\chi^{g}(X^g)=\chi(X)\,\,\mathrm{ for}\,\,\mathrm{ all }\,\, X \in  C_{\SL_2(\cO_{2\ell})} (\psi_{A})/K^\ell, \,\, \mathrm{and}\,\,\mathrm{ for}\,\,\mathrm{ all }\,\, g \in \langle e_\lambda\rangle.$$
	Therefore by Brauer Theorem (see \cite[Theorem~6.32]{MR2270898}), we obtain that 
	$$|\{\chi \in \mathrm{Irr}( C_{\SL_2(\cO_{2\ell})} (\psi_{A})/K^\ell )  \mid \chi^{e_\lambda}=\chi \}|=|\{X \in  C_{\SL_2(\cO_{2\ell})} (\psi_{A})/K^\ell  \mid X^{e_\lambda}=X \}|.$$
Since   $  C_{\SL_2(\cO_{2\ell})} (\psi_{A})\cong  C_{\SL_2(\cO_{2\ell})} (\tilde{A}) \mod(\pi^\ell),$ we have 
	$$|\{X \in  C_{\SL_2(\cO_{2\ell})} (\psi_{A})/K^\ell  \mid X^{e_\lambda}=X \}|=|\{Z \in  C_{\SL_2(\cO_{2\ell})} (\tilde{A}) \mid 
e_\lambda Ze_\lambda^{-1}=Z 
\,\, \mathrm{mod}\,(\pi^\ell) \}|.$$
By Lemma~\ref{lem:centralizer-form}(2), 
$C_{\SL_2(\cO_{2\ell})} (\tilde{A})=\{x I + y\tilde{A} \mid x,y \in \cO_{2\ell}, \, \det(x I + yA)=x^2+\beta x y +\alpha y^2 =1 \}.$
	For $Z=x I + y\tilde{A} \in  C_{\SL_2(\cO_{2\ell})} (\tilde{A}),$ by Lemma~\ref{centralizer-operations}(2), we have $e_\lambda Ze_\lambda^{-1}=Z \,\, \mathrm{mod}\,(\pi^\ell)$ if and only if $\lambda y =0 \,\, \mathrm{mod}\,(\pi^\ell).$ Therefore 
$|\{Z \in  C_{\SL_2(\cO_{2\ell})} (\tilde{A}) \mid e_\lambda Ze_\lambda^{-1}=Z\,\, \mathrm{mod}\,(\pi^\ell) \}|=\theta_\lambda.$
 This completes the proof of the lemma.
\end{proof}
\begin{lemma}\label{lem:calculation of Delta1 and 2}
	\begin{enumerate}
		\item If $\mathbb M_A = C_{\SL_2(\cO_{2\ell})} (\psi_{A}),$ then 
		$|\Delta_1|=\frac{|C_{\SL_2(\cO_{2\ell})} (\psi_{A})|^2}{|K^\ell|\times |C_{\SL_2(\cO_{2\ell})} (\psi_{[A]})|}$ and $|\Delta_2|=0.$
		\item If $\mathbb M_A = C_{\SL_2(\cO_{2\ell})} (\psi_{A})\langle e_\lambda\rangle$ for some  $\lambda \in \mathrm{h}_{\tilde{A}}^{\ell}\setminus \pi^\ell\cO_{2\ell} ,$ then 	$|\Delta_1|=2\theta_\lambda \times \frac{|\mathbb M_A|}{|C_{\SL_2(\cO_{2\ell})} (\psi_{[A]})|}$ and $|\Delta_2|=\frac{|C_{\SL_2(\cO_{\ell})} (A)|-\theta_\lambda}{2}\times \frac{|\mathbb M_A|}{|C_{\SL_2(\cO_{2\ell})} (\psi_{[A]})|}.$
	\end{enumerate}
\end{lemma}
\begin{proof}
	Let $\phi \in  \mathrm{Irr}( \mathbb M_A\mid  \psi_{[A]}) $. Let $\mathrm{Orb}(\phi) \subseteq  \mathrm{Irr}(\mathbb {M}_A\mid  \psi_{[A]})$ be the orbit of $\phi$ under the conjugation action of $C_{\SL_2(\cO_{2\ell})}(\psi_{[A]})$.   By Theorem~\ref{main_theorem-2}(3), we  have that the induced representation  $\mathrm{Ind}_{\mathbb M_A}^{\SL_2(\cO_{2\ell})} (\phi) $ is irreducible. In particular,  $\mathrm{Ind}_{\mathbb M_A}^{C_{\SL_2(\cO_{2\ell})} (\psi_{[A]})} (\phi) $ is irreducible.  Therefore by Clifford Theory, the stabilizer of $\phi$ in $C_{\SL_2(\cO_{2\ell})} (\psi_{[A]})$ is equal to $\mathbb M_A.$ Hence  $ |\mathrm{Orb}(\phi)|= \frac{|C_{\SL_2(\cO_{2\ell})} (\psi_{[A]})|}{|\mathbb M_A|}.$
Now $(1)$ follows from the fact that every representation in $\mathrm{Irr}(C_{\SL_2(\cO_{2\ell})} (\psi_{A})\mid  \psi_{[A]})$ has dimension one and $|\mathrm{Irr}(C_{\SL_2(\cO_{2\ell})} (\psi_{A})\mid  \psi_{[A]})|=\frac{|C_{\SL_2(\cO_{2\ell})} (\psi_{A})|}{|K^\ell|}$ and $(2)$  follows from Lemma~\ref{lem:repres. of CSL/KL}.
\end{proof}

\subsection{$\SL_2(\cO_{ 2})$ with $\cO = \mathbb{Z}_2 \,\,\mathrm{or}\,\, \mathbb F_2 \llbracket t \rrbracket $}\label{subsec:example-1}
\subsubsection{Orbits} For both $\cO = \mathbb{Z}_2$ and $\cO= \mathbb F_2 \llbracket t \rrbracket ,$  $\cO_1\cong \mathbb F_2.$ Therefore the orbits are same in these cases.
From Section~\ref{invertible-orbits-function-field}, the orbit representatives for $A \in \Sigma$ with $\mathrm{trace}(A)$ invertible are given by the set
$\{ (1,0,1), (1,1,1) \}.$ Recall that we have identified $A = \left[ \begin{matrix}
0 & a^{-1} \alpha \\ a & \beta 
\end{matrix} \right]$ with corresponding tuple $(a, \alpha, \beta).$
Now (\ref{condition-1}), (\ref{condition-2}) and  (\ref{condition-3}) imply  that there is exactly one orbit representative for $A \in \Sigma$ with $\mathrm{trace}(A)=0.$ Choose $(1,0,0)$ as a representative in this case. 
\begin{proposition}\label{prop:example-O2}
	For  $A \in \Sigma,$ the one-dimensional representation $\psi_{[A]}$ extends to $C_{\SL_2(\cO_{2})} (\psi_{[A]}).$ 
\end{proposition}
\begin{proof}
	Let $A=(a,\alpha, \beta) \in \Sigma.$ For   $\beta=0,$
\begin{eqnarray*}
\mathrm{h}_{\tilde{A}}^{1} &=& \{ x \in \cO_{2 } \mid 2x = 0  \,\, \mathrm{mod}\, (\pi), \, \,  x(x+\tilde \beta) = 0 \,\, \mathrm{mod}\, (\pi ) \}\\
&=& \{ x \in \cO_{2 } \mid   x^2 = 0 \,\, \mathrm{mod}\, (\pi ) \}\\
&=&\pi \cO_2.
\end{eqnarray*}
By Lemma~\ref{lem:stabilizer-form}(3), we obtain that 
 $C_{\SL_2(\cO_{2})} (\psi_{[A]})=C_{\SL_2(\cO_{2})} (\psi_{A}). $
Therefore in this case, the result follows by Lemma \ref{lem:centralizer-sl}.
	For   $\beta$ invertible, we have $\mathrm{h}_{\tilde{A}}^{1}=\{0,\tilde \beta\}+\pi \cO_2,$	where $\tilde \beta$ is a fixed lift of $\beta.$ By Lemma~\ref{lem:stabilizer-form}(3), 
$C_{\SL_2(\cO_{2})} (\psi_{[A]})=C_{\SL_2(\cO_{2})} (\psi_{A})\langle e_{\tilde \beta}\rangle. $
Therefore to show the result, it is enough to show that $E_{\tilde \beta, \tilde A}=E_{\tilde \beta, \tilde A}^\circ$ (by Theorem~\ref{thm:condition for extension}).
	Let $(x,y)\in E_{\tilde \beta, \tilde A}$. By definition of $E_{\tilde \beta, \tilde A},$ $x=1 \,\, \mathrm{mod}\,(\pi)$ and $y=0 \,\, \mathrm{mod}\,(\pi).$ 
This implies  $x^2=1$ and $y^2=0.$ 
Therefore
$$\psi(f(\tilde \beta,x,y))=\psi(xy\tilde \beta (\tilde \beta- \tilde \beta) -\tilde \alpha \tilde \beta y^2 +\tilde \beta( x^2-1) ) =\psi(0)= 1 .$$
 Hence $E_{\tilde \beta, \tilde A}\subseteq E_{\tilde \beta, \tilde A}^\circ.$ By definition we have $E_{\tilde \beta, \tilde A}^\circ \subseteq E_{\tilde \beta, \tilde A}.$ Therefore $E_{\tilde \beta, \tilde A}=E_{\tilde \beta, \tilde A}^\circ.$
\end{proof}

Note that we have not used Theorem~\ref{main_theorem-2} to prove Proposition~\ref{prop:example-O2}.
Recall that Theorem~\ref{main_theorem-2} was proved for $r \geq R_\cO.$  In particular Theorem~\ref{main_theorem-2} was proved  for $\cO = \mathbb Z_2$ with $r \geq 4.$ However, for $\mathbb Z_2$ with $r =2,$ we have  seen that 
$$\mathbb M_A=C_{\SL_2(\cO_{2})} (\psi_{[A]})=\begin{cases}
C_{\SL_2(\cO_{2})} (\psi_{A})& \beta =0,\\
C_{\SL_2(\cO_{2})} (\psi_{A})\langle e_{\tilde \beta}\rangle & \beta \in \cO_1^\times, 
\end{cases}$$
 satisfies all the properties as mentioned in Theorem~\ref{main_theorem-2}.  Therefore we can make use of Lemmas~\ref{lem:repres. of CSL/KL} and \ref{lem:calculation of Delta1 and 2} in that case also. We encode obtained information in this case in Table \ref{table1:SL2-O-2}. Next, Table \ref{table2:SL2-O-2} contains the information about the numbers and dimensions of primitive irreducible representations of $\SL_2(\cO_2).$
\begin{table}[h]
\centering
	\begin{tabular}{|c|c|c|c|c|c|c|}
		\hline
		Orbit Rep.  & $|C_{\SL_2(\cO_1)}(A)|$ & $|\mathbb M_A|$&$\theta_\lambda$&$|\Delta_1|$&$|\Delta_2|$& $\frac{|\SL_2(\cO_2)|}{|\mathbb M_A|}$\\
		\hline
		$(1,0,1)$&$1$&$16$&$1$&$2$&$0$&$3$\\
		\hline
		$(1,1,1)$&$3$&$3\times 16$&$1$&$2$&$1$&$1$\\
		\hline
		$(1,0,0)$&$2$&$16$&--&$2$&$0$&$3$\\
		\hline
	\end{tabular}
\caption{}
\label{table1:SL2-O-2}
\end{table}

\begin{table}[h]
\centering
	\begin{tabular}{|c|c|c|c|}
		\hline
		\multirow{2}{*}{Orbit Rep.}  & 	\multicolumn{3}{c|}{\# Irred. Rep. of $\SL_2(\cO_2)$ with} \\
		\cline{2-4}
		&dim. $=1$ &dim.  $=2$&dim. $=3$\\
		\hline
		$(1,0,1)$&$0$&$0$&$2$\\
		\hline
		$(1,1,1)$&$2$&$1$&$0$\\
		\hline
		$(1,0,0)$&$0$&$0$&$2$\\
		\hline
		Total \# &$2$&$1$&$4$\\
		\hline
	\end{tabular}
\caption{For primitive representations of $\SL_2(\cO_2)$}
\label{table2:SL2-O-2}
\end{table}
From above, combined with the general methods of constructing representations of $\SL_2(\cO_2)$ and the results of groups algebra $\mathbb C[\SL_2(\cO_1)],$ we see that $\mathbb C[\SL_2(\mathbb Z/2^{2}\mathbb Z)]\cong \mathbb C[\SL_2(\mathbb F_2[t]/(t^{2}))].$

\subsection{ $\SL_2(\cO_{ 4})$ with $\cO = \mathbb{Z}_2$}\label{case l=2 and q=2}

\subsubsection{Orbits}
From Section~\ref{invertible-orbits-number-field}, the orbit representatives for $A \in \Sigma$ with $\mathrm{trace}(A)$ invertible,  is given by the set 
$\{ (1,0,1), (1,1,1) \}.$
Note that $ \pi^{\ee}\cO_2=\pi \cO_2.$ Also, observe that 
\begin{equation}\label{determinant-O2}
\{x^2-\alpha y^2 \mid x,y \in \cO_{2}\}\cap \cO_2^\times= \begin{cases}\{1\} &   \mathrm{if}\,\, \alpha =0,3, \\
\{1,3\} &  \mathrm{if}\,\, \alpha =1,2  .\end{cases} 
\end{equation}
Therefore by Section~\ref{non-invertible-orbits-number-field}, the orbit representatives for $A \in \Sigma$ with $\mathrm{trace}(A)$ non-invertible  are given by the set 
$\{ (1,0,0), (3,0,0), (1,1,0),(1,2,0), (1,3,0), (3,3,0)\}.$

\begin{proposition}
Let $A=(a,\alpha,\beta) \in \Sigma.$ Then 
	\begin{enumerate}
		\item $C_{\SL_2(\cO_{4})} (\psi_{[A]})=\begin{cases} C_{\SL_2(\cO_{4})} (\psi_{A})&  \mathrm{for}\,\, \beta \in \cO_2^\times,\\
		C_{\SL_2(\cO_{4})} (\psi_{A})\langle e_2\rangle &  \mathrm{for}\,\, \beta \in 2 \cO_2.
		\end{cases}$
		\item For $A \in\{ (1,0,1), (1,1,1)\}\cup \{ (1,1,0), (1,3,0), (3,3,0)\},$ $\mathbb M_A =C_{\SL_2(\cO_{4})} (\psi_{[A]}).$
		\item  For $A \in\{ (1,0,0), (3,0,0),(1,2,0)\},$ $\mathbb M_A = C_{\SL_2(\cO_{4})} (\psi_{A}).$
	\end{enumerate}
	
\end{proposition}
\begin{proof}
	(1) will easily follows from Lemma~\ref{lem:stabilizer-form} and  Proposition~\ref{prop:charactrization of h_Atilda^ell and h_Atilda^ell'}(2). Now we proceed to prove  (2) and (3).
	For $A \in\{ (1,0,1), (1,1,1)\},$ the result follows from (1) and the fact that $C_{\SL_2(\cO_{4})} (\psi_{A})\subseteq \mathbb {M}_A \subseteq C_{\SL_2(\cO_{4})} (\psi_{[A]}).$
 Let  $A \in \{ (1,1,0),  (1,3,0), (3,3,0)\}\cup \{ (1,0,0), (3,0,0),(1,2,0)\}.$ Note that  $\beta=0$ always. Therefore by Theorem~\ref{S_A-in-number}(2), we obtain the following: 
\[
E_{\tilde{A}} = \begin{cases} \{ 0,  2 \}+ 2^2 \cO_4&  \mathrm{for}\,\, \alpha = 1 \,\, \mathrm{mod}\,(2),\\
2^2 \cO_4&   \mathrm{for}\,\, \alpha \neq1  \,\, \mathrm{mod}\,(2).
\end{cases}
	\]
Here we used the fact that $\xi=1$ for $\cO = \mathbb{Z}_2.$	
The result now follows from (1).
\end{proof}

 We encode obtained information from  Lemmas~\ref{gl-centralizer-cardinality}, \ref{lem:repres. of CSL/KL} and \ref{lem:calculation of Delta1 and 2} in this case in Table \ref{table1:SL2-O-4n}. Next, Table \ref{table2:SL2-O-4n} contains the information about the numbers and dimensions of primitive irreducible representations of $\SL_2(\mathbb Z/2^{4}\mathbb Z)$.

\begin{table}[h]
\centering
	\begin{tabular}{|c|c|c|c|c|c|c|}
		\hline
		Orbit Rep.  & $|C_{\SL_2(\mathbb Z/2^{2}\mathbb Z)}(A)|$ & $|\mathbb M_A|$&$\theta_\lambda$&$|\Delta_1|$&$|\Delta_2|$& $\frac{|\SL_2(\mathbb Z/2^{4}\mathbb Z)|}{|\mathbb M_A|}$\\
		\hline
		$(1,0,1)$&$2$&$2^7$&--&$2$&$0$&$3\times 2^3$\\
	\hline
	$(1,1,1)$&$3\times 2$&$3\times 2^7$&--&$3\times 2$&$0$&$2^3$\\
	\hline
	$(1,1,0)$&$2^2$&$2^9$&$2^2$&$2^3$&$0$&$3\times 2$\\
	\hline
	$(1,2,0)$&$2^2$&$2^8$&--&$2$&$0$&$ 3\times 2^2$\\
	\hline
	$(1,3,0),(3,3,0)$&$2^3$&$2^{10}$&$2^2$&$2^3$&$2$&$3$\\
	\hline
	$(1,0,0),(3,0,0)$&$2^3$&$2^9$&--&$2^2$&$0$&$3\times 2$\\
	\hline
\end{tabular}
\caption{}
\label{table1:SL2-O-4n}
\end{table}

\begin{table}[h]
\centering
	\begin{tabular}{|c|c|c|c|c|c|c|}
		\hline
		\multirow{2}{*}{Orbit Rep.}  &	\multirow{2}{*}{\#Orbits} & 	\multicolumn{5}{c|}{\# Irred. Rep. of $SL_2(\mathbb Z/2^{4}\mathbb Z)$ with dim. $=$} \\
		\cline{3-7}
		&& $2^3$ & $3$& $3\times 2$& $3\times 2^2$& $3\times 2^3$\\
		\hline
		$(1,0,1)$&$1$&$0$&$0$&$0$&$0$&$2$\\
		\hline
		$(1,1,1)$&$1$&$6$&$0$&$0$&$0$&$0$\\
		\hline
		$(1,1,0)$&$1$&$0$&$0$&$8$&$0$&$0$\\
		\hline
		$(1,2,0)$&$1$&$0$&$0$&$0$&$2$&$0$\\
		\hline
		$(1,3,0),(3,3,0)$&$2$&$0$&$8$&$2$&$0$&$0$\\
		\hline
		$(1,0,0),(3,0,0)$&$2$&$0$&$0$&$4$&$0$&$0$\\
		\hline
	\multicolumn{2}{|c|}{Total \# }	& $6$ & $16$ & $20$ & $2$ & $2$\\
		\hline
	\end{tabular}
\caption{For primitive representations of $\SL_2(\mathbb Z/2^{4}\mathbb Z)$}
\label{table2:SL2-O-4n} 
\end{table}

\subsection{  $\SL_2(\cO_{ 4})$ with $ \cO=\mathbb F_2 \llbracket t \rrbracket $}

\subsubsection{Orbits}
From Section~\ref{invertible-orbits-function-field}, the orbit representatives for $A \in \Sigma$ with $\mathrm{trace}(A)$ invertible are
$ \bigcup_{\beta \in \cO_2^\times }\{ (1,0,\beta), (1,1,\beta) \}.$
Now for $\beta \in t \cO_2,$ observe that 
$$| \{x^2+\beta x \mid x \in\cO_2\} |=2 \,\, \mathrm{ and }\,\, \{x^2+\beta x \mid x \in\cO_2\} = \mathbb F_2 \,\, \mathrm{mod}\,(t).$$  
Therefore by equations (\ref{condition-1}), (\ref{condition-2}) and (\ref{condition-3}), all orbit representatives for $A \in \Sigma$ with $\mathrm{trace}(A)$ non-invertible are contained in 
$W:=\{ (a,\alpha,\beta) \mid a \in \cO_2^\times\,\, \mathrm{ and }\,\, \alpha, \beta \in \pi \cO_2 \}.$
Also, note that  for  $ (a,\alpha,\beta), (a',\alpha',\beta')\in W,$ $[ (a,\alpha,\beta)]=[ (a',\alpha',\beta')]$ implies $\alpha'=\alpha$ and $\beta'=\beta.$
Now for $\alpha= t u$ and $\beta=t v,$ Observe that

$$ \{x^2+\beta xy +\alpha y^2 \mid x,y \in \cO_2\}\cap \cO_2^\times=\begin{cases}\mathbb F_2^\times &   \mathrm{if}\,\,  u, v \in t \cO_2,\\
\mathbb F_2^\times \{ 1+\pi ( v y +u y^2) \mid y \in \cO_2\} &  \mathrm{if}\,\,  u,v \in  \cO_2^\times, \\
\cO_2^\times&   \mathrm{otherwise}.
\end{cases}
$$
Therefore by equations (\ref{condition-1}), (\ref{condition-2}) and (\ref{condition-3}), all orbit representatives for $A \in \Sigma$ with $\mathrm{trace}(A)$ non-invertible are: 
\begin{enumerate}
	\item $\{ (1 ,0,0), (1+t ,0,0)  \} ;$
	\item $\{ (1,t ,0),(1,0,t )   \};$
	\item $ \{ (1,t ,t ),(1+t ,t ,t )\}.$
\end{enumerate}
We say $A$ is of type $(i)$ if $A$ is in the $(i)^{th}$  set above.
\begin{proposition}
	For  $A=(a,\alpha,\beta) \in \Sigma,$ 
	\begin{enumerate}
		\item $C_{\SL_2(\cO_{4})} (\psi_{[A]})=\begin{cases} C_{\SL_2(\cO_{4})} (\psi_{A})\langle e_{\tilde \beta}\rangle &  \mathrm{for}\,\, \beta \in \cO_2^\times,\\
		C_{\SL_2(\cO_{4})} (\psi_{A}) \langle e_t \rangle &  \mathrm{for}\,\, \beta \in t \cO_2,
		\end{cases}$\\
		where $\tilde \beta$ is a fixed lift of $\beta.$
		
		\item $\mathbb M_A=\begin{cases}
		C_{\SL_2(\cO_{4})} (\psi_{[A]})&  \mathrm{if}\,\, \beta\in  (\cO_2^\times)^2 =\{1\},\\
		C_{\SL_2(\cO_{4})} (\psi_{A})&  \mathrm{otherwise}.
		\end{cases}$
	\end{enumerate}
	
\end{proposition}
\begin{proof}
	(1) follows from Lemma~\ref{lem:stabilizer-form} and  Proposition~\ref{prop:charactrization of h_Atilda^ell and h_Atilda^ell'}(1). To prove  (2), first assume $ \beta \in \cO_2^\times.$ Then 
	note that for  $\lambda= \tilde \beta \in \cO_4^\times,$ we have $\bol i_\lambda =\bol k =0<1=\bol s  $ and $\bol j_\lambda =\ell'=2.$ Therefore $\tilde \beta$ satisfies  \textit{(II)} and \textit{(III)} of Theorem~\ref{thm:simplification of extension conditions}. Now by \textit{(I)} of Theorem~\ref{thm:simplification of extension conditions},
	$\tilde \beta \in E_{ \tilde A}$ if and only if 
	$\beta\in  (\cO_2^\times)^2.$ Therefore for $ \beta \in \cO_2^\times,$ (2) follows from the fact that $\mathbb M_A = C_{\SL_2(\cO_{2\ell})}(\psi_{A}) \mathbb E_{\tilde{A}}.$
For $\beta \in t \cO_2,$ observe that  \textit{(I)} of Theorem~\ref{thm:simplification of extension conditions} is not satisfied for $\lambda=t.$ Therefore $t \notin E_{ \tilde A}.$ Hence (2) holds for  $\beta \in t \cO_2.$
\end{proof}
The obtained information from  Lemmas~\ref{lem:repres. of CSL/KL} and \ref{lem:calculation of Delta1 and 2} in this case is given in Table~\ref{table1:SL2-O-4f}. Table \ref{table2:SL2-O-4f} contains the information about the numbers and dimensions of primitive irreducible representations of $\SL_2(\mathbb F_2[t]/(t^{4}))$. 

\begin{table}[h]
\centering
		\begin{tabular}{|c|c|c|c|c|c|c|}
			\hline
			Orbit Rep.  & $|C_{\SL_2(\mathbb F_2[t]/(t^{2}))}(A)|$ & $|\mathbb M_A|$&$\theta_\lambda$&$|\Delta_1|$&$|\Delta_2|$& $\frac{|\SL_2(\mathbb F_2[t]/(t^{4}))|}{|\mathbb M_A|}$\\
		\hline
			$(1,0,1)$&$2$&$2^8$&$2$&$2^2$&$0$&$3\times 2^2$\\
			\hline
			$(1,1,1)$&$3\times 2$&$3\times 2^8$&$2$&$2^2$&$2$&$2^2$\\
			\hline
			$(1,0,1+t)$&$2$&$2^7$&--&$1$&$0$&$3\times 2^3$\\
			\hline
			$(1,1,1+t)$&$3\times 2$&$3\times 2^7$&--&$3$&$0$&$2^3$\\
			\hline
			type $(2)$&$2^2$&$2^8$&--&$2$&$0$&$3\times 2^2$\\
			\hline
			type $(1)$ or $(3)$ &$2^3$&$2^9$&--&$2^2$&$0$&$3\times 2$\\
			\hline
		\end{tabular}
\caption{}
\label{table1:SL2-O-4f}
\end{table}
	
\begin{table}[h]
\centering
		\begin{tabular}{|c|c|c|c|c|c|c|}
			\hline
			\multirow{2}{*}{Orbit Rep.} &  	\multirow{2}{*}{\#Orbits} & 	\multicolumn{5}{c|}{\# Irred. Rep. of $SL_2(\mathbb F_2[t]/(t^{4}))$ with dim.$=$} \\
				\cline{3-7}
			& & $2^2$& $2^3$ & $3\times 2$& $3\times 2^2$& $3\times 2^3$\\
			\hline
			$(1,0,1)$&$1$&$0$&$0$&$0$&$4$&$0$\\
			\hline
			$(1,1,1)$&$1$&$4$&$2$&$0$&$0$&$0$\\
			\hline
			$(1,0,1+t)$&$1$&$0$&$0$&$0$&$0$&$1$\\
			\hline
			$(1,1,1+t)$&$1$&$0$&$3$&$0$&$0$&$0$\\
			\hline
			type $(2)$&$2$&$0$&$0$&$0$&$2$&$0$\\
			\hline
			type $(1)$ or $(3)$&$4$&$0$&$0$&$4$&$0$&$0$\\
			\hline
				\multicolumn{2}{|c|}{Total \# }& $4$ & $5$ & $16$ & $8$ & $1$ \\
			\hline
		\end{tabular}
\caption{For primitive representations of $\SL_2(\mathbb F_2[t]/(t^{4}))$}
\label{table2:SL2-O-4f}	
\end{table}

\subsection{ $\SL_2(\cO_{6})$ with $\cO = \mathbb{Z}_2$}
\subsubsection{Orbits}
From Section~\ref{invertible-orbits-number-field}, the orbit representatives for $A \in \Sigma$ with $\mathrm{trace}(A)$ invertible are $\{ (1,0,1),(1,1,1),(1,2,1),(1,3,1)\}.$

Now observe that 
\begin{equation}\label{determinant-O3}
\{x^2-\alpha y^2 \mid x,y \in \cO_{3}\}\cap \cO_3^\times= \begin{cases}\{1,1-\alpha\} &  \mathrm{if}\,\, \alpha \in 2\cO_3, \\
\{1,5,-\alpha,4-\alpha\} & \mathrm{if }\,\,\alpha \in \cO_3^\times .  \end{cases} 
\end{equation}
Therefore from Section~\ref{non-invertible-orbits-number-field}, the orbit representatives for $A \in \Sigma$ with $\mathrm{trace}(A)$ non-invertible are:
\begin{enumerate}
	\item $\{ (1,0,0), (3,0,0), (5,0,0), (7,0,0)\};$
	\item $\{ (1,2,0), (5,2,0)\};$
	\item $\{ (1,4,0), (3,4,0)\};$
	\item $\{ (1,6,0), (5,6,0)\};$
	\item $\{ (1,1,0), (1,5,0)\};$
	\item $\{  (1,3,0),  (3,3,0), (1,7,0),(3,7,0)\}.$
\end{enumerate}
As before, we say $A$ is of type $(i)$ if $A$ is in the $(i)^{th}$  set above.

\begin{proposition}
	For  $A=(a,\alpha,\beta) \in \Sigma,$ 
	\begin{enumerate}
		\item $C_{\SL_2(\cO_{6})} (\psi_{[A]})=\begin{cases} C_{\SL_2(\cO_{6})} (\psi_{A}) &  \mathrm{for}\,\, \beta \in \cO_3^\times,\\
		C_{\SL_2(\cO_{6})} (\psi_{A})\langle e_{2^2} \rangle &  \mathrm{for}\,\, \beta \in 2 \cO_3.
		\end{cases}$
		\item  $\mathbb M_A = C_{\SL_2(\cO_{6})} (\psi_{A}).$
	\end{enumerate}
	
\end{proposition}
\begin{proof}
	(1)  follows from Lemma~\ref{lem:stabilizer-form} and  Proposition~\ref{prop:charactrization of h_Atilda^ell and h_Atilda^ell'}(2). Also, (2) follows from Theorem~\ref{S_A-in-number}.	
\end{proof}
From Lemma~\ref{lem:image-det-map} and  (\ref{determinant-O3}) we have
$$|\det(C_{\GL_2 (\cO_{3})}(A))| =\begin{cases}
4 & \mathrm{for}\,\, \beta \in \cO_3^\times\,\, \mathrm{or}\,\, \mathrm{type}\,\,5,\\
1&   \mathrm{for}\,\,\mathrm{type}~1,\\
2 & \mathrm{otherwise }.
\end{cases}$$
From Lemma~\ref{gl-centralizer-cardinality}, 
$$|C_{\GL_2(\cO_{ 3})} (A)| =\begin{cases}
16 & \mathrm{for}\,\, \beta \in \cO_3^\times \,\, \mathrm{ and }\,\,\alpha \in 2\cO_3,\\
48 &   \mathrm{for}\,\, \beta \in \cO_3^\times \, \, \mathrm{ and }\,\,\alpha \in \cO_3^\times,\\
32 &  \mathrm{for}\,\,\beta \in 2\cO_3.
\end{cases}$$
Therefore 
$$|C_{\SL_2(\cO_{ 3})} (A)|=\frac{|C_{\GL_2(\cO_{ 3})} (A)|}{|\det(C_{\GL_2 (\cO_{3})}(A))|} =\begin{cases}
4 &   \mathrm{for}\,\, \beta \in \cO_3^\times \,\, \mathrm{ and }\,\,\alpha \in 2\cO_3,\\
12 &  \mathrm{for}\,\, \beta \in \cO_3^\times \,\, \mathrm{ and }\,\,\alpha \in \cO_3^\times,\\
32&  \mathrm{for}\,\,\mathrm{type}~1,\\
8 &   \mathrm{for}\,\, \mathrm{type}~5,\\
16 &   \mathrm{otherwise }.
\end{cases}$$

Tables~\ref{table1:SL2-O-6n} and \ref{table2:SL2-O-6n}, parallel to Tables~\ref{table1:SL2-O-4f} and \ref{table2:SL2-O-4f}, contain corresponding information for $\SL_2(\mathbb Z/2^{6} \mathbb Z)$.

\begin{table}[h]
\centering
		\begin{tabular}{|c|c|c|c|c|c|}
			\hline
			Orbit Rep.  & $|C_{\SL_2(\mathbb Z/2^{3} \mathbb Z)}(A)|$ & $|\mathbb M_A|$&$|\Delta_1|$&$|\Delta_2|$& $\frac{|\SL_2(\mathbb Z/2^{6} \mathbb Z)|}{|\mathbb M_A|}$\\
			\hline
			$(1,0,1),(1,2,1)$&$2^2$&$ 2^{11}$&$2^2$&$0$&$3\times 2^5$\\
		\hline
		$(1,1,1),(1,3,1)$&$3\times 2^2$&$3\times  2^{11}$&$3\times 2^2$&$0$&$2^5$\\
		\hline
		type $(1)$&$2^5$&$2^{14}$&$2^4$&$0$&$3\times 2^2$\\
		\hline
		type $(5)$&$2^3$&$2^{12}$&$2^2$&$0$&$3\times 2^4$\\
		\hline
		type $(2),$$(3),$$(4)$ \& $(6)$&$2^4$&$2^{13}$&$2^3$&$0$&$3\times 2^3$\\
			\hline
		\end{tabular}
\caption{}
\label{table1:SL2-O-6n}
\end{table}

\begin{table}[h]
\centering
		\begin{tabular}{|c|c|c|c|c|c|c|}
			\hline
			\multirow{2}{*}{Orbit Type}  & 	\multirow{2}{*}{\#Orbits} &
			\multicolumn{5}{c|}{\# Irred. Rep. of $SL_2(\mathbb Z/2^{6} \mathbb Z)$ with dim. $=$} \\
			\cline{3-7}
			&&$2^5$ &$3\times 2^2$&$3\times 2^3$ &$3\times 2^4$& $3\times 2^5$\\
			\hline
			$(1,0,1),(1,2,1)$&$2$&$ 0$&$ 0$&$ 0$&$0$&$4$\\
			\hline
			$(1,1,1),(1,3,1)$&$2$&$12$&$0$&$0$&$0$&$0$\\
			\hline
			type $(1)$&$4$&$0$&$16$&$0$&$0$&$0$\\
			\hline
			type $(5)$&$2$&$0$&$0$&$0$&$4$&$0$\\
			\hline
			type $(2),$$(3),$$(4)$ \& $(6)$&$10$&$0$&$0$&$8$&$0$&$0$\\
			\hline	
				\multicolumn{2}{|c|}{Total \# }&$24$ & $64$ & $80$ & $8$&$8$ \\
			\hline
		\end{tabular}
	\caption{For primitive representations of $\SL_2(\mathbb Z/(2^{6} \mathbb Z))$}
\label{table2:SL2-O-6n}
\end{table}

\subsection{\  $\SL_2(\cO_{6})$ with $\cO= \mathbb F_2 \llbracket t \rrbracket $}
\subsubsection{Orbits}
From Section~\ref{invertible-orbits-function-field}, the orbit representatives for $A \in \Sigma$ with $\mathrm{trace}(A)$ invertible are
$\bigcup_{\beta \in \cO_3^\times }\{ (1,0,\beta), (1,1,\beta) \}$. Now for $\beta \in t \cO_3,$ observe that 
$$ \{x^2+\beta x \mid x \in\cO_3\} =\begin{cases}\{0,1,t^2,1+t^2\} &    \mathrm{for}\,\,\beta=0,t^2,\\
\{0,1+t\}  &   \mathrm{for}\,\,\beta=t,\\
\{0,1+t+t^2\}  &   \mathrm{for}\,\,\beta=t+t^2.\\
\end{cases}$$  
Therefore by equations (\ref{condition-1}), (\ref{condition-2}) and (\ref{condition-3}), all orbit representatives for $A \in \Sigma$ with $\mathrm{trace}(A)$ non-invertible are  contained in 
$$W:=\bigcup_{a \in \cO_3^\times }\left( \{ (a,0,\beta), (a,t,\beta)\}_{\beta =0,t^2}\cup \{(a,0,\beta),(a,t,\beta),(a,t^2,\beta),(a,t+t^2,\beta) \}_{\beta =t, t+t^2}\right).$$
Also, note that  for  $ (a,\alpha,\beta), (a',\alpha',\beta')\in W,$ $[ (a,\alpha,\beta)]=[ (a',\alpha',\beta')]$ implies $\alpha'=\alpha$ and $\beta'=\beta.$
Now for $ (a,\alpha,\beta)\in W,$ observe that
$$ \{x^2+\beta xy +\alpha y^2 \mid x,y \in \cO_3\}\cap \cO_3^\times=\begin{cases}(\cO_3^\times)^2 &    \mathrm{if}\,\,\mathrm{both}\,\, \alpha, \beta\,\, \mathrm{ are}\,\, \mathrm{ either}\,\, \mathrm{ in }\,\, t^2 \cO_3 \,\,  \mathrm{ or}\,\,\mathrm{ in }\,\, t\cO_3^\times,\\
\cO_3^\times&   \mathrm{otherwise}.
\end{cases}
$$
Therefore by equations (\ref{condition-1}), (\ref{condition-2}) and (\ref{condition-3}), all orbit representatives for $A \in \Sigma$ with $\mathrm{trace}(A)$ non-invertible are: 
\begin{enumerate}
	\item $\{ (1,0,0) ,(1+t,0,0)\};$
	\item  $ \{(1,0,t^2) ,(1+t,0,t^2)\}; $
	\item  $\{ (1,t,0),(1,t,t^2)\};$
	\item $\{ (1,0,\beta),(1,t^2,\beta)\mid \beta=t,t+t^2\};$
	\item $\{ (1,t,\beta),(1+t,t,\beta),(1,t+t^2,\beta),(1+t,t+t^2,\beta)\mid \beta=t,t+t^2\}  .$
\end{enumerate}
Again, we say $A$ is of type $(i)$ if $A$ is in the $(i)^{th}$  set above.	 
\begin{proposition}
	For  $A=(a,\alpha,\beta) \in \Sigma,$ 
	\begin{enumerate}
		\item $C_{\SL_2(\cO_{r})} (\psi_{[A]})=\begin{cases} C_{\SL_2(\cO_{r})} (\psi_{A})\langle e_{\tilde \beta}\rangle &  \mathrm{for}\,\,\beta \in \cO_3^\times,\\
		C_{\SL_2(\cO_{r})} (\psi_{A})\{e_z \mid z=0,\tilde \beta,t^2, \tilde \beta+t^2 \} &   \mathrm{for}\,\,\beta \in \pi \cO_3,
		\end{cases}$\\
		where $\tilde \beta$ is a fixed lift of $\beta.$
		\item For $\beta \in \cO_3^\times,$ 
		$ \mathbb M_A = \begin{cases}
		C_{\SL_2(\cO_{r})} (\psi_{[A]})  &  \mathrm{if}\,\,\beta \in (\cO_3^\times)^2,\\
		C_{\SL_2(\cO_{r})} (\psi_{A}) &   \mathrm{if}\,\,\beta \notin (\cO_3^\times)^2.
		\end{cases}$
		\item For $\beta \in t\cO_3,$
		$ \mathbb M_A = \begin{cases}
		C_{\SL_2(\cO_{r})} (\psi_{A})\langle e_{t^2}\rangle &  \mathrm{if}\,\,A \,\, \mathrm{is}\,\, \mathrm{of}\,\, \mathrm{type}\,\, (2)\,\, \mathrm{or}\,\, (3),\\
		C_{\SL_2(\cO_{r})} (\psi_{A}) &  \mathrm{if}\,\,A\,\, \mathrm{is}\,\, \mathrm{of}\,\, \mathrm{type}\,\,(1), \,\, (4)\,\,\mathrm{or}\,\, (5).
		\end{cases}$
	\end{enumerate}
\end{proposition}
\begin{proof}
	(1) will  follows from Lemma~\ref{lem:stabilizer-form} and  Proposition~\ref{prop:charactrization of h_Atilda^ell and h_Atilda^ell'}(1). To prove  (2),  
	note that for  $\lambda= \tilde \beta \in \cO_6^\times,$ we have $\bol i_\lambda =\bol k =0<1=\bol s  $ and $\bol j_\lambda =\ell'=3.$ Therefore $\tilde \beta$ satisfies  \textit{(II)} and \textit{(III)} of Theorem~\ref{thm:simplification of extension conditions}. 
	Now by \textit{(I)} of Theorem~\ref{thm:simplification of extension conditions}, 
	$\tilde \beta \in E_{\tilde A}$ if and only if 
	$\beta\in  (\cO_3^\times)^2.$ Therefore  (2) follows.
We now proceed to prove  (3). For $A$ of type $(1)$,$(4)$ or $(5),$  there is no $z  \in\{\tilde \beta,t^2, \tilde \beta+t^2 \}\setminus\{0\}$ satisfying both  \textit{(I)} and \textit{(II)} of Theorem~\ref{thm:simplification of extension conditions}. Therefore $E_{\tilde A}=t^3 \cO_6.$ Hence the result follows in this case. 
For $A$ of type $(2)$ or $(3),$  $\beta =t^2.$ Therefore by (1), we obtain that 
	$C_{\SL_2(\cO_{r})} (\psi_{[A]})=C_{\SL_2(\cO_{r})} (\psi_{A})\langle e_{t^2}\rangle.$ Now take $\lambda =t^2.$ It is easy to  show that this $\lambda$ satisfies \textit{(I)}, \textit{(II)} and \textit{(III)} of Theorem~\ref{thm:simplification of extension conditions}. Therefore  $E_{\tilde A}=\{0, t^2\}+t^3\cO_6$. Hence the result follows. 
\end{proof}
We encode obtained information from  Lemmas~\ref{lem:repres. of CSL/KL} and \ref{lem:calculation of Delta1 and 2} in Table \ref{table1:SL2-O-6f}. Next, Table \ref{table2:SL2-O-6f} contains the information about the numbers and dimensions of primitive irreducible representations of $\SL_2(\mathbb F_2[t]/(t^{6}))$.

\begin{table}[h]
\centering
		\begin{tabular}{|c|c|c|c|c|c|c|}
			\hline
			Orbit Rep.  & $|C_{\SL_2(\mathbb F_2[t]/(t^{3}))}(A)|$ & $|\mathbb M_A|$&$\theta_\lambda$&$|\Delta_1|$&$|\Delta_2|$& $\frac{|\SL_2(\mathbb F_2[t]/(t^{6}))|}{|\mathbb M_A|}$\\
			\hline
				$(1,0,\beta)$; $\beta = 1, 1+t^2$&$2^2$&$2^{12}$&$2$&$2^2$&$1$&$3\times 2^4$\\
			\hline
			$(1,0,\beta)$; $\beta = 1+t, 1+t+t^2$&$2^2$&$2^{11}$&--&$2$&$0$&$3\times 2^5$\\
			\hline
			$(1,1,\beta)$; $\beta = 1, 1+t^2$&$3\times 2^2$&$3\times 2^{12}$&$2$&$2^2$&$5$&$ 2^4$\\
			\hline
			$(1,1,\beta)$; $\beta = 1+t, 1+t+t^2$&$3\times 2^2$&$3\times 2^{11}$&--&$3\times 2$&$0$&$ 2^5$\\
			\hline
			$(1)$&$2^4$&$2^{13}$&--&$2^3$&$0$&$3\times 2^3$\\
			\hline
			$(2)$&$2^4$&$2^{14}$&$2^{3}$&$2^4$&$2^2$&$3\times 2^2$\\
			\hline
			$(3)$&$2^3$&$2^{13}$&$2^{3}$&$2^4$&$0$&$3\times 2^3$\\
			\hline
			$(4)$&$2^3$&$2^{12}$&--&$2$&$0$&$3\times 2^4$\\
			\hline
			$(5)$&$2^4$&$2^{13}$&--&$2^2$&$0$&$3\times 2^3$\\
				\hline
		\end{tabular}
\caption{}
\label{table1:SL2-O-6f}
\end{table}
\begin{table}[h]
\centering
		\begin{tabular}{|c|c|c|c|c|c|c|c|}
			\hline	
			\multirow{2}{*}{Orbit Type}  & 	\multirow{2}{*}{\#Orbits} &
			\multicolumn{6}{c|}{\# Irred. Rep. of $SL_2(\mathbb F_2[t]/(t^{6}))$ with dim. $=$} \\
			\cline{3-8}
			&&$2^4$ &$2^5$&$3\times 2^2$ &$3\times 2^3$& $3\times 2^4$& $3\times 2^5$\\
			\hline
			$(1,0,\beta)$; $\beta = 1, 1+t^2$&$2$&$ 0$&$ 0$&$ 0$&$0$&$4$&$1$\\
			\hline
			$(1,0,\beta)$; $\beta = 1+t, 1+t+t^2$&$2$&$0$&$0$&$0$&$0$&$0$&$2$\\
			\hline
			$(1,1,\beta)$; $\beta = 1, 1+t^2$&$2$&$4$&$5$&$0$&$0$&$0$&$0$\\
			\hline
			$(1,1,\beta)$; $\beta = 1+t, 1+t+t^2$&$2$&$0$&$6$&$0$&$0$&$0$&$0$\\
			\hline
			$(1)$&$2$&$0$&$0$&$0$&$8$&$0$&$0$\\
			\hline
			$(2)$&$2$&$0$&$0$&$16$&$4$&$0$&$0$\\
			\hline
			$(3)$&$2$&$0$&$0$&$0$&$16$&$0$&$0$\\
			\hline
			$(4)$&$4$&$0$&$0$&$0$&$0$&$2$&$0$\\
			\hline
			$(5)$&$8$&$0$&$0$&$0$&$4$&$0$&$0$\\
			\hline
				\multicolumn{2}{|c|}{Total \# } & $8$ & $22$ & $32$ & $88$ & $16$&$6$ \\
 			\hline
		\end{tabular}
	\caption{For primitive representations of $\SL_2(\mathbb F_2[t]/(t^{6}))$}
	\label{table2:SL2-O-6f}	
\end{table}

\subsection{Primitive representation zeta polynomial of  $\SL_2(\cO_{2 \ell}),$ $1\leq \ell \leq 3$  }
\label{subsec:primitive-zeta-polynomials}
Recall from Section~\ref{sec:group-algebras} that the primitive representation zeta polynomial of  $\SL_2(\cO_{2\ell})$ is  defined as below.  
\[
\calp^\pr_{\SL_2(\cO_{2 \ell})}(X) = \sum_{\rho \in \mathrm{Irr}^\pr(\SL_2(\cO_{2 \ell})) } X^{\dim(\rho)},
\]
where $\mathrm{Irr}^\pr(\SL_2(\cO_{2 \ell}))$
is the set of primitive irreducible representations of $\SL_2(\cO_{2 \ell}).$ 
Summarizing the computations from the above subsections we have the primitive  representation zeta polynomial of  $\SL_2(\cO_{2 \ell}),$ $\cO=\mathbb Z_2\,\, \mathrm{or}\,\, \mathbb F_2 \llbracket t \rrbracket $ and $1\leq \ell \leq 3,$ as follows.
\begin{enumerate}
	\item For $\cO=\mathbb Z_2\,\, \mathrm{or}\,\, \mathbb F_2 \llbracket t \rrbracket ,$     $\calp^\pr_{\SL_2(\cO_{2})}(X)=4X^3 +X^2 +2 X.$
	\item $\calp^\pr_{\SL_2(\mathbb Z/2^4\mathbb Z)}(X)=2X^{24} +2X^{12} +6 X^8+20X^6+16X^3.$
	\item $\calp^\pr_{\SL_2(\mathbb F_2[t]/(t^{4}))}(X)=X^{24} +8X^{12} +5 X^8+16X^6+4X^4.$
		\item $\calp^\pr_{\SL_2(\mathbb Z/2^6\mathbb Z)}(X)=8X^{96} +8X^{48} +24 X^{32}+80X^{24}+64X^{12}.$
		\item $\calp^\pr_{\SL_2(\mathbb F_2[t]/(t^{6}))}(X)=6X^{96} +16X^{48} +22 X^{32}+88X^{24}+8X^{16}+32X^{12}.$
\end{enumerate}

\appendix 

\section{}
\label{appendix}

In this part, we explain the relation between our notations of $\bol s$ and $\bol k$ with some of the notations of \cite{2017arXiv171009112H}. 
This relation was used while writing a few results of \cite{2017arXiv171009112H} in our notations.

Assume $\cO \in \dvrtwoplus.$
Let $A = \mat 0{a^{-1} \alpha}a{\beta} \in M_2(\cO_{\ell'})$ and  $\tilde{A} = \mat 0{\tilde{a}^{-1}\tilde \alpha}{\tilde{a}}{\tilde \beta} \in M_2(\cO_r)$ be a lift of $A.$ Note that $\mathrm{trace}(A)=\beta$ and $\det(A)=\alpha.$
We recall the definition of $w(A)$ and odd depth of $A$, denoted by $\delta(A)$, from \cite{2017arXiv171009112H}.
$$
w(A)=\begin{cases}
\val(\beta)&   \mathrm{if}\,\, \beta \neq 0,\\
\ell' &   \mathrm{if}\,\, \beta = 0.\end{cases}
 $$
And $\delta(A)$ is the smallest $ i \in \left [ 0, \lfloor\frac{w(A)}{2}\rfloor \right )$ such that $(\alpha)_{2i+1}\neq 0 $. If such $i$ does not exist then we define  $\delta(A)=\lfloor\frac{w(A)}{2}\rfloor.$
Recall that $\bol k$ and $\bol s$ are defined as follows:  
\begin{eqnarray}
\bol  k & = &  \val(\beta),  \nonumber \\
\bol s & = &\begin{cases}
2 \lfloor \bol k/2 \rfloor + 1 &  \mathrm{if}\,\,   \alpha = v^2 \,\,\mathrm{mod}\,(\pi^ \bol k ),  \\ m & \mathrm{if}\,\, \alpha = v_1^2 + \pi^m v_2^2 \,\,\mathrm{mod}\,(\pi^ \bol k ) \,\,\mathrm{for}\,\,\mathrm{odd}\,\, m < \bol k \,\,\mathrm{and}\,\, v_2 \in \cO_{\ell'}^\times. 
\end{cases}\nonumber
\end{eqnarray}
We also note that every element $x\in \cO_m$ for $\cO \in \dvrtwoplus$ is of the form $x = (x)_0 + (x)_1 \pi + \cdot + (x)_{m-1} \pi^{m-1}$, where $(x)_i \in \cO_1$. Further $x \in (\cO_m)^2$ if and only if $(x)_{2i+1} = 0$ for all $i$.  
\begin{proposition}
	Given the above definitions, $w(A)=\bol k$ and $\delta(A) =\lfloor \frac{\bol s}{2} \rfloor.$
\end{proposition}
\begin{proof}

For $\beta\neq 0, $ clearly $w(A)=\bol k.$
For $\beta = 0\in \cO_{\ell'},$ we have  $w(A)= \ell' = \bol k.$ Now we proceed to prove the result for $\delta(A)$. For $\delta(A)=\lfloor\frac{w(A)}{2}\rfloor,$ by definition of  $\delta(A),$ we have  $(\alpha)_{2i+1}= 0$ for all $0\leq i <\lfloor\frac{w(A)}{2}\rfloor.$ Therefore $\alpha \in \cO_{\ell'}^2  \,\,\mathrm{mod}\,(\pi^{2\lfloor\frac{w(A)}{2}\rfloor+1}).$ Since $2\lfloor\frac{w(A)}{2}\rfloor+1=2\lfloor\frac{\bol k}{2}\rfloor+1\geq \bol k,$ we obtain $\alpha \in \cO_{\ell'}^2  \,\,\mathrm{mod}\,(\pi^{\bol k}).$ Therefore by definition of $\bol s$ we have $\bol s=2\lfloor\frac{\bol k}{2}\rfloor+1$ in this case. Combining all this, we obtain that $\delta(A) = \lfloor\frac{w(A)}{2}\rfloor = \lfloor\frac{\bol k}{2}\rfloor = \lfloor\frac{\bol s}{2}\rfloor $ for $\delta(A)=\lfloor\frac{w(A)}{2}\rfloor.$
For $\delta(A)<\lfloor\frac{w(A)}{2}\rfloor,$ by definition of $\delta(A),$ we must have $(\alpha)_{2\delta(A)+1}\neq 0$ and $(\alpha)_{2i+1}= 0$ for all $0\leq i <\delta(A).$
	 Therefore  we obtain $\alpha=v_1^2 +\pi^{2\delta(A)+1}v_2^2 $ for some choice of $v_1\in \cO_{\ell'}$ and $v_2\in \cO_{\ell'}^{\times}.$ By our assumption on $\delta(A)$, we have $2\delta(A)+1 < 2\lfloor\frac{w(A)}{2}\rfloor + 1$. As both $2\delta(A)+1$ and $2\lfloor\frac{w(A)}{2}\rfloor + 1$ are odd, this implies $ 2\delta(A)+1 < 2\lfloor\frac{w(A)}{2}\rfloor \leq w(A) = \bol k$. Therefore by definition of $\bol s$ we must have $\bol s=2\delta(A)+1.$ Therefore $\delta(A)=\lfloor\frac{\bol s}{2} \rfloor$ also for  $\delta(A)<\lfloor\frac{w(A)}{2}\rfloor.$ 
\end{proof}

\bibliography{refs}{}
\bibliographystyle{siam}

\end{document}